\documentclass{amsart}

\usepackage{amsmath, amsthm, amsfonts, amssymb, graphicx}
\usepackage{color, url}
\usepackage[all]{xy}
\usepackage{comment}
\usepackage{tikz-cd} 
\usepackage{tikz} 
\usepackage{float}
\usepackage[normalem]{ulem}
\theoremstyle{definition}

\theoremstyle{plain}

\newcommand{\R}{\mathbb{R}}

\newcommand{\ZZ}{\mathbb{Z}}

\newcommand{\bs}{\setminus}
\newcommand{\ve}{\varepsilon}

\newtheorem{theorem}{Theorem}[section]
\newtheorem{lemma}[theorem]{Lemma}

\newtheorem{proposition}[theorem]{Proposition}

\newtheorem{corollary}[theorem]{Corollary}

\theoremstyle{definition}
\newtheorem{definition}[theorem]{Definition}

\newtheorem{example}[theorem]{Example}

\newtheorem{defi}[theorem]{Definition}

\newtheorem{thm}[theorem]{Theorem}
\newtheorem{lem}[theorem]{Lemma}

\newtheorem{rem}[theorem]{Remark}

\theoremstyle{remark}
\newtheorem{remark}[theorem]{Remark}

\numberwithin{equation}{section}

\usepackage[backend=biber,hyperref=true,isbn=false,maxnames=7,maxalphanames=4,minalphanames=4,style=alphabetic,doi=false,url=false]{biblatex}

\begin{document}

\author{Qiyue Chen}
\address{School of Mathematics \\ University of Minnesota \\ Minneapolis\\ MN 55455\\ USA}
\email{chen8448@umn.edu}

\author{Gregg Musiker}
\address{School of Mathematics \\ University of Minnesota \\ Minneapolis\\ MN 55455\\ USA}
\email{musiker@umn.edu}

\thanks{Q.C. is partially supported by a Graduate Endowment Merit Fellowship and G.M. is partially supported by a Simons MPS Travel grant.}
\title{Gale-Robinson Quivers and Principal Coefficients}

\date{\today}

\maketitle

\begin{abstract}
In this paper, we provide a combinatorial interpretation for Laurent polynomials obtained by iteratively mutating a certain periodic quiver that has been framed with frozen vertices. This yields a family of cluster variables with principal coefficients associated to a family of integer sequences known as Gale-Robinson sequences.  The work of this paper completes arguments for preliminary results announced in earlier work of Jeong-Musiker-Zhang, and relates to works of Bousquet-M\'elou-Propp-West, Speyer, Vichitkunakorn, and of Eager-Franco.      
\end{abstract}

\tableofcontents

\section{Introduction}
In this paper, we consider a variant of the Gale-Robinson sequence \cite{Gale1991strange}, i.e. $\{x_n\}$ satisfying $x_n x_{n-N} = x_{n-r}x_{n-N+r} + x_{n-s} x_{n-N+s}$, where we include a second alphabet of coefficients $\{y_1,y_2,\dots, y_n\}$ that breaks the symmetry of this recurrence.  This deformation is motivated by the the theory of \emph{cluster algebras with principal coefficients}.

The results of this paper were first announced in an extended conference abstract  \cite{MusikerJeongZhang2013Jan} by the second author (and two REU students In-Jee Jeong and Sicong Zhang) for \emph{Formal Power Series and Algebraic Combinatorics} in 2013.  The present work completes the argument that was laid out therein, where the main proof was broken into three steps, as will be  discussed later.  

The undeformed version of the Gale-Robinson sequence had previoulsy been studied by several authors \cite{ProppMelouWest2009, Speyer2007May}.  For example, Bousquet-M{\'e}lou, Propp, and West \cite{ProppMelouWest2009} describe sequences of graphs, termed \emph{pinecones}, such that the $n$th term in the associated Gale-Robinson sequence enumerates perfect matchings in the $n$th pinecone graph.  Such pinecones can also be constructed by using Speyer's ``crosses-and-wrenches'' method \cite{Speyer2007May}, which provides graph theoretical formulas for Laurent expansions of expressions satisfying the \emph{Octahedron recurrence}.  In particular, if one chooses the appropriate plane of initial conditions, then one can build graphs that are known by experts to be isomorphic (modulo elementary transformations) to the pinecones (see also Remark \ref{Rem:Speyer}).  

We also note work such as that of Panupong Vichitkunakorn \cite{Vichitkunakorn2016Jun}
``Solutions to the T-Systems with Principal Coefficients'' 
in the intervening years, who gives solutions in terms of perfect matchings, non-intersecting paths and networks. 
The $A_\infty$ $T$-system considered in \cite{Vichitkunakorn2016Jun} is equivalent to the Octahedron recurrence, and while the Gale-Robinson sequences studied in this paper are nontrivial specializations of the Octahedron recurrence,  adapting these specializations to the setting with principal coefficients takes additional work. 
Furthermore, in contrast to Vichitkunakorn's work, the formula in this paper for the Gale-Robinson sequences utilizes the theory of brane tilings from the physics literature which motivates future study of more general families of examples.
More precisely, our solution states that:

\begin{theorem} 
\label{Thm:GR-WH-1}
Let $\widehat{\mathcal{A}}_{Q_N^{(r,s)}} \subset \mathbb{Q}[y_1,y_2, \dots, y_N][x_1^\pm,x_2^\pm, \dots, x_N^\pm]$ denote the cluster algebra with principal coefficients associated to the Gale-Robinson quiver of type $(r,s,N)$.  For $n \in \{N+1,N+2,\dots\}$, define the cluster variables $\widehat{x_n}$ by mutating the initial seed $(\widehat{Q}_N^{(r,s)},\{x_1,x_2,\dots, x_N,y_1,y_2,\dots, y_N\})$ periodically by the sequence $1,2,3,\dots, N, 1,2,\dots$.  
Let $G_n^{(r,s,N)}$ (which equals $P(n;r,N-r,s,N-s)$ up to vertical reflection) be the pinecone graph constructed in Section \ref{Sec:pinecones} or in \cite{ProppMelouWest2009}.
Then, for $n\geq N+1$, the Laurent expansion of $\widehat{x_n}$ is given by the following combinatorial formula
\begin{eqnarray*} \widehat{x_n} = cm(G_n^{(r,s,N)})\left(\sum_{M \mathrm{~is~a~perfect~matching~of~} G_{n}^{(r,s,N)}} x(M)y(M)\right)\end{eqnarray*}
where $x(M)$, $y(M)$ are the weights and heights of perfect matching $M$, respectively.  
\end{theorem}

The terms used in the statement of this theorem are developed through the course of our work.  The paper is organized as follows: In Section \ref{sec:prelim}, we introduce the definitions of quiver mutation, cluster algebra, periodic quiver and coefficients that provide the framework for generating the sequences of Laurent polynomials that we study. In Section \ref{Sec:GR}, we introduce Gale-Robinson sequences, their associated periodic quivers, and construct brane tilings and pinecone subgraphs whose perfect matchings yield our desired combinatorial interpretations.   After defining several auxiliary functions in Section \ref{sec:main} we state our main theorem and list the propositions needed to prove it. The propositions naturally separate our proof into three steps, with Section \ref{sec:Step1}  devoted to the proof of Proposition \ref{prop:GRprinc}, Propositions \ref{Prop:cent-strip} and \ref{Prop:borders} proved in Section \ref{sec:Step2}, and Section \ref{sec:Step3} deals with Proposition \ref{Prop:superpositions}. We conclude with the proof of Theorem \ref{Thm:GR-WH} in Section \ref{sec:conclusions} where further remarks and directions are also provided.

\section{Preliminaries: Periodic Quivers and Cluster Mutation}

\label{sec:prelim}

In this section, we review the necessary background material on cluster mutation and periodic quivers from \cite{Fomin2002Apr} and \cite{Fordy2010}.  
A \emph{quiver} $Q = (Q_0,Q_1)$ is a directed finite graph with vertex set $Q_0$ and edge set $Q_1$ (also known as the set of arrows).  If $a$ is the arrow directed from vertex $v$ to vertex $w$, we refer to $w$ as the head of $a$, denoted $h(a) = w$ and $v$ as the tail, i.e. $t(a)=v$.  In this paper, we will usually assume 
that quivers have no $1$-cycles nor $2$-cycles, and state when this restriction is relaxed.  Let $|Q_0|=N$ and arrange the vertices on a regular $N$-gon in clockwise order.
 
\begin{definition} [Quiver Mutation] \label{Def:QM} The mutation of $Q$ at vertex $k$, denoted by $\mu_k Q$ is constructed (from $Q$) as follows:
 
1) For every $2$-path $i \to k \to j$ in $Q$, add an arrow $i \to j$.
 
2) Reverse the direction of all arrows incident to vertex $k$.
 
3) Remove any $2$-cycles that were created by implementing the previous steps.
\end{definition}
 
To any quiver, we can associate a \emph{cluster algebra}\footnote{This initial description is only for the skew-symmetric case with trivial coefficients.  Later in this section, we discuss the case of principal coefficients.} defined as follows.  First, we associate a variable, which we denote as $x_i$, to each vertex of $Q$.  Therefore, we have a subset $\{x_1,x_2,\dots, x_N\}$ associated to $Q$.  This is known as an \emph{initial cluster}.  We then define a cluster mutation that proceeds alongside the aforementioned quiver mutation. 

\begin{definition} [Cluster Mutation] Given a quiver $Q$ and a cluster 
 
\noindent $X=\{X_1,X_2,\dots, X_N\}$, we define the mutation of the \emph{cluster seed} $(Q,X)$ in the direction $k$  as $\mu_k(Q,X) = (\mu_k Q, X')$ where $X' = X \setminus \{X_k\} \cup \{X_k'\}$ and $X_k'$ is defined as 
$$\frac{ \prod_{i \to k \mathrm{~in~}Q} X_i + \prod_{k \to j \mathrm{~in~}Q} X_j}{X_k}$$
where multiplicities of arrows $i \to k$ or $k \to j$ are possible and yield variables raised to exponents.
\end{definition}
 
Given our initial cluster and quiver $Q$, we have an initial seed which can be mutated in $N = |Q_0|$ directions.  These newly constructed seeds can then be again mutated in $N$ directions, noting that $\mu_k^2 = id$.  There will possibly be cycles in this mutation graph, but we generically get an infinite tree where each vertex has degree $N$.
 
\begin{definition} [Cluster variables and algebras] The set of \emph{cluster variables} is the union of all clusters obtained via all finite sequences of mutations.  
The \emph{cluster algebra} $\mathcal{A}_Q$ associated with the intial seed $(Q,\{x_1,\dots, x_N\})$ is the subalgebra of $\mathbb{Q}(x_1,\dots, x_N)$, the field of rational functions in $N$ variables, generated by the set of cluster variables.
\end{definition}
 
See \cite{Fomin2002Apr,gekhtman2010cluster} for more details about cluster algebras in general.  We now introduce Fordy and Marsh's notion of periodic quivers \cite{Fordy2010}.  Let $\rho$ denote the permutation $(1,N,N-1,N-2,\dots, 3,2)$, which corresponds to clockwise rotation of the quiver $Q$.
 
\begin{definition} [Periodic Quiver] We say that a quiver $Q$ is \emph{periodic}, of period $m$, if $Q^{(m)} = \mu_{m}\circ \dots \circ \mu_2 \circ \mu_1(Q)$ equals $\rho^m(Q)$.
In other words, the quiver obtained by mutating by $1, 2, \dots, m$ in sequence is equal to the quiver obtained by cyclically permutating the vertex labels of $Q$.  
\end{definition}
 
In particular, a quiver $Q$ is of \emph{period} $1$ if and only if mutating at vertex $1$ and then applying $\rho^{-1}$ (sending $2 \to 1, 3 \to 2, \dots, N \to N-1, 1 \to N$) yields back the original quiver $Q$.  The importance of period $1$ quivers is that as long as we mutate at $1,2,3,\dots$ in sequence and periodically, the quivers obtained by mutation are equivalent to one another, up to cyclic permutation.

When $Q$ is a periodic quiver, we may start with an intial cluster $\{x_1,x_2,\dots, x_N\}$, and define $x_n$, for all $n \geq 1$ by mutating periodically at $1,2,3,\dots$.  
For example, we denote the new clusters $\mu_{1}( \{x_1,x_2,\dots, x_N\}, Q)$ and $\mu_2 \circ \mu_1 (\{x_1,x_2,\dots, x_N\}, Q)$ as 
$\{x_{N+1}, x_2,\dots, x_N\}$ and $\{x_{N+1}, x_{N+2},\dots, x_N\}$, respectively.  
 
More generally, for $n =mq+r$, we define $x_n$ to be the $r$th element of the cluster obtained by $\mu_r \circ \mu_{r-1}\circ \dots \circ \mu_1 \circ (\mu_m \mu_{m-1}\circ \dots \mu_1)^q(\{x_1,x_2,\dots, x_N\},Q)$.  We therefore obtain a one-parameter infinite subsequence of cluster variables, which we index by the set of postive integers.
 
If $Q$ is of period $1$, then there is a single recurrence relation 
$$x_n x_{n-N} = \prod_{i \to 1 \mathrm{~in~}Q} x_{n-i} + \prod_{1 \to j \mathrm{~in~}Q} x_{n-j}$$
that is satisfied by all $n\geq N+1$.  For higher period, there will be a family of $m$ interlaced recurrence relations instead.    

A coefficient system for a cluster algebra can be constructed by enlarging the set of initial cluster variables by including so called \emph{frozen variables}.  These variables correspond to new vertices at which mutation is disallowed.  A system of \emph{principal coefficients} is a special case where the arrows incident to the new vertices are particularly simple.  By Theorem 3.7 of \cite{Fomin2007Jan}, it follows that any coefficient system of geometric type can be algebraically deduced from a system of principal coefficients.

\begin{definition} [Quiver with Principal Coefficients]
Given a quiver $Q$ with $N$ vertices, we let $\widehat{Q}$ denote the quiver on $2N$ vertices that 

\begin{itemize}
 \item Contains $Q$ as an induced subgraph on vertices $\{1,2,\dots, N\}$.

 \item  For each vertex $v \in \{N+1,N+2,\dots, 2N\}$, include a single arrow $v \to v-N$.  
\end{itemize}
 \end{definition}

\noindent We then let $\widehat{\mathcal{A}_Q}$ denote the cluster algebra $\mathcal{A}_{\widehat{Q}}$, which we refer to as 
the \emph{cluster algebra for $Q$ with principal coefficients}.
We obtain an infinite sequence of cluster variables by mutating the expanded quiver $\widehat{Q}$ at $1,2,\dots, N, 1,2,\dots$ periodically.  
We let $\{x_1,x_2,\dots, x_N, y_1,y_2,\dots, y_N\}$ denote the corresponding initial cluster, and let $\{\widehat{x_{N+1}},x_2,\dots, x_N, y_1,y_2,\dots, y_N\}$ and $\{\widehat{x_{N+1}},\widehat{x_{N+2}},\dots, x_N, y_1,y_2,\dots, y_N\}$ denote 
the next two clusters.  Continuing in this way, we let  $\{\widehat{x_n}:n \geq N+1\}$ denote the infinite sequence of non-initial cluster variables obtained by this periodic mutation sequence.  Since we never mutate at vertex $v$ for $v \in \{N+1,N+2,\dots, 2N\}$, it follows that all of the $\widehat{x_n}$'s are Laurent polynomials whose denominators are free of $y_i$'s.  
\begin{rem}[$\bf{c}$-vectors]\label{rem:cvectors}
     We can keep track of how $\widehat{Q}$ changes as we apply a mutation sequence by using $\bf{c}$-vectors of \cite{Fomin2007Jan}. For $1 \leq i \leq N$, the {\bf c}\emph{-vector}, $\mathbf{c}_i$, is the length $N$ vector whose $j$th entry signifies the number of arrows $y_j \to x_i$ in $Q$.  This entry is negative if the arrows go from $x_i \to y_j$ instead.  In particular, for the initial quiver $\widehat{Q}$, each $\mathbf{c}_i$ equals the $i$th  unit vector $\mathbf{e}_i$.  We let $\mathbf{c}_i^{(\ell)}$ denote the $i$th $\mathbf{c}$-vector after mutating $\ell$ times periodically along the sequence $1,2,\dots, N, 1,2, \dots$.    

One of the fundamental results in the theory of cluster algebras is the \emph{sign-coherence} of $\bf{c}$-vectors stating that the components of each $\bf{c}$-vector are either all non-positive or all non-negative, which is proven in \cite{Derksen2010Jul,GrossHackingKeelKontsevich18}.
\end{rem}

We will be using $\bf{c}$-vectors extensively in Section \ref{sec:Step1}.

\section{Gale-Robinson Sequences} \label{Sec:GR}

We now introduce our main object of study, \emph{Gale-Robinson Sequences} and generalize Theorem 9 of \cite{ProppMelouWest2009} by utilizing the theory of cluster algebras with principal coefficients with respect to a certain two-parameter family of Fordy-Marsh Period $1$ Quivers. 
The cluster variables resulting from iterated mutation of such quivers correspond to the Gale-Robinson sequence \cite{Gale1991strange}, which were studied, implicitly, in work by Bosquet-Melou, Propp, and West \cite{ProppMelouWest2009}.  The Somos 4 and Somos 5 sequences (due to M. Somos as described in \cite{Gale1991strange}) appear as special cases.  Any Gale-Robinson sequence can also be shown to be a specialization of the Octahedron Recurrence \cite{Speyer2007May}, also known as a T-system \cite{DiFrancesco2014}.  See Remark \ref{Rem:Speyer} for details.

\begin{definition} [Gale-Robinson Sequences] \label{Def:GRS}
For positive integers $r \leq s \leq N/2$, the \emph{Gale-Robinson seqeunce} of type $(r,s,N)$ is defined to be the sequence $\{x_n : n\geq 1\}$ satisfying the recurrence relation 
\begin{eqnarray} \label{eq:GaleRob} x_n x_{n-N} = x_{n-r} x_{n-N+r} + x_{n-s} x_{n-N+s}\end{eqnarray} for $n\geq N+1$.
\end{definition}

As explained in Example 8.7 of \cite{Fordy2010}\footnote{Technically, we use the opposite quiver here so that our later statements regarding minimal matchings are simpler.}, for each triple of positive integers $(r,s,N)$ with $r < s \leq N/2$, there is a unique period $1$ quiver whose mutations yield the sequence $x_n$ satisfying recurrence relation (\ref{eq:GaleRob}).  We let $Q_N^{(r,s)}$ denote this quiver, and we review its construction below for the reader's convenience.

\begin{definition} [The Gale-Robinson Quiver] \label{def:GRQ} For $(r,s,N)$ as above, we let $Q_N^{(r,s)}$ denote the quiver constructed by the following four step process, starting with the edge-less quiver on $N$ vertices.  

\begin{enumerate}
 \item For all $1\leq i\leq N-r$, draw an arrow $i \to i+r$, and for all $1\leq j\leq r$, draw an arrow $j \to N-r+j$.

 \item For all $1 \leq i\leq N-s$, draw an arrow $s+i \to i$, and for all $1 \leq j\leq s$, draw an arrow $N-s+j \to j$.

 \item For all $1 \leq i \leq N-r-s$, draw an arrow from $r+i \to s+i$ and for all $1 \leq j\leq s-r$, draw an arrow $r+j \to N-s+j$.

 \item Erase any $2$-cycles created in $Q_N^{(r,s)}$.
\end{enumerate}
Note that there might be multiple arrows between vertices $i$ and $j$.  Such multiplicities definitely occur in the special case that $N$ is even and $s=N/2$, but for small $N$, there could be other coincidences.  See Figure \ref{Fig:237-GR} for an example.
\end{definition}

\begin{remark} \label{Rem:gcd}
If the parameters $r$, $s$, and $N$ share a common factor (say their $\gcd$ is $d$) then $Q_N^{(r,s)}$ is a disconnected graph with $d$ connected components.  
Each component is isomorphic to $Q_{N/d}^{(r/d,s/d)}$.  Thus without loss of generality through the rest of this paper, we assume that $\gcd(r,s,N)=1$.  
\end{remark}

\begin{figure}
\begin{center}
\includegraphics[height=1in]{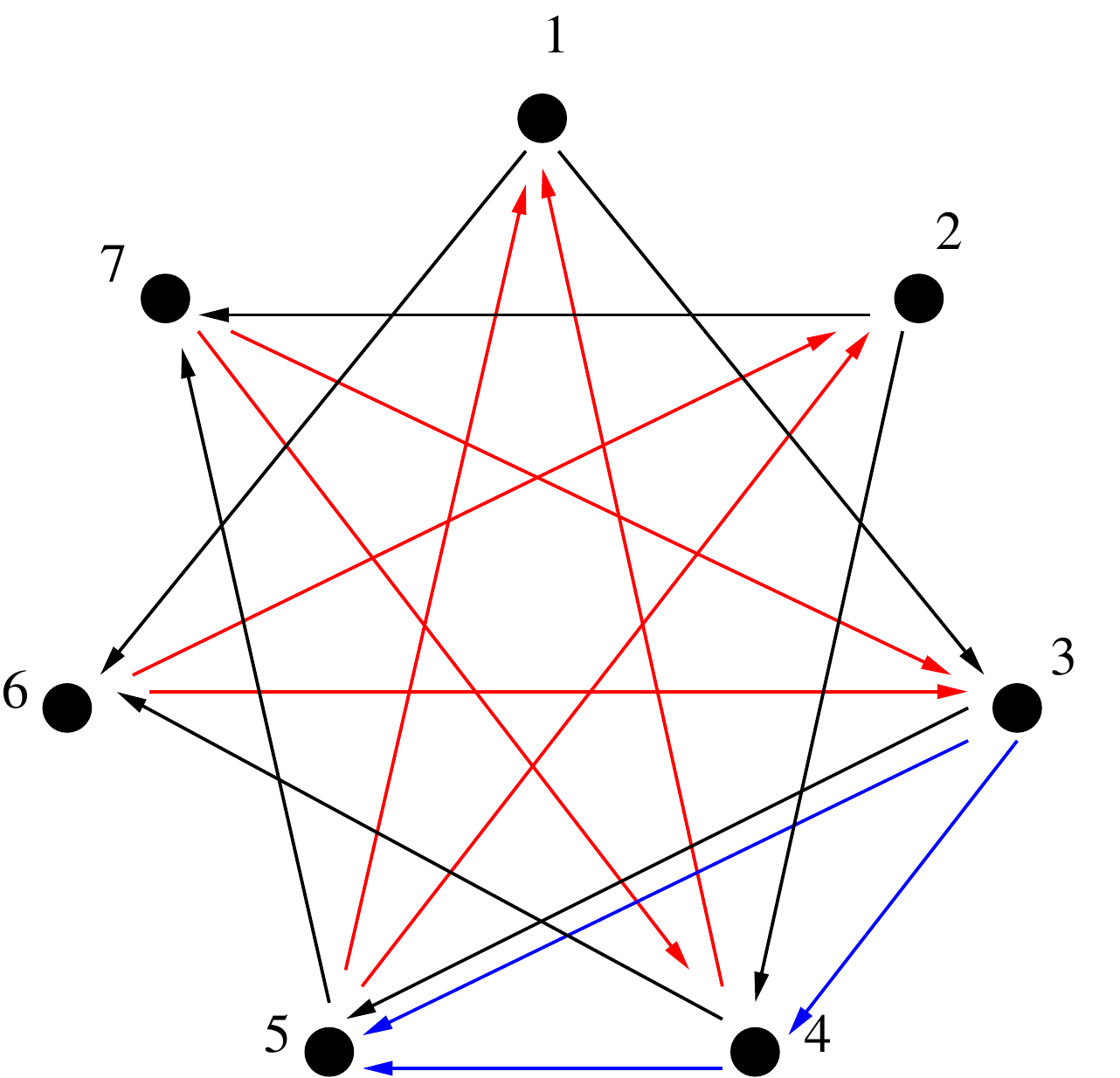} \hspace{0.0em}$=$\hspace{0.0em}  
\includegraphics[height=1in]{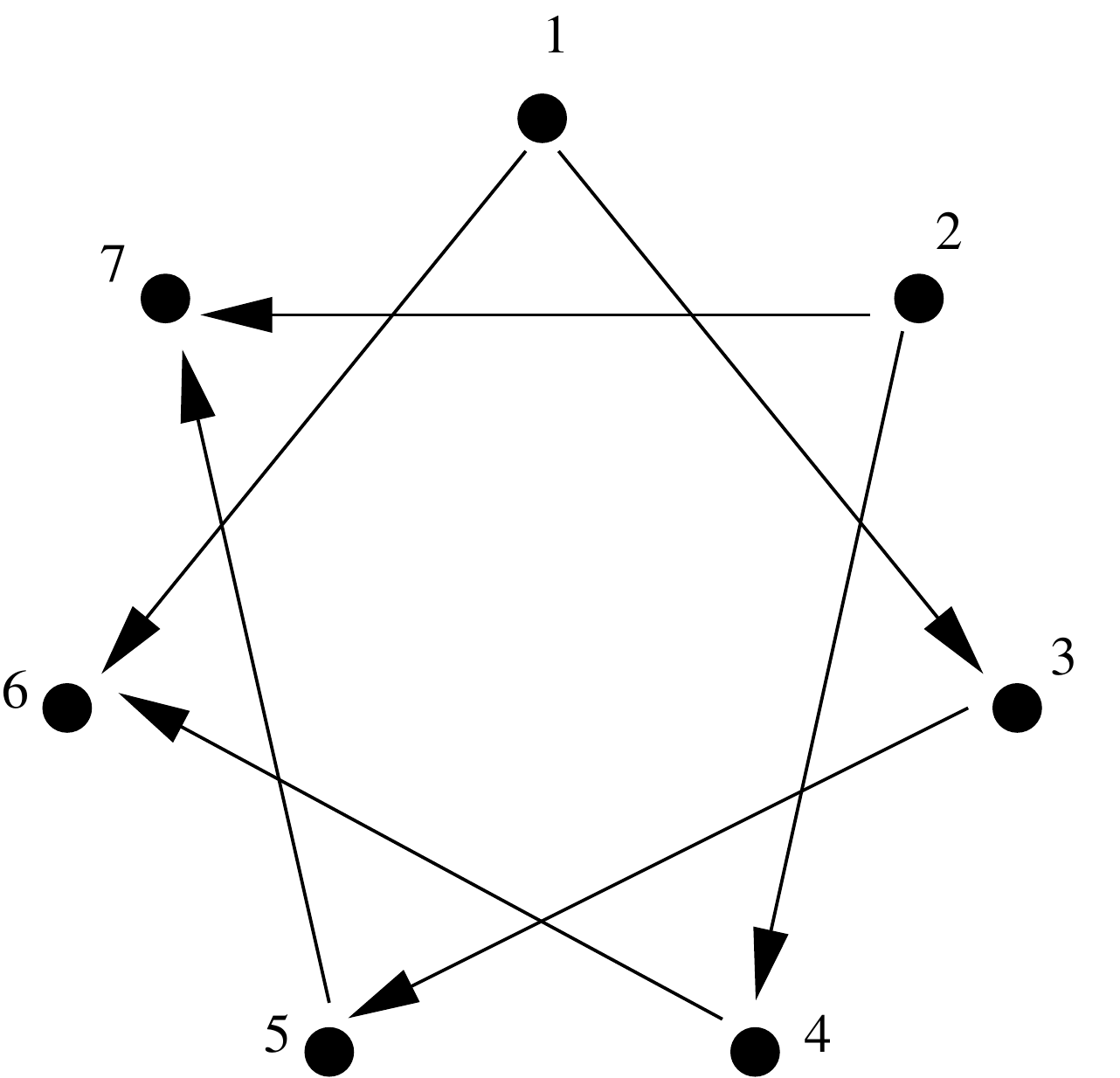} \hspace{0.0em}$+$\hspace{0.0em}
\includegraphics[height=1in]{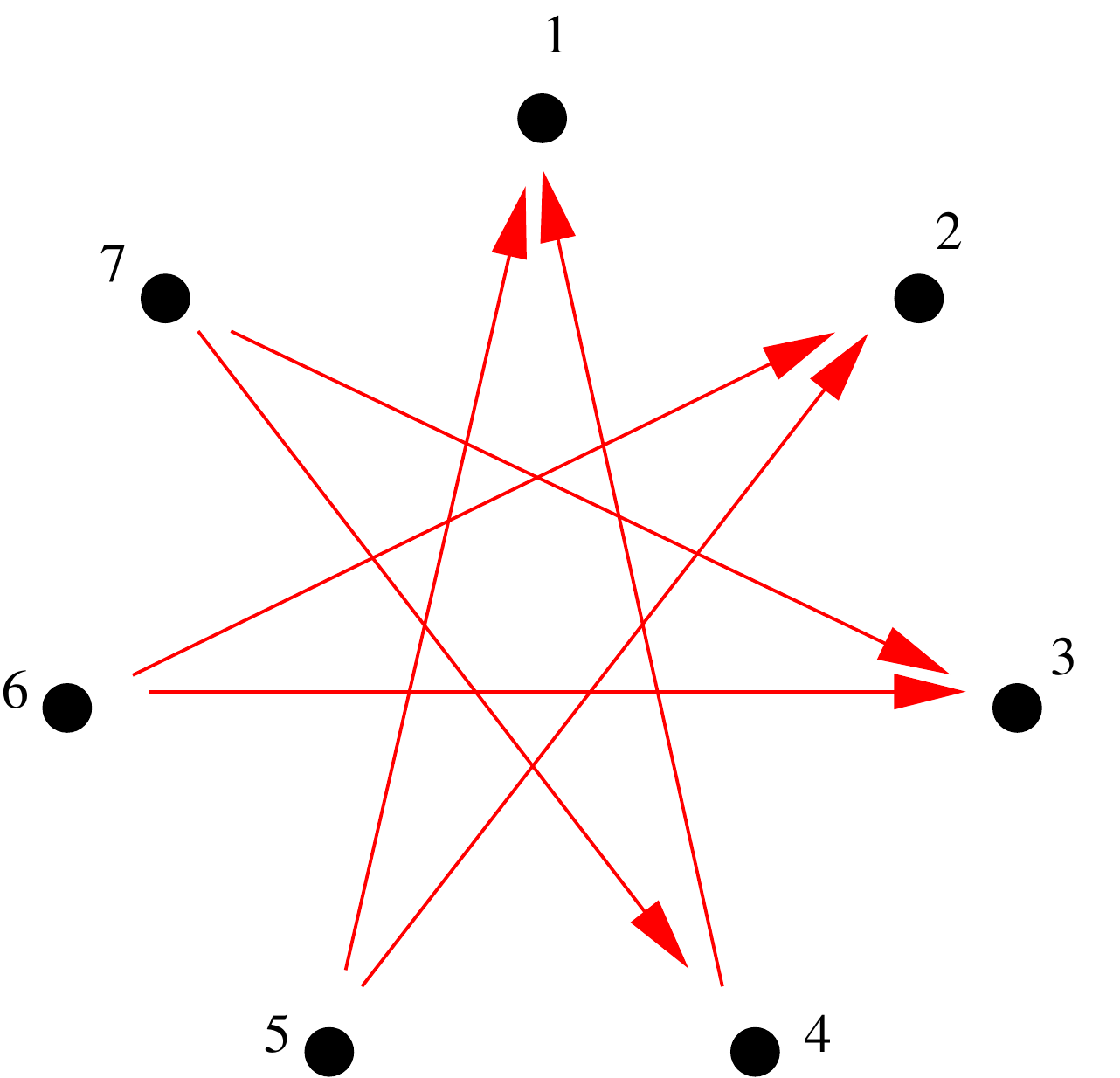} \hspace{0.0em}$+$\hspace{0.0em}
\includegraphics[height=1in]{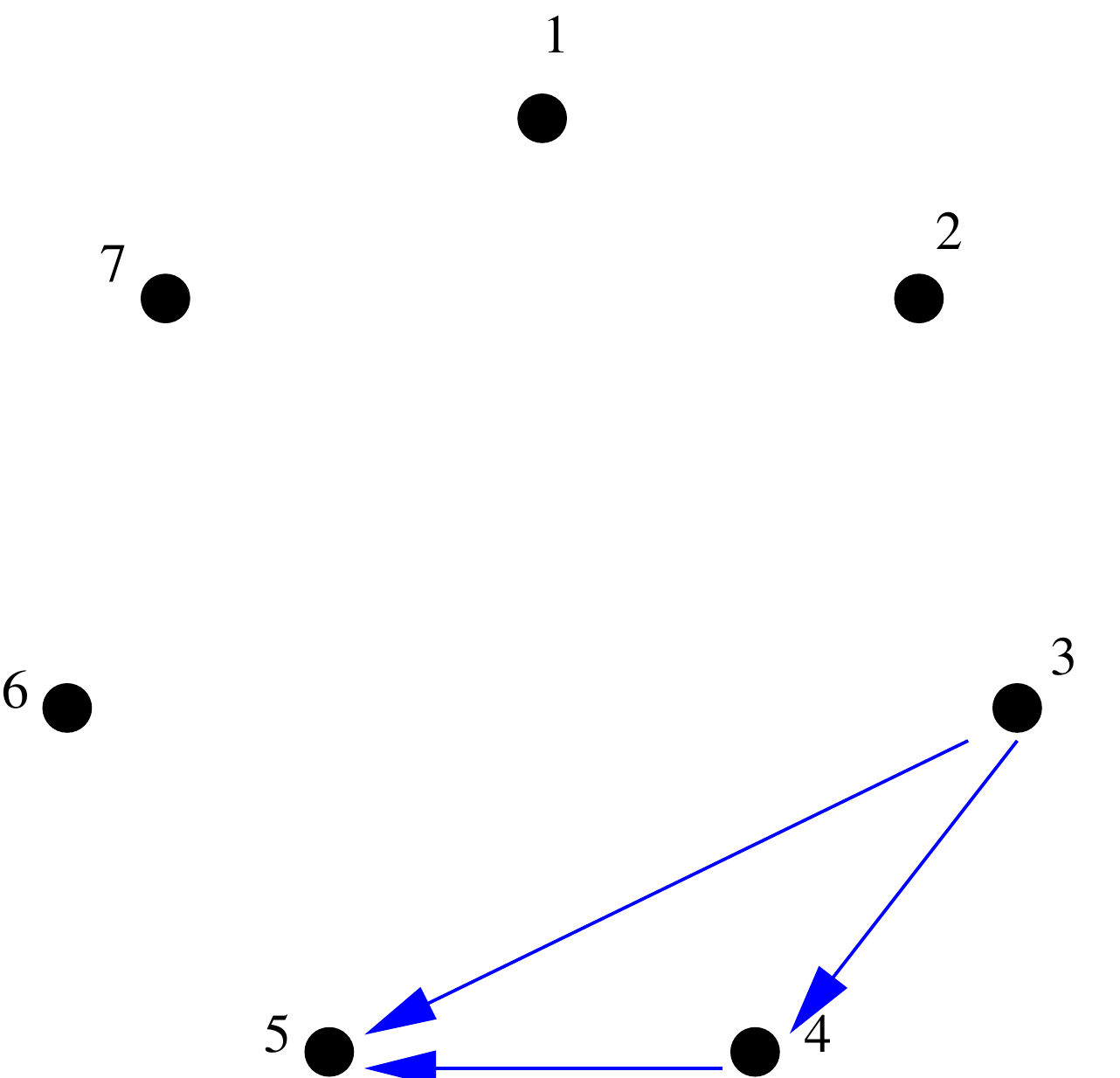} 
\end{center}

\caption{The Gale-Robinson Quiver $Q^{(2,3)}_7$ as a sum of the primitive period $1$ quivers
$-P_7^{(2)}$, $+P_7^{(3)}$, and $-P_{\{4,5,6\}}^{(1)}$.}
\label{Fig:237-GR}
\end{figure}
 
\begin{remark} Using the language of \cite{Fordy2010}, each of these three steps corresponds to adjoining a primitive period $1$ quiver.  In particular, let
$P_N^{(k)}$ denote the unique primitive period $1$ quiver on $N$ vertices with an arrow $k \to 1$ and 
$P_{\{i_1,i_2,\dots, i_M\}}^{(k)}$ denotes the primitive quiver $P_M^{(k)}$ where the vertices are relabeled using $i_1$ through $i_M$.
Then 
\begin{eqnarray*} Q_N^{(r,s)} &=& P_N^{(s)} - P_N^{(r)} - P_{\{r+1,r+2,\dots, N-r\}}^{(s-r)} \mathrm{~~if~~} s < N/2, \mathrm{~and}\\ 
Q_N^{(r,s)} &=& 2P_N^{(N/2)} - P_N^{(r)} - 2P_{\{r+1,r+2,\dots, N-r\}}^{(N/2-r)} \mathrm{~~if~~} s = N/2.
\end{eqnarray*}
\end{remark}

In \cite{ProppMelouWest2009}, the authors provide a combinatorial interpretation for the sequence $\{x_n: n\geq 1\}$ with the initial conditions $x_1=x_2=\dots = x_N=1$.  In particular, each $x_n$ is an integer, which is a non-trivial fact since the recurrence relation (\ref{eq:GaleRob}) involves division.  This was proven directly in \cite{Gale1991strange} but also follows from Fomin and Zelevinsky's Laurent Phenomenon \cite{Fomin2002Apr}, which states that every cluster variable is a \emph{Laurent polynomial} in terms of the initial cluster.

More specifically, in \cite{ProppMelouWest2009}, they introduce a family of graphs, known as \emph{pinecones}.  For each quadruple of positive integers $(n, r,s, N)$ such that $r < s \leq N/2$ and $n > N$, they define $P(n;r,N-r,s,N-s)$ so that the specialized cluster variable $x_n(x_1=x_2=\dots=x_N=1)$ counts the number of perfect matchings in $P(n;r,N-r,s,N-s)$.  In the next section, we provide an alternate construction of pinecones that is motivated by recent literature on supersymmetric quiver gauge theories.  For example, see \cite{Eager2012, FENG2001165, Franco2006Jan, Franco2006Jan2, Hanany2012Jun, Hanany2012Aug}.

\begin{remark} \label{Rem:Speyer} While it has not been written down explicitly in print, the pinecone graphs constructed in \cite{ProppMelouWest2009} are equivalent to the subgraphs obtained in \cite{Speyer2007May} by David Speyer using his method of ``crosses and wrenches''.  
In fact, Speyer goes further, and gives a method to constructing families of graphs that can be associated to any sequence of cluster variables $\{f(a,b,c)\}$ that come from a specialization of the \emph{Octahedron Recurrence}
$$f(n,i,j) f(n-2,i,j) = f(n-1,i-1,j)f(n-1,i+1,j) - f(n-1,i,j-1)f(n-1,i,j+1).$$
In terms of the notation of this paper, Speyer defines a family of $f(n,i,j)$'s satisfying the Octahedron Recurrence by defining
$$f(n,i,j) = x_m \mathrm{~~for~~} m=\frac{nN+(2r-N)i + (2s-N)j}{2}$$ for each $n$, $i$, $j$.  Hence, as explained in \cite[Section 1.3]{Speyer2007May}), one can use an almost plane of initial conditions, $(n,i,j)$
satisfying  
$-N < \frac{Nn + (2r-N)i+(2s-N)j}{2} \leq 0$, to build
``crosses-and-wrenches'' graphs equivalent to the family of pinecones. By solving the above inequalities for a solution so that $n$ and $i+j$ have the same parity modulo $2$,  we obtain that $n$ is a funciton of $i$ and $j$ defined by $n(i,j)=i+j-2\left\lceil\frac{ri+sj}{N}\right\rceil.$
\end{remark}

\subsection{From Gale-Robinson Quivers to Brane Tilings} \label{Sec:Brane}

We now describe how to use techniques from Supersymmetric Quiver Gauge Theories to obtain the pinecones more directly.  By letting $r = a$ and $s=c$, the Gale-Robinson sequence $\{x_n\}$ defined above agrees with the $\{Z_n\}$'s appearing in \cite[Section 9.1]{Eager2012}.  In the quiver gauge theory and brane tiling literature, $Z_n$ denotes a \emph{Pyramid Partition Function} (cluster variable) associated to a certain \emph{cascade of Seiberg dualties} (mutation sequence).  

The example highlighted in Section 9.1 of \cite{Eager2012} is inspired by a $L^{a,b,c}$-geometry which comes from a toric Calabi-Yau $3$-manifold.  See \cite{Franco2006Jan} for more on the construction of the $L^{a,b,c}$-geometry and how to obtain a corresponding brane tiling.  Further details also appear in \cite{Eager2011}, which describes connections to \cite{Speyer2007May}, as in Remark \ref{Rem:Speyer}, in this language. 

In particular, when $\gcd(r,s,N)=1$, the Gale-Robinson quiver $Q_N^{(r,s)}$ is the quiver corresponding to the geometry of $L^{r,N-r,s}$ (i.e. let $a=r, b = N-r, c = s$.  There is a fourth parameter often included when discussing an $L^{a,b,c}$ geometry, namely $d = a+b-c$.  In this dictionary, $d =N-s$ and indeed the Gale-Robinson recurrence
(\ref{eq:GaleRob}) looks very symmetric as $x_n x_{a+b} = x_{n-a} x_{n-b} + x_{n-c} x_{n-d}$.
As a special case, the $Y^{p,q}$ geometry corresponds to $L^{p-q,p+q,p}$ and hence $Q_{2p}^{p-q,p}$ where we assume $1 \leq q \leq p-1$.  

Most simply stated, a \emph{brane tiling} is a tiling of the torus, which we visualize as a doubly-periodic tiling of its universal cover, the infinite plane.  For the reader's convenience, we now summarize a step-by-step process to go from a Gale-Robinson quiver to an associated brane tiling.  
Towards this end, we must now allow quivers with $2$-cycles.  In particular, let 
$\overline{Q_N^{(r,s)}}$ denote the quiver obtained by following steps (1)-(3) of Definition \ref{def:GRQ}.  By abuse of notation, we will also refer to $\overline{Q_N^{(r,s)}}$ as a Gale-Robinson quiver, since $2$-cycles do not affect the associated recurrence. 

\begin{enumerate}
\item 
Firstly, since $\overline{Q_N^{(r,s)}}$ is highly symmetric, we can embed it onto the surface of a torus.  Then using the fact that the torus has $\mathbb{R}^2$ as its 
universal cover, we can unfold $\overline{Q_N^{(r,s)}}$ onto the plane as a doubly-periodic infinite quiver.  (Note that in contrast to the previous sections, the word ``periodic'' has a geometric rather than an algebraic or combinatorial meaning here.) 
Since $\overline{Q_N^{(r,s)}}$ has such a nice structure, the unfolded infinite quiver, which we denote as $\widetilde{Q}_N^{(r,s)}$, is quite simple to describe:

a) Start with the $\mathbb{Z}^2$ lattice as an undirected graph, connecting $(a,b)$ with $(a \pm 1, b)$ and $(a, b \pm 1)$.
{We will refer to this as the $\ZZ^2$ square grid, and the connected components of its complement can be naturally viewed as the dual graph of the $\ZZ^2$ square grid graph.}

b) Label the vertex at the origin $(0,0)$ as $1$.  For all integer points $(A,B)$, we label the corresponding vertex as $(1+ Ar + Bs) (\mathrm{mod}~ N) \in \{1,2,\dots, N\}$.  

c) We now turn this lattice into a directed graph.  For all horizontal edges, we orient $i \to j$ if and only if $i < j$.  
For all vertical edges, we do the opposite (orient $i \to j$ if and only if $i > j$). 

d) Lastly, we add diagonal arrows as needed so that all triangles or squares in this planar directed graph are cyclically oriented.  
Proposition \ref{Prop:sqorient} ensures that this process is well-defined.

\item Secondly, we take the planar dual of $\widetilde{Q}_N^{(r,s)}$, and label its faces using the labels of vertices of $\widetilde{Q}_N^{(r,s)}$.  We use a black/white coloring with the convention that \emph{white is on the right}.
\end{enumerate}

The resulting doubly-periodic tiling of the plane is the \emph{brane tiling} $\mathcal{T}_N^{(r,s)}$.  See Figure \ref{Fig:237-BT} for an example.

\begin{proposition} \label{Prop:sqorient}
Consider a square $S$ with vertices corresponding to $i, i+s, i+r+s, i+r \in \{1,2,3,\dots, N\}$, taken modulo $N$ and in clockwise order starting from the lower-left.  Orient the four edges of the square using the convention of (1c).  Then, as in Figure \ref{Fig:FourSquares}, either the edges of $S$ form an oriented $4$-cycle, or can be split into two cyclically oriented triangles by adding a single oriented diagonal.
\end{proposition}

\begin{figure}
\includegraphics[height=1in]{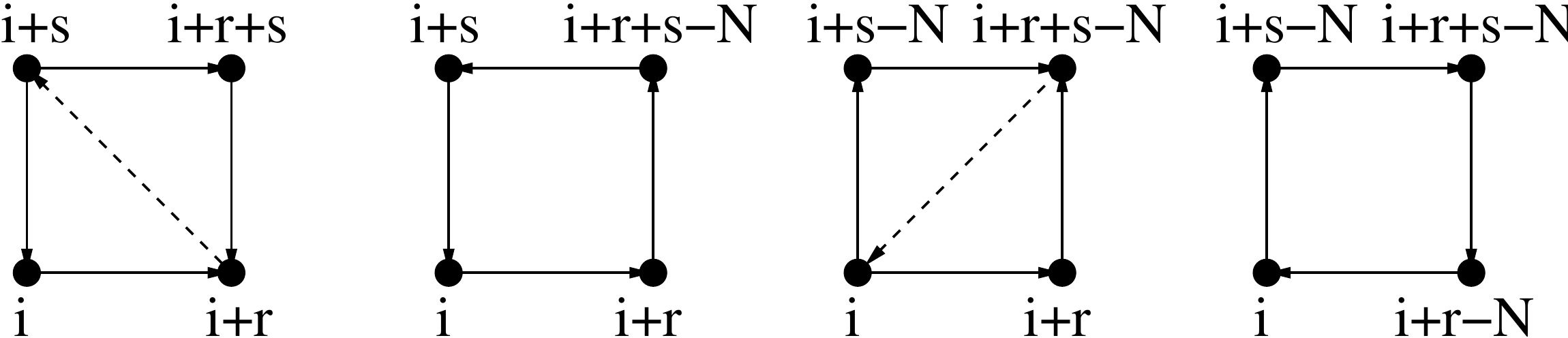}
\caption{The four possible local configurations in $\widetilde{Q}_N^{(r,s)}$.}
\label{Fig:FourSquares}
\end{figure}

\begin{figure}

\scalebox{0.69}{
\xymatrix{
& & & & & & & & & \\
& \ar@{<-}[l]\ar@{->}[u]{5} \ar@{->}[r] \ar@{<--}[rd] \ar@{<--}[ur] & {7} \ar@{->}[u] \ar@{<-}[r] & {2} \ar@{->}[r] \ar@{<-}[u] & {4} \ar@{<-}[u] \ar@{->}[r] \ar@{<--}[rd] \ar@{-->}[lu] & {6} \ar@{<-}[r] \ar@{->}[u] & {1} \ar@{<-}[u] \ar@{->}[r] & 
{3} \ar@{<-}[u] \ar@{->}[r] \ar@{-->}[ld] \ar@{-->}[lu] & {5} \ar@{->}[u] \ar@{->}[r] \ar@{<--}[rd] \ar@{<--}[ur]& \\
&\ar@{<-}[u] \ar@{<-}[d] \ar@{->}[l] {2} \ar@{->}[r] & {4} \ar@{->}[d] \ar@{<-}[u] \ar@{->}[r] \ar@{<--}[rd] & {6} \ar@{<-}[r] \ar@{->}[u]\ar@{->}[d] & {1} \ar@{<-}[d]\ar@{<-}[u] \ar@{->}[r] & 
{3} \ar@{<-}[d]\ar@{<-}[u] \ar@{->}[r] \ar@{-->}[ld]& {5} \ar@{->}[d]\ar@{->}[u] \ar@{->}[r] \ar@{<--}[rd] &
{7} \ar@{->}[u] \ar@{<-}[r] \ar@{->}[d] & {2} \ar@{<-}[u] \ar@{<-}[d] \ar@{->}[r] & \\
& & & & & & & & &
}}

\includegraphics[height=1.45in]{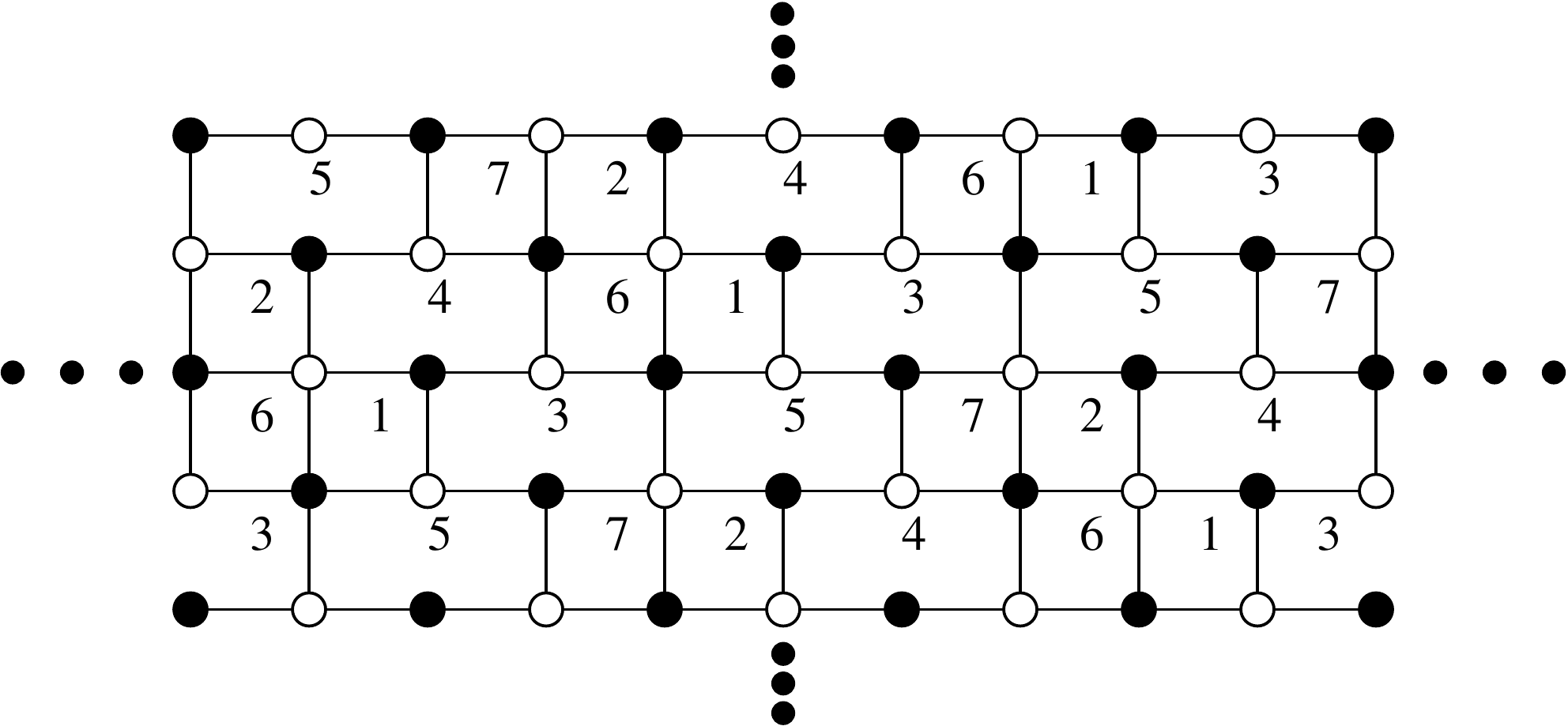}

\caption{The unfolded quiver ${\widetilde{Q}^{(2,3)}_7}$ and brane tiling $\mathcal{T}_7^{(2,3)}$.}
\label{Fig:237-BT}
\end{figure}

\begin{proof}
Since $r < s \leq N/2$, we note that $i+r < i+s < i+r+s$ all lie between $1$ and $2N-1$.  Consequently, we have four cases: 

1) $i < i+r < i+s < i+r+s \leq N$,

2) $i < i+r < i+s \leq N < i+r + s < 2N$,

3) $i < i+r \leq N < i+s < i+r + s < 2N$, or

4) $i \leq N < i+r < i+s < i+r + s < 2N$.

\noindent Since we label the vertices with their value modulo $N$, the ordering of the vertex labels in thse four cases are as follows:

1) $i < i+r < i+s < i+r+s$,

2) $i + r + s - N < i < i+r < i+s$,

3) $i + s - N < i+ r + s - N < i < i+r $, or

4) $i+r - N < i+s - N < i+r + s - N < i$.

\noindent Using the edge orientation rules listed above, we obtain a counter-clockwise (resp. clockwise) $4$-cycle in case 2 (resp. 4).  Further, adding a diagonal edge $i+r \to i+s$ properly orients case 1, and adding a diagonal edge $i+r+s-N \to i$ properly orients case 3.
\end{proof}

\begin{remark} The above four configurations correspond to local configurations appearing in the statistical mechanical \emph{square-ice} or \emph{six-vertex} models \cite{lieb1967exact}.  
\end{remark}

We now prove the correctness of the above construction.

\begin{proposition} Construct $\widetilde{Q}_N^{(r,s)}$ as above and then identify vertices with the same labels.  The resulting folded-up quiver exactly agrees with the Gale-Robinson quiver $\overline{Q_N^{(r,s)}}$ (possibly with $2$-cycles). 
\end{proposition}

\begin{proof}
When defining oriented edges in step 1 of $\overline{Q_N^{(r,s)}}$'s construction, we see that two vertices $i$ and $j$ are adjacent if and only if $|i-j| = r$, and the edge is oriented so that the smaller label points to the larger one.  This definition exactly agrees with the construction of the horizontal edges defined in step (1c) of the construction of $\widetilde{Q}_N^{(r,s)}$. 

Similarly, in step 2, two vertices $i$ and $j$ are adjacent if and only if $|i-j| = s$, but this time the larger label points towards the smaller one.  This coincides with the construction of the vertical edges in step (1c) of the construction of $\widetilde{Q}_N^{(r,s)}$.

It thus suffices to show that the oriented edges in $\overline{Q_N^{(r,s)}}$ constructed in step 3 agree with the oriented diagonal edges defined in step (1d) of $\widetilde{Q}_N^{(r,s)}$.  As in Cases 1 and 3 of Proposition \ref{Prop:sqorient}, the oriented edges we adjoin are $i+r \to i+s$ if $i \geq 1$ satisfies $i+r+s \leq N$, and $j+r+s-N \to j$ when $j$ satisfies $j+r \leq N$ but $j+s > N$.  However, these are just simple rephrasing of the rules in step 3 of $\overline{Q_N^{(r,s)}}$'s construction.    
\end{proof}

\begin{remark} If we worked with the $2$-cycle-less $Q_N^{(r,s)}$ instead of unfolding $\overline{Q_N^{(r,s)}}$, we would be missing some of the diagonal edges which are 
relevant for obtaining a regular pattern of hexagons.  See Example \ref{Ex:Somos5} for details.
\end{remark}

\begin{corollary} \label{Cor:sqhex}
For $1\leq i \leq r$, and $N-r \leq i \leq N$, the faces labeled with an $i$ are squares.  All other faces of the brane tiling are hexagons.
\end{corollary}

\begin{proof}
By considering the planar dual, this is equivalent to considering the degrees of the various vertices in the doubly periodic quiver.  Since this quiver looks locally like its folded version, the degree of vertex $i$ equals its degree in $\overline{Q_N^{(r,s)}}$.  However, by thinking of this quiver as the sum of three primitive period 1 quivers (except that the third has a restricted vertex set) exactly gives 
us vertices of degrees $4$ and $6$ as desired.
\end{proof}

Note: When drawing brane tilings or their subgraphs, we will typically depict hexagonal faces as horizontal rectangles of height one and width two.

\subsection{From Brane Tilings to Pinecones} \label{Sec:pinecones}

We now describe how to obtain the pinecone graphs, $P(n;r,N-r,s,N-s)$, constructed in \cite{ProppMelouWest2009}, that were mentioned above in Section \ref{Sec:GR}.  Our construction will 
produce these graphs from brane tilings directly.  Given a Gale-Robinson sequence and quiver $\overline{Q_N^{(r,s)}}$, we described in the last section how to construct the associated brane tiling $\mathcal{T}_N^{(r,s)}$.  We now describe how to construct a family of finite subgraphs of $\mathcal{T}_N^{(r,s)}$, each of which we denote as $G_{n}^{(r,s,N)}$ for $n\geq N+1$.  

\begin{definition} [Gale-Robinson Brane Subgraphs] \label{Def:pinecones}

For $N+1 \leq n \leq N+r$, we define $G_n^{(r,s,N)}$ as the subgraph of $\mathcal{T}_N^{(r,s)}$ consisting of the square face labeled $n-N$.
If $n > N+r$, we instead build $G_n^{(r,s,N)}$ layer-by-layer.  For this construction, we need some notation.  For $n >N+r$, let $\overline{n} \in \{1,2,\dots, r\}$ denote the integer such that $n \equiv N+ \overline{n}~ (\mathrm{mod}~ r)$ and $\overline{\overline{n}} \in \{1,2,\dots, N-r\}$ denote the integer such that 
$n \equiv N + \overline{\overline{n}}~ (\mathrm{mod}~ (N-r))$.
Define the horizontal strip $H_n^{(r,N)}$ to be the induced subgraph of $\mathcal{T}_N^{(r,s)}$ obtained by taking the grid graph of unit height and width equal to 
$2\lfloor \frac{n-N-1}{r}\rfloor + 1$ starting with the square face labeled as $\overline{n}$ as the left-most face.  In particular, $H_n^{(r,N)}$ is defined to be empty if 
$n \leq N$ and  $H_n^{(r,N)}$ ends with the square or hexagonal face labeled 
as $\overline{\overline{n}}$ as the right-most face.

For $n > N+r$, we then construct a graph by using $H_n^{(r,N)}$ as a central horizontal strip, and then gluing to its top (resp. bottom) the strips 
$H_{n-(N-s)}^{(r,N)}$, $H_{n-2(N-s)}^{(r,N)}$, $\dots$ (resp.  $H_{n-s}^{(r,N)}$, $H_{n-2s}^{(r,N)}$, $\dots$) until the strips added 
above and below are empty.  We glue these together in the unique way so that successive strips, emanating out from the center, are contained in the interior of the more central strip.
This defines an induced subgraph of $\mathcal{T}_N^{(r,s)}$, that we denote as $G_n^{(r,s,N)}$.
\end{definition}

\begin{example} \label{Ex:237}
Consider the case $r=2$, $s=3$, and $N=7$.  The corresponding quiver $Q_7^{(2,3)}$ appears in Figure \ref{Fig:237-GR} and its brane tiling $\mathcal{T}_7^{(2,3)}$ appears in Figure \ref{Fig:237-BT}.  Then for $8 \leq n \leq 16$, the strips $H_n^{(2,7)}$ are given as:

\vspace{0.5em}

\includegraphics[height=0.3in]{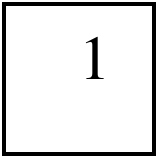}, \hspace{2em}
\includegraphics[height=0.3in]{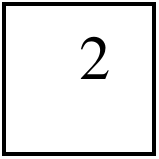}, \hspace{2em}
\includegraphics[height=0.3in]{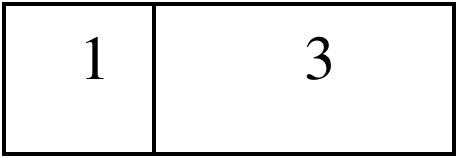}, \hspace{2em}
\includegraphics[height=0.3in]{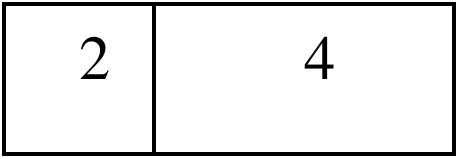}, \hspace{2em}
\includegraphics[height=0.3in]{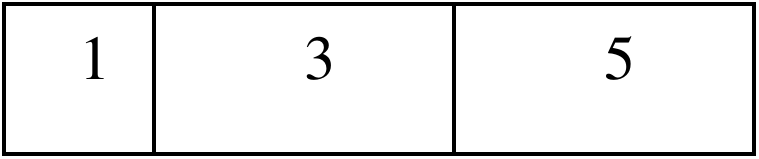}, 

\vspace{0.5em}

\includegraphics[height=0.3in]{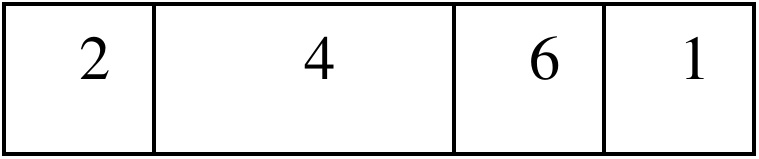}, \hspace{2em}
\includegraphics[height=0.3in]{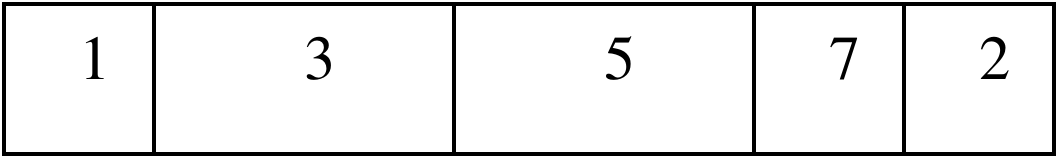}, 

\vspace{0.5em} 

\includegraphics[height=0.3in]{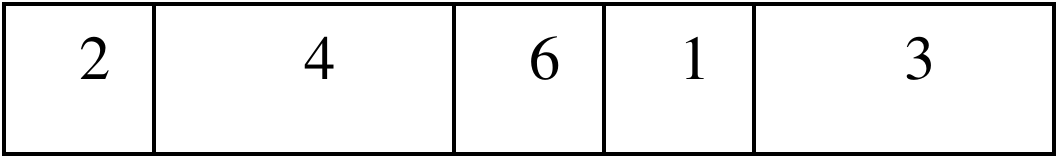}, \hspace{1em} and \hspace{1em} \includegraphics[height=0.3in]{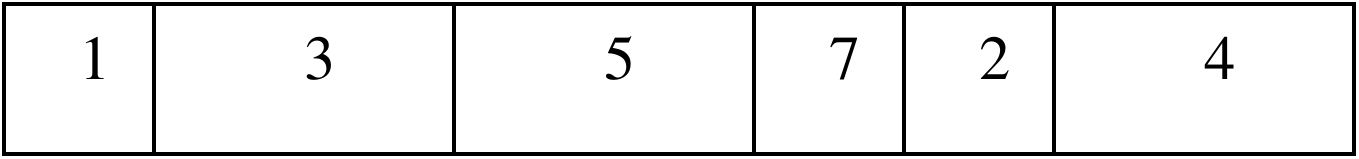}.

\vspace{1em}

\noindent Gluing these strips together as described above, we obtain the following Gale-Robinson Brane Subgraphs $\{G_n^{(2,3,7)}\}$ for $8 \leq n \leq 16$:

\vspace{0.5em}

\includegraphics[height=0.2in]{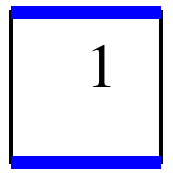}, \hspace{2em}
\includegraphics[height=0.2in]{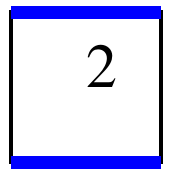}, \hspace{2em}
\includegraphics[height=0.2in]{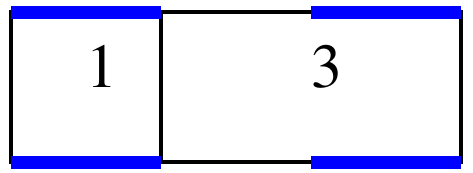}, \hspace{2em}
\includegraphics[height=0.4in]{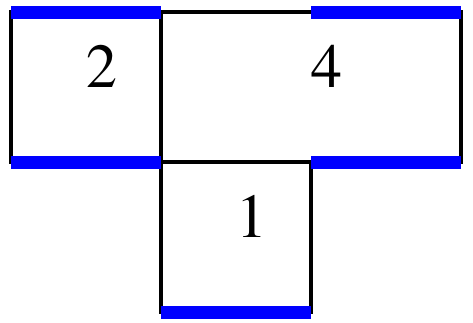}, \hspace{2em}
\includegraphics[height=0.6in]{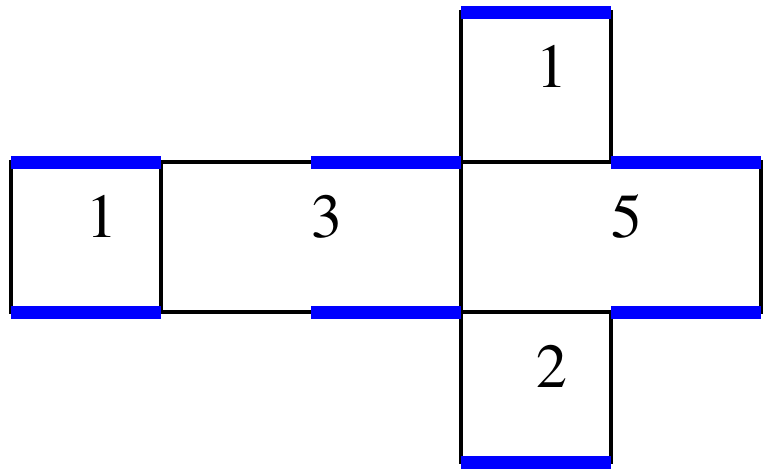}, 

\vspace{0.5em}

\includegraphics[height=0.6in]{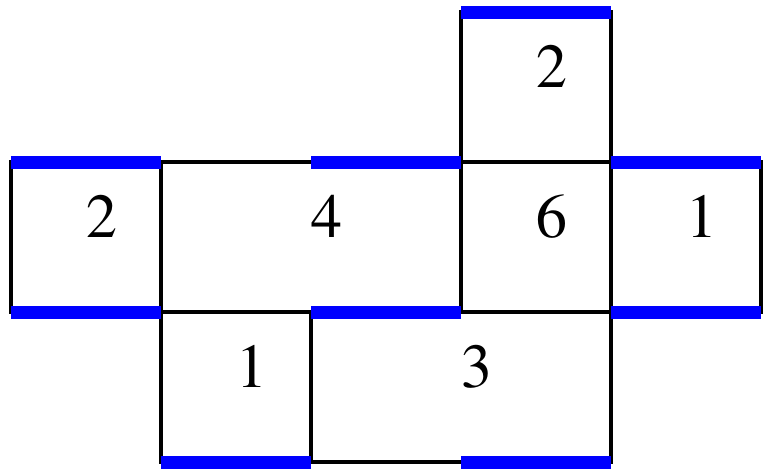}, \hspace{2em}
\includegraphics[height=0.8in]{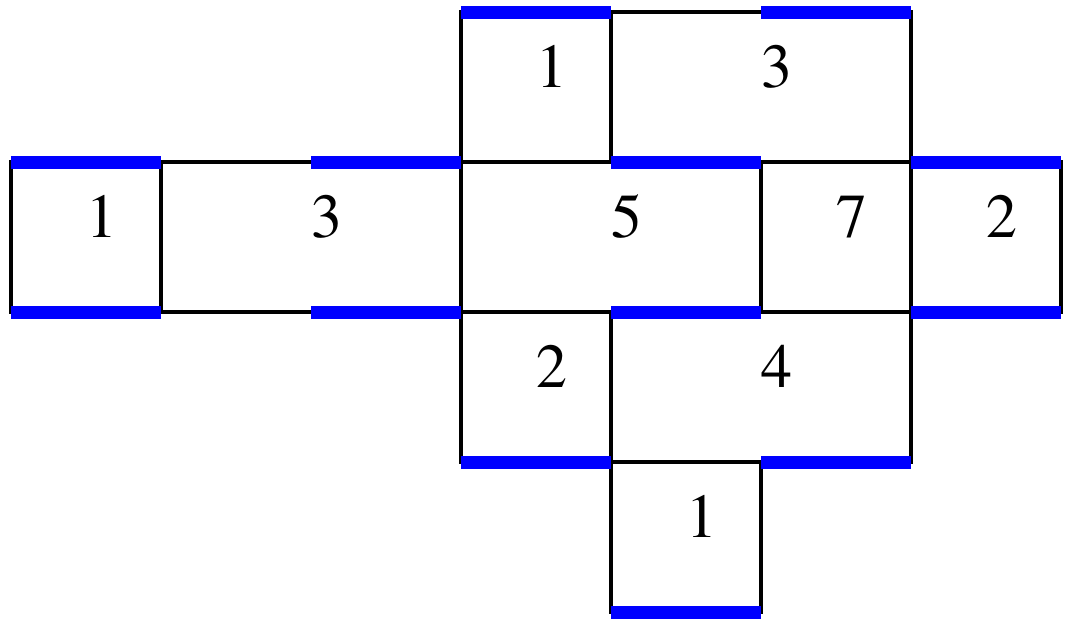}, 

\vspace{0.5em} 

\includegraphics[height=0.8in]{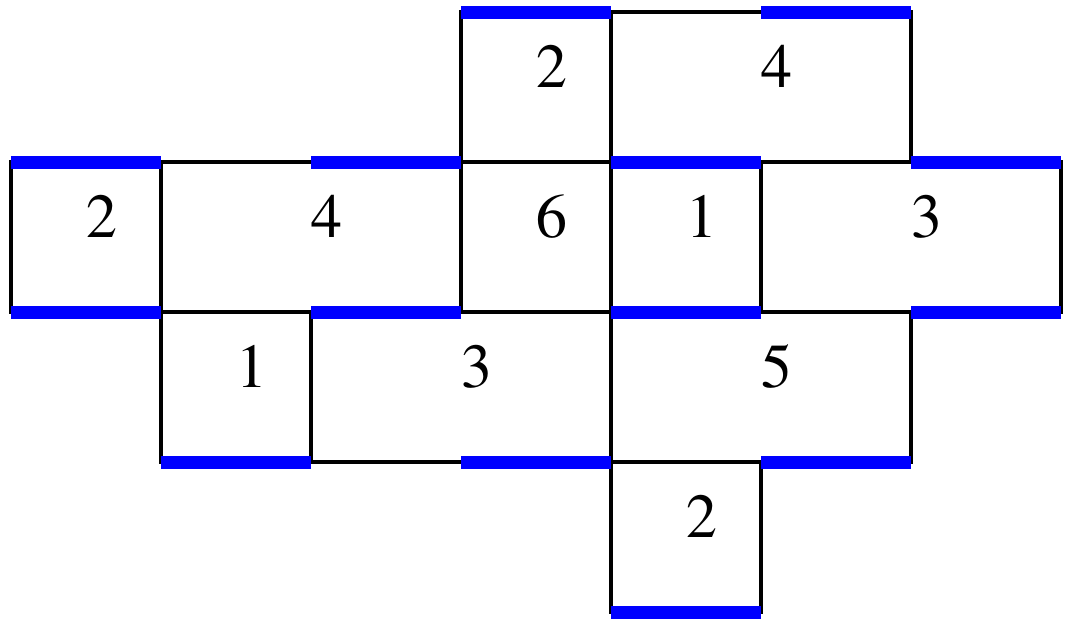}, \hspace{1em} and \hspace{1em} \includegraphics[height=1.0in]{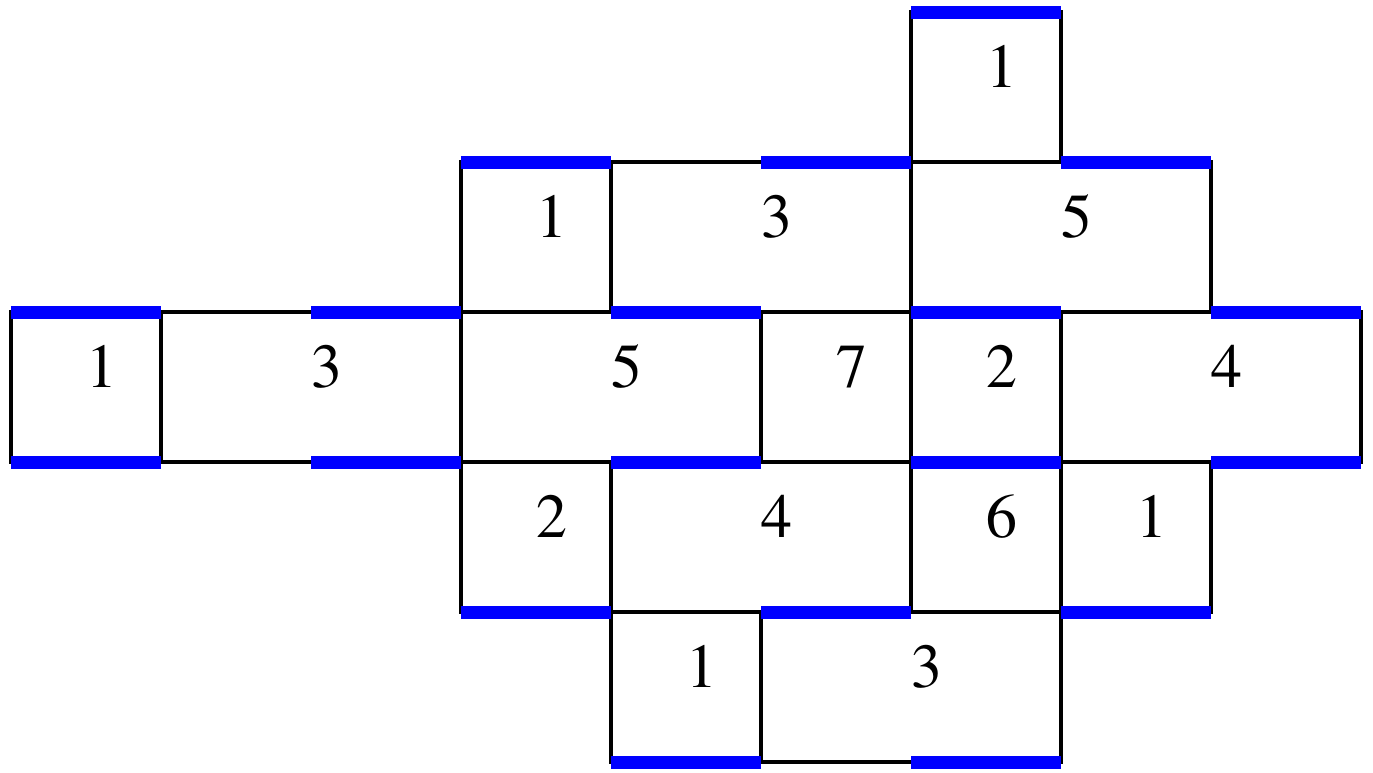}.

\noindent For example, the graph $G_{16}^{(2,3,7)}$ is obtained by gluing together the horizontal strips (from top to bottom) 
$H_{8}^{(2,7)}, H_{12}^{(2,7)}, H_{16}^{(2,7)}, H_{13}^{(2,7)}$, and $H_{10}^{(2,7)}$.  
(The highlighted edges are minimal matchings which are discussed further in Definition \ref{Def:Height}.) 
\end{example}

\begin{remark}\label{rem:Aztec-diamond}
The graphs $G_n^{(r,s,N)}$ can also be constructed by superimposing Aztec Diamonds of increasing sizes centered on top of a face (of the center row) of the brane tiling $\mathcal{T}_N^{(r,s)}$.  In particular, the first $r$ graphs are the squares labeled with $1\leq i \leq r$.  Subsequently, we have $r$ subsequences of Aztec Diamonds.  In particular, for $N' \geq 0$, the graph $G_{rN'+N+i}^{(r,s,N)}$ can be constructed by the following:

(i) locate a face of $\mathcal{T}_N^{(r,s)}$ labeled as $i$.  If it is a square, let $(a,b)$ denote this face, as viewed in the $\mathbb{Z}^2$ lattice.  If instead it is a hexagon, let $(a,b)$ denote the left-hand-side of this face.

(ii) Take the Aztec Diamond of size $(N'+1)$ (which has a central row of size $2N'+1$) and center it on top of 
the cell $(a+N', b)$.  

(iii) This superposition will usually result in a graph containing vertices of degree one.  By removing these, one-by-one, we obtain the desired subgraph $G_n^{(r,s,N)}$. See Figure \ref{Fig:237Core} for an example.
Note that this procedure is equivalent to taking the \emph{core} of a pinecone, as described in \cite[Section 2.4]{ProppMelouWest2009}. The procedure of taking core will be extensively used in Section \ref{sec:Step3}. 

The validity of this remark follows from Proposition \ref{Prop:borders}. In detail, the proposition (and its proof) shows that we can build \emph{recursively} the graphs $G_n^{(r,s,N)}$ from its four subgraphs $G_{n-s},G_{n-r},G_{n-N+s},G_{n-N+r}$. On the other hand, it is shown in \cite[Prop. 8 and the remark below it]{ProppMelouWest2009} that we can build $G_n$ (denoted $P(n)$ therein) from its four subgraphs in exactly the same way. The equality between $G_n$ and $P(n)$ follows once we notice that they are equal when $n\le N+r$.\footnote{We note that $N,r,s$ are denoted $m,j,k$ in \cite{ProppMelouWest2009}.}
\end{remark}

\begin{figure}

\includegraphics[height=0.42in]{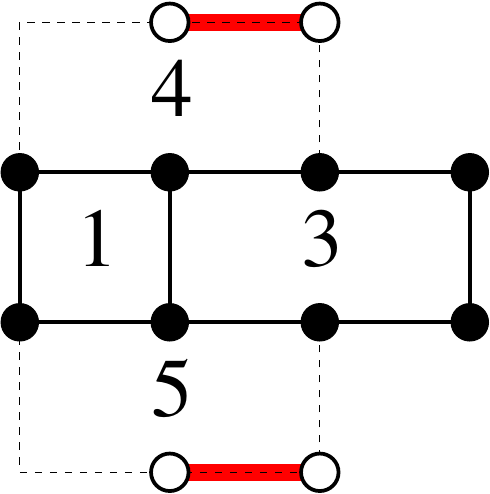} \hspace{1em}
\includegraphics[height=0.42in]{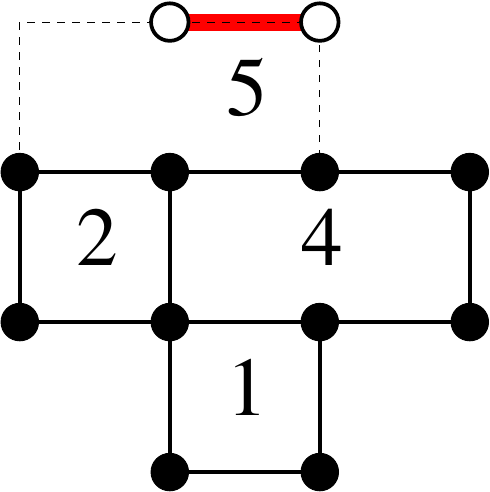} \hspace{1em}
\includegraphics[height=0.67in]{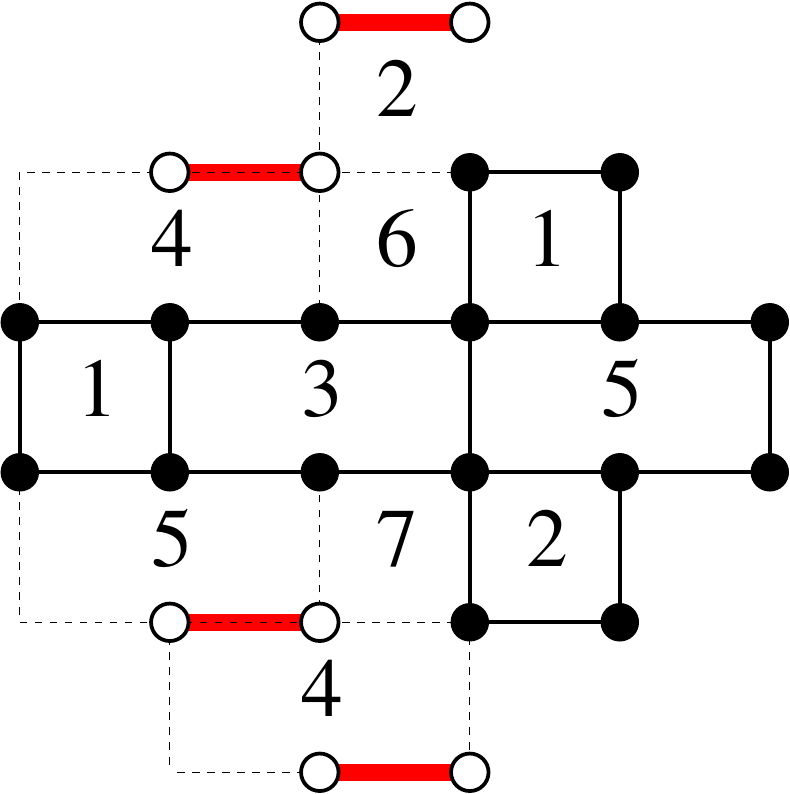} \hspace{1em}
\includegraphics[height=0.67in]{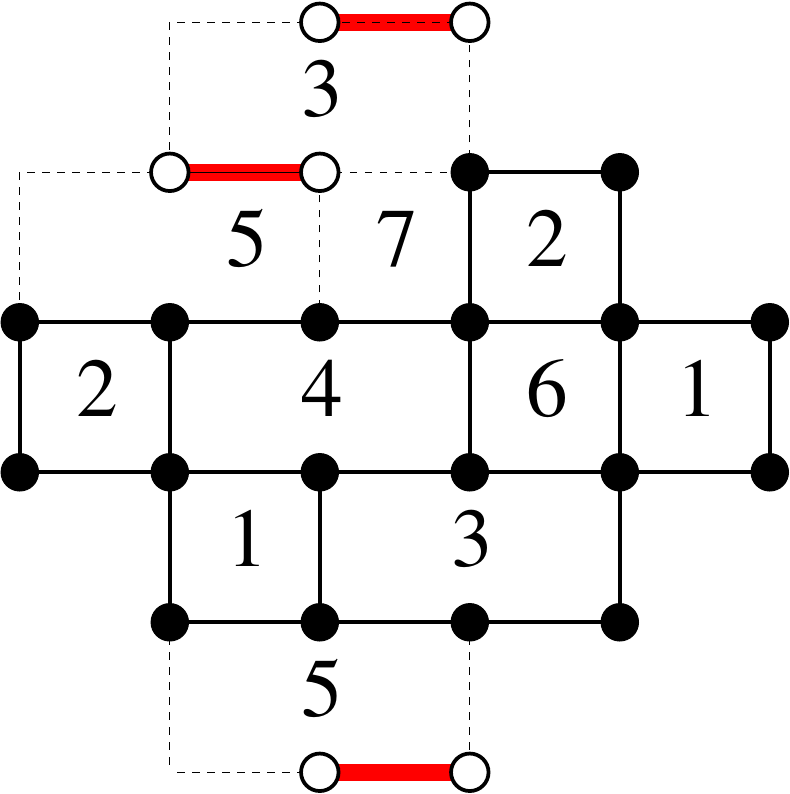} \hspace{1em}
\includegraphics[height=0.85in]{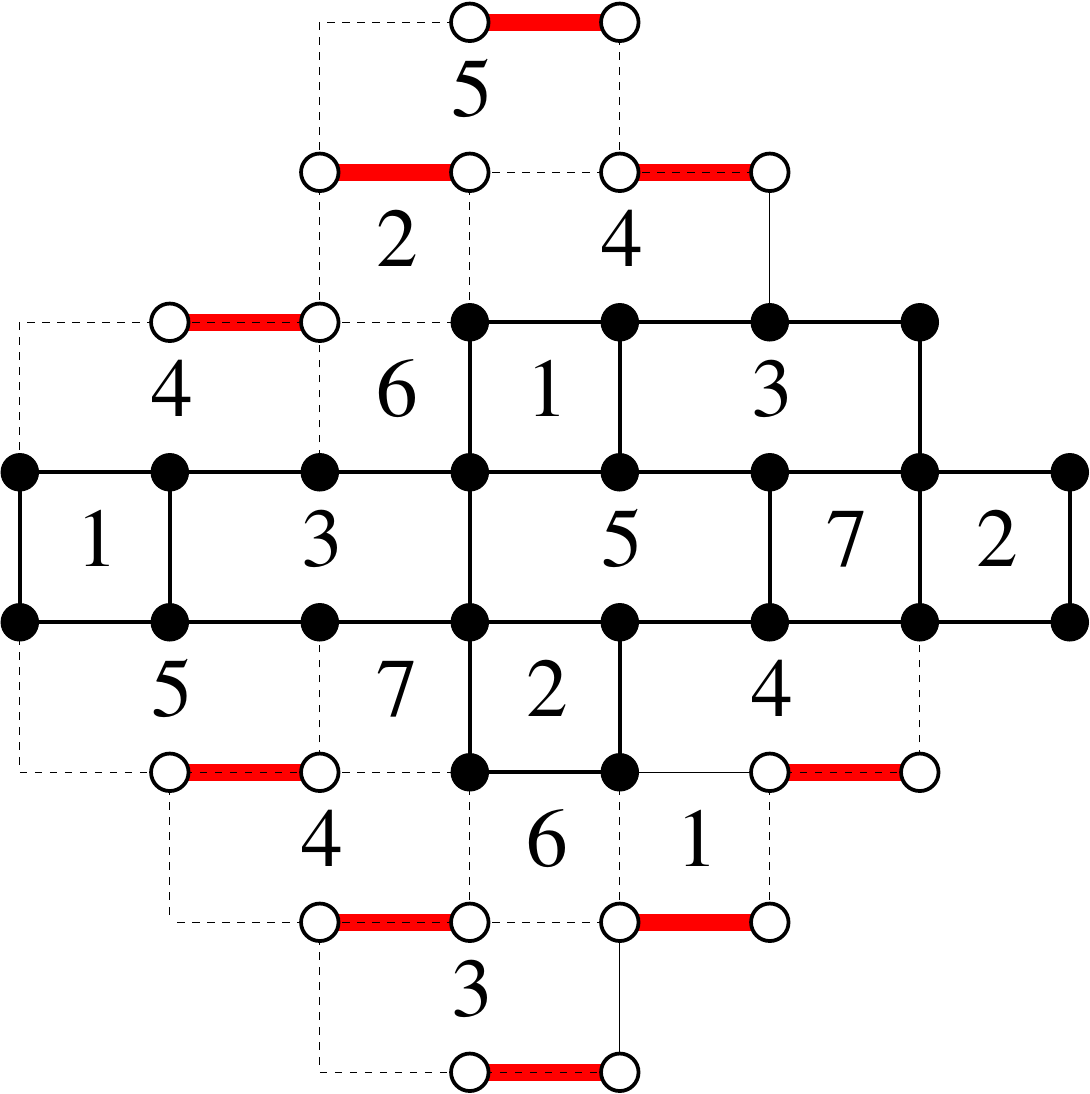} \hspace{1em}
\includegraphics[height=0.85in]{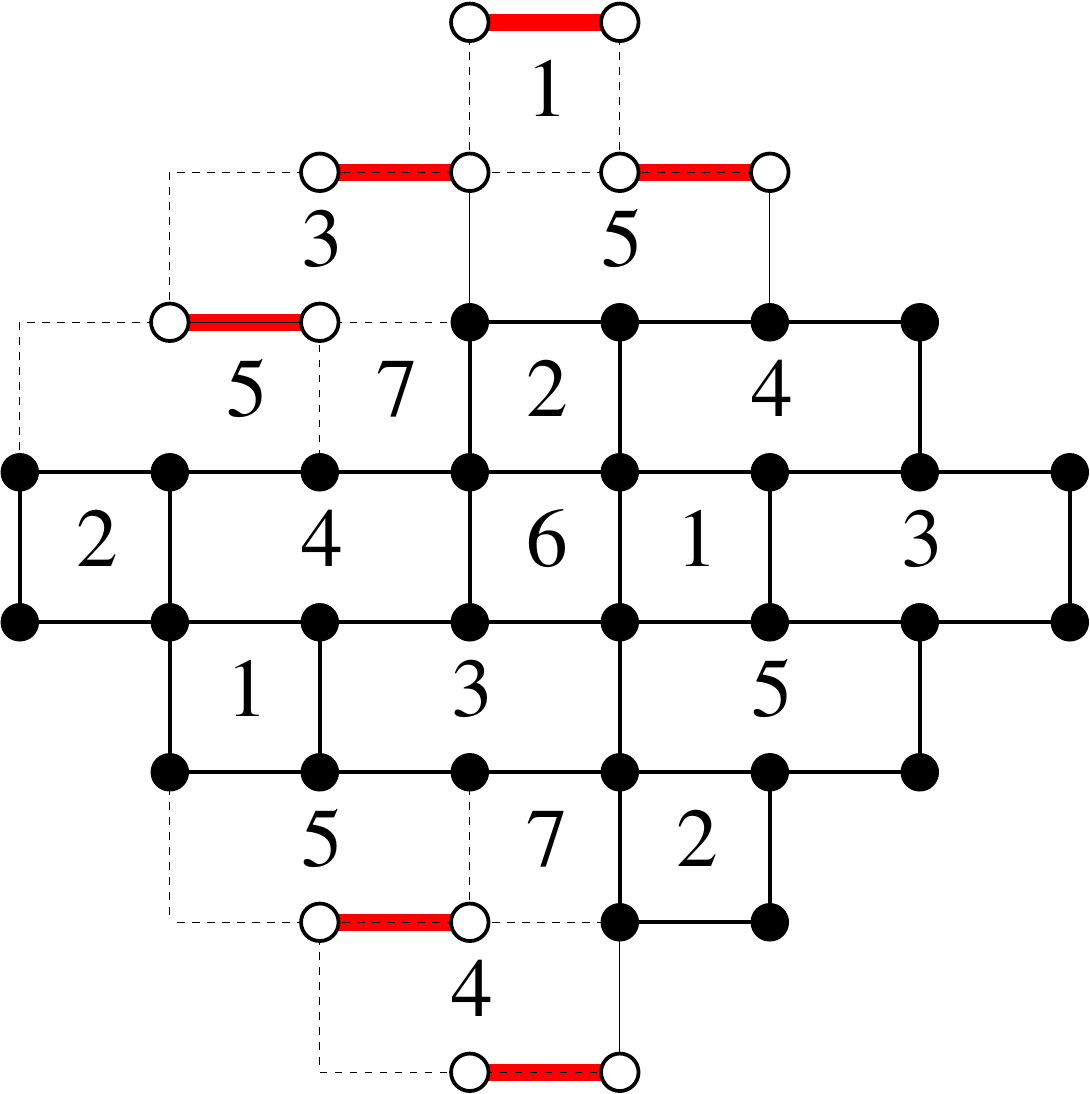}

\caption{Recovering the Gale-Robinson subgraph for $(r,s,N) = (2,3,7)$ from the associated Aztec Diamond for $10 \leq n \leq 15$.  Called the \emph{core} of a pinecone in \cite{ProppMelouWest2009}.  Compare with Example \ref{Ex:237} and $\mathcal{T}_7^{(2,3)}$ of Figure \ref{Fig:237-BT}.}
\label{Fig:237Core}
\end{figure}

\begin{remark}
A method for constrcting subgraphs of brane tilings also has appeared in the string theory literature.  For instance, see \cite[Sections 6, 7.3]{Eager2012} where a process is given for constructing the ``shadow of a pyramid''.  
\end{remark}
 
\begin{example} \label{Ex:Somos5}
Consider the two quivers in Figure \ref{Fig:dP2Opp}. 
The leftmost quiver is the Gale-Robinson Quiver 
$\overline{Q_5^{(1,2)}}$ with a $2$-cycle, and the rightmost quiver $Q_5^{(1,2)}$ has the $2$-cycle removed.
Mutating $Q_5^{(1,2)}$ periodically at $1,2,3,4,5,1,2,\dots$ yields the Somos-5 sequence as cluster variables.  However, to obtain a brane tiling interpretation corresponding to pinecones, we must unfold the quiver $\overline{Q_5^{(1,2)}}$ instead.  See 
Figure \ref{Fig:Somos5Brane} for the associated brane tiling $\mathcal{T}_5^{(1,2)}$.
This example is also discussed in detail in Section 9.4 of \cite{Eager2012}.  In contrast, see Figures \ref{Fig:Somos5AltQuiver} and \ref{Fig:Somos5AltBrane} for the brane tiling $\mathcal{{T}'}_5^{(1,2)}$associated to $Q_5^{(1,2)}$.
\end{example}
 
\begin{figure}
\includegraphics[height=1.3in]{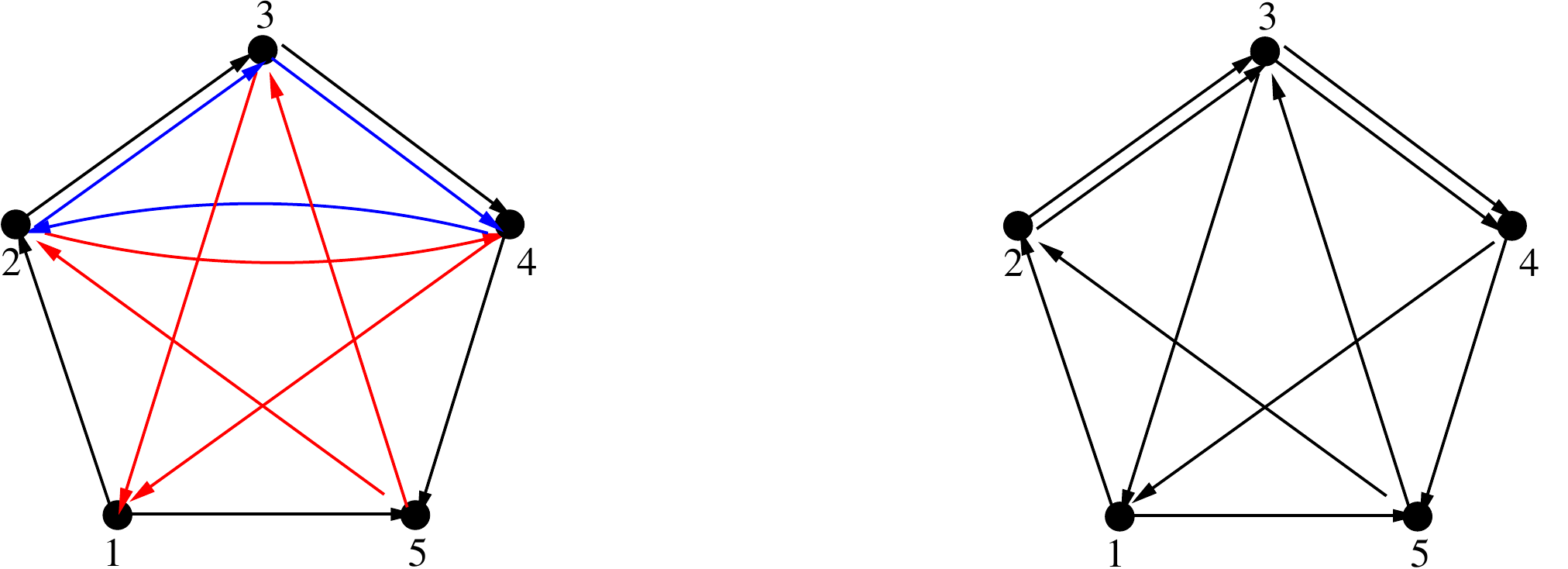}
\caption{(Left): Del-Pezzo $2$ quiver $\overline{Q_5^{(1,2)}}$; (Right): Psuedo-Del-Pezzo $2$ quiver $Q_5^{(1,2)}$.}
\label{Fig:dP2Opp}
\end{figure}

\begin{figure}
\includegraphics[height=1.5in]{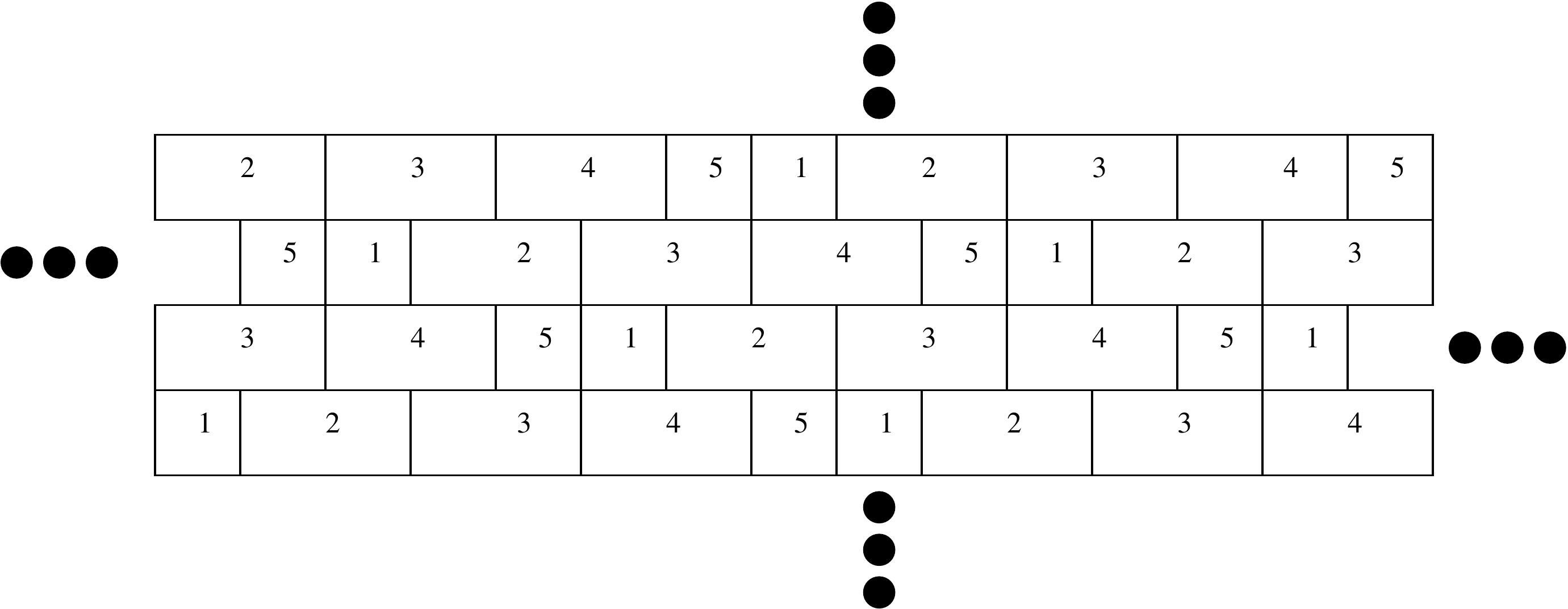}
\caption{The Brane Tiling $\mathcal{T}_5^{(1,2)}$.  Note that face $2$ borders face $4$ in two different ways, yielding the $2$-cycle in the associated quiver.}
\label{Fig:Somos5Brane}
\end{figure}
 
\begin{figure}
\includegraphics[height=2.6in]{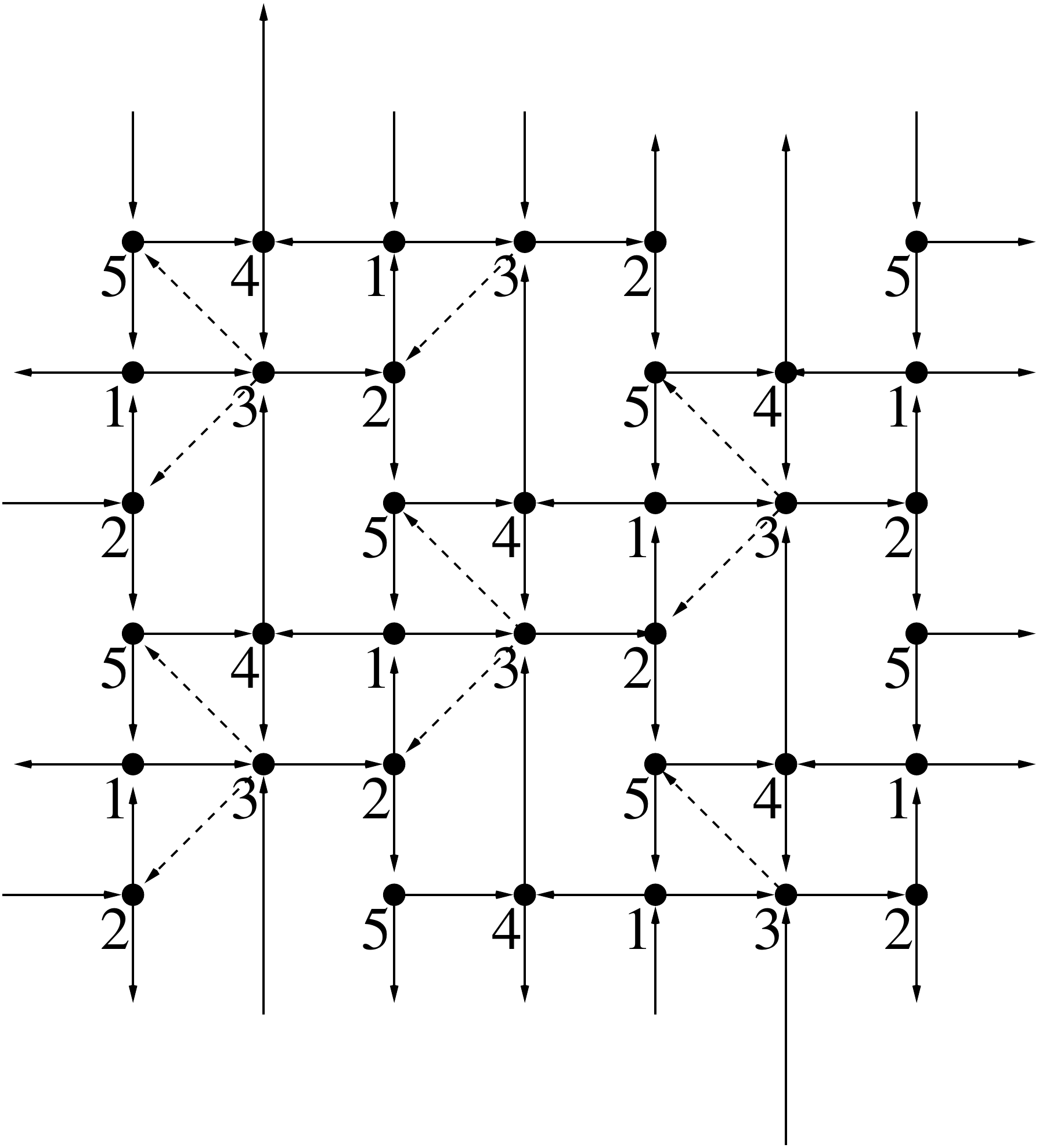}
\caption{The doubly-periodic quiver associated to this brane tiling.  Notice that this quiver is an unfolding of $Q_5^{(1,2)}$, but is different 
than the construction of Section \ref{Sec:Brane}.}
\label{Fig:Somos5AltQuiver}
\end{figure}

\begin{figure}
\includegraphics[height=2.6in]{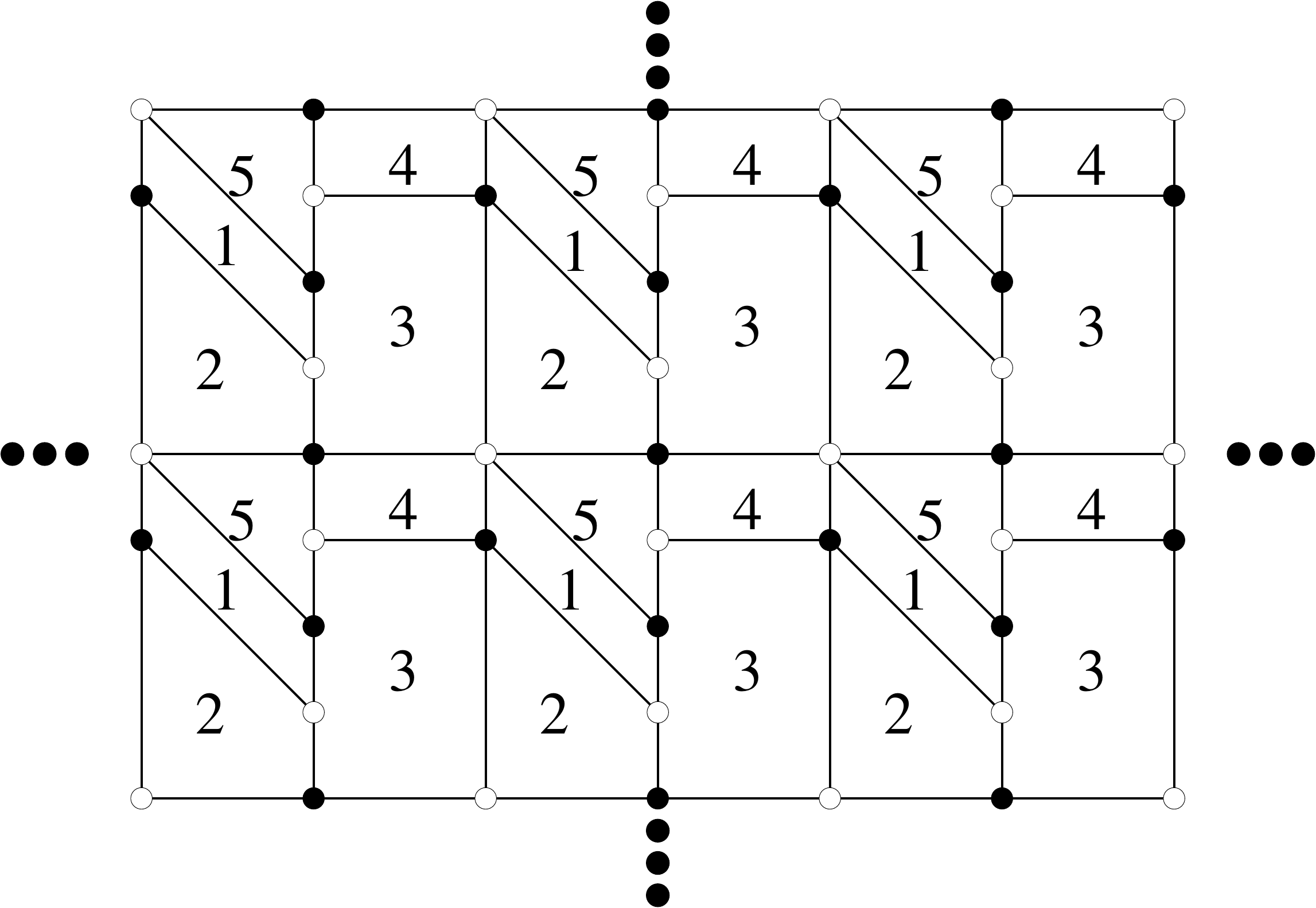}
\caption{
The Brane Tiling ${\mathcal{T}'}_5^{(1,2)}$ adapted from \cite[Figure 15]{Franco2006Jan2} up to relabeling vertices according to permutation $(1,2,5,3)(4)$.  
}
\label{Fig:Somos5AltBrane}
\end{figure}

\section{Main Results}

\label{sec:main}

Before stating our main result, we introduce certain functions that are necessary in the development. 
\begin{defi}[Weight functions of edges] \label{Def:Weight}
	Given a subgraph $G$ of a brane tiling (with face labels
	$F_i$), we define the weight $x(e)$ of an edge $e$ (straddling faces $F_i$ and $F_j$) to be $x(e) =
	\frac{1}
	{x_ix_j}$ . Given a
	perfect matching $M$ of $G$, we define $x(M) = \prod_{e\in M}
	x(e)$.
\end{defi}

For every subgraph $G$, we also have to define a certain monomial that is given by the labels of the faces appearing in $G$ and its boundary in the ambient tiling $\mathcal{T}$.

\begin{defi}[Covering monomial]\label{Def:CovMonom}
	Given a \emph{nonempty} subgraph $G$ of a brane tiling $T$ (with face
	labels $F_i$), let $\overline{G}$ denote the subgraph of $T$ consisting of all faces that are incident to an edge appearing
	in $G$. In particular, $\overline{G}$ contains $G$ as a proper subgraph, as well as a “ring” of outer faces. 
	Then for any face $F$ of $G$, with label $i$, we define $m(F) = x_i^{\frac{\# \text{edges in F}}{2}-1}$.
    
    We denote $I(G)=\prod_{F\in G} m(F)${, which stands for ``interior'' faces}.
For any of the open faces
	$F \in  \overline{G}\bs G$, with label $i$, we define $m(F) = x_i^{\lceil\frac{\# \text{edges in F incident to G}}{2} \rceil}.$ 
    We denote $O(G)=\prod_{F\in \overline{G}\bs G} m(F)${, which stands for ``outer" or ``open'' faces}.
Then the covering monomial of $G$
	is defined to be $cm(G) = \prod_{F \in G} m(F)=I(G)O(G)$.
\end{defi}

\begin{rem}\label{rem:base-case-x}
    For the Gale-Robinson sequence of type $(r,s,N)$, we \emph{declare} that the empty pinecone $G_n^{(r,s,N)}$ where $n\le N$ have $O(G_n^{(r,s,N)}):=x_n$. This validates our main Theorem \ref{Thm:GR-WH} even if the pinecone is empty. We will further set the convention of what the outer face is for the empty pinecone $G_n^{(r,s,N)}$ in Section \ref{sec:condensation-x}, and the declaration will be shown to be coherent with our proof method in Section \ref{sec:condensation-x}.
\end{rem}

We additionally utilize height functions, which have appeared elsewhere in the literature such as \cite{MUSIKER20112241positivity, Propp02, Thurston01101990}.

\begin{defi} \label{Def:Height}
	Fix the minimal matching $M_-$. The height $y(M)$ of a perfect matching $M$ is given by $$
	y(M)=\prod_{i=1}^N \prod_{\text {Face } F \text { of graph } G \text { labeled as } i} y_i^{\# \text { cycles of } M \oplus M_{-} \text {enclosing the face } F} .
	$$	
\end{defi}

Given the above definitions, the \emph{weight} of a graph $G$ is defined as 
\begin{equation} \label{Eq:GraphWeight} w(G)=cm(G) \cdot \sum_{M} x(M)y(M),\end{equation} where the sum is taken over all perfect matchings $M$ of $G$.
With our preliminary definitions and background material described, we now review our main theorem and the three step proof outlined in \cite{MusikerJeongZhang2013Jan}.

\begin{theorem} \label{Thm:GR-WH}
Let $\widehat{\mathcal{A}}_{Q_N^{(r,s)}} \subset \mathbb{Q}[y_1,y_2, \dots, y_N][x_1^\pm,x_2^\pm, \dots, x_N^\pm]$ denote the cluster algebra with principal coefficients associated to the Gale-Robinson quiver of type $(r,s,N)$.  For $n \in \{N+1,N+2,\dots\}$, define the cluster variables $\widehat{x_n}$ by mutating the initial seed $(\widehat{Q}_N^{(r,s)},\{x_1,x_2,\dots, x_N,y_1,y_2,\dots, y_N\})$ periodically by the sequence $1,2,3,\dots, N, 1,2,\dots$.  
Let $G_n^{(r,s,N)}$ (which equals $P(n;r,N-r,s,N-s)$ up to vertical reflection) be the pinecone graph constructed in Section \ref{Sec:pinecones} or in \cite{ProppMelouWest2009}.
Then, the Laurent expansion of $\widehat{x_n}$ is given by the following combinatorial formula
\begin{eqnarray} \label{eq:GR-WH} \widehat{x_n} = cm(G_n^{(r,s,N)})\left(\sum_{M \mathrm{~is~a~perfect~matching~of~} G_{n}^{(r,s,N)}} x(M)y(M)\right)\end{eqnarray}
where $x(M)$, $y(M)$ are the weights and heights of perfect matching $M$, respectively.  See Definitions \ref{Def:Weight}, 
\ref{Def:CovMonom}, and \ref{Def:Height}. 
\end{theorem}

\begin{definition}
\label{def:func_d}
Given integers $m$, $a$, and $b$, we let
$d(m, a, b)$ denote the number of ways to choose nonnegative integers $A$ and $B$ such that 
$$m=A\cdot a + B \cdot b.$$
\end{definition}

\begin{remark} The expression $d(m,a,b)$ is also known as the \emph{restricted partition function}, and due to a Theorem of Popoviciu, can be written explicitly as 
$$d(m,a,b) = \frac{m}{ab} - \bigg\{ \frac{b^{-1}n}{a}\bigg\} - \bigg\{ \frac{a^{-1}n}{b}\bigg\} + 1.$$
Here $a^{-1}$ denotes the inverse of $a$ modulo $b$, $b^{-1}$ denotes the inverse of $b$ modulo $a$, and $\{\frac{c}{d}\}$ denotes the fractional part of 
$\frac{c}{d}$.  See Theorem 1.5 of \cite{Beck-Robbins} for more details.  We will not need this formula in this treatment however. 
\end{remark}

\begin{proposition}
\label{prop:GRprinc}
Starting with the quiver $\widehat{Q}_N^{(r,s)}$, with principal coefficients, and mutating periodically at $1,2,3,\dots, N,1,2,\dots$, we obtain a Gale-Robinson recurrence (with coefficients) of the following form:
\begin{equation} \label{Eq:GRprinc} x_n x_{n-N} = x_{n-r} x_{n-N+r} ~~+~~ \prod_{i=1}^N y_i^{d(n-N-i,r,N-r)}~~x_{n-s}x_{n-N+s}\end{equation}
where the exponent of $y_i$ involves the restricted partition function as defined in Definition \ref{def:func_d} where $m=n-N-i$, $a=r$, and $b=N-r$.
\end{proposition}

Next, we show that the exponent $d(n-N-i,r,N-r)$ enumerates faces with a given labeling in the desired graph theoretical interpretation.

\begin{proposition} \label{Prop:cent-strip}
For $i \in \{1,2,\dots, N\}$, and an integer $n$ such that $n > N$, the number of faces labeled $i$ in the central strip $H_n^{(r,N)}$ of $G_n^{(r,s,N)}$ is given by the 
formula $d(n-N-i,r,N-r)$.
\end{proposition}

With this description of the face labels of $H_n^{(r,N)}$ in hand, we obtain the following.

\begin{proposition} \label{Prop:borders}
The western borders of $G_{n-r}^{(r,s,N)}$ and $G_n^{(r,s,N)}$ agree, as do the eastern (resp. northern, southern) borders of $G_{n-(N-r)}^{(r,s,N)}$ (resp. 
$G_{n-(N-s)}^{(r,s,N)}$, $G_{n-s}^{(r,s,N)}$) and $G_n^{(r,s,N)}$.  Furthermore, placing these four subgraphs along these borders, we see that 
$$G_{n-r}^{(r,s,N)} \cap G_{n-(N-r)}^{(r,s,N)} = G_{n-s}^{(r,s,N)} \cap G_{n-(N-s)}^{(r,s,N)} = G_{n-N}^{(r,s,N)} \cap G_{n-(N-r)}^{(r,s,N)}$$
as subgraphs.
\end{proposition}
Together with Remark \ref{rem:center-is-intersection}, this proposition justifies Remark \ref{rem:Aztec-diamond}.

\begin{proposition} \label{Prop:superpositions}
Let $w(G)=cm(G)\sum_M x(M)y(M)$ where the sum is taken over all perfect matchings $M$ of $G$. 
Then $x_n:=w(G_n^{r,s,N})$ satisfies the recurrence \begin{equation}
		x_n x_{n-N}=x_{n-r} x_{n-N+r}+\prod_{i=1}^N y_i^{d_i} x_{n-s} x_{n-N+s}
	\end{equation} where $d_i$ equals the number of faces labeled $i$ in the central strip $H^{(r;N)}_n$ of $G^{(r;s;N)}_n$.
\end{proposition}

Note that by Proposition \ref{Prop:cent-strip}, $d_i = d(n-N-i,r,N-r)$.
Putting these altogether allows us to prove the main result, Theorem \ref{Thm:GR-WH}, of our paper regarding Gale-Robinson sequences with principal coefficients, which will be done in Section \ref{sec:conclusions}. 

\section{Step 1: The recurrence with principal coefficients}

\label{sec:Step1}

Before proving these propositions, hence completing the proof of Theorem \ref{Thm:GR-WH}, we recall the $\bf{c}$-vectors introduced in Remark \ref{rem:cvectors}. In detail, we let $Q$ denote a quiver on the vertices $\{x_1,x_2,\dots, x_N, y_1, y_2, \dots, y_N\}$.  For $1 \leq i \leq N$, the {\bf c}\emph{-vector}, $\mathbf{c}_i$, is the $N$-length vector whose $j$th entry signifies the number of arrows $y_j \to x_i$ in $Q$.  This entry is negative if 
the arrows go from $x_i \to y_j$ instead.  In particular, for the initial quiver $\widehat{Q_N^{(r,s)}}$, each $\mathbf{c}_i$ equals the $i$th 
unit vector $\mathbf{e}_i$.  We let $\mathbf{c}_i^{(\ell)}$ denote the $i$th $\mathbf{c}$-vector after mutating $\ell$ times periodically along the sequence $1,2,\dots, N, 1,2, \dots$.    

Let $\underline{i} = \underline{i}(i,\ell)$ denote $\lfloor \frac{\ell + r -i}{N}\rfloor N + i$.  Holding $\ell$ constant, this function is piecewise-linear in $i$ with 
$$\underline{i} = \begin{cases}
i &\mathrm{~for~}1 \leq i \leq \ell+r \\ 
i-N &\mathrm{~for~}\ell+r+1 \leq i \leq N+\ell+r \\ 
i-2N &\mathrm{~for~}N+\ell + r + 1 \leq i \leq 2N+\ell+r, \mathrm{~etc.}
\end{cases}$$

We define two auxiliary vectors associated with index $\underline{i} = \underline{i}(i,\ell)$, namely we use Definition \ref{def:func_d} to define
 $E_{\underline{i}} = \sum_{j=1}^N d(\underline{i}-j,r,N-r)~ \mathbf{e}_j$ and we use $\chi(X) = \begin{cases} 1 &\mathrm{~if~}X\mathrm{~is~true} \\ 0 &\mathrm{~otherwise}\end{cases}$ to define $F_{\underline{i}} = \sum_{j=1}^N \chi(j \equiv \underline{i} +N \mod N-r) ~ \mathbf{e}_j$.  
 By the assumption $r \leq N/2$, it follows that 
$F_{\underline{i}}$ consists of one or two terms, with the latter case only happening if $r = s = N/2$.

\begin{lemma} \label{Lem:cdef}
For $\ell \geq r$, 

$$\mathbf{c}_i^{(\ell)} = \begin{cases} -E_{\underline{i}} \mathrm{~~if~~} \ell+1 -r \leq \underline{i} \leq \ell \\
                                        E_{\underline{i}} \mathrm{~~if~~} \ell+1 \leq \underline{i} \leq \ell + r \\
                                        F_{\underline{i}} \mathrm{~~otherwise.~~}
                          \end{cases}$$

\end{lemma}

Note that in the event $r = s = N/2$, then by abuse of notation, we still define the $\mathbf{c}_i^{(\ell)}$'s the same way, however it follows that they all equal $\pm E_{\underline{i}}$.

\begin{proof}
By assumption, $\mathbf{c}_i^{(0)} = \mathbf{e}_i$ for $1 \leq i \leq N$.  By direct computation via quiver mutation, we observe that 
$$\mathbf{c}_i^{(r)}  = \begin{cases} -\mathbf{e}_i \mathrm{~~if~~} 1 \leq i \leq r \\
                                       \mathbf{e}_{i-r} + \mathbf{e}_{i}  \mathrm{~~if~~} r+1 \leq i \leq \min(N-r,2r) \\
                                       \mathbf{e}_{i-(N-r)} + \mathbf{e}_{i-r} + \mathbf{e}_{i}  \mathrm{~~if~~} N-r + 1 \leq i \leq 2r \mathrm{~(if~possible)~}\\
                                       \mathbf{e}_i \mathrm{~~if~~} 2r+1 \leq i \leq N-r \mathrm{~(alternatively),~~and~~} \\
                                       \mathbf{e}_{i-(N-r)} + \mathbf{e}_i \mathrm{~~if~~} \max(N-r + 1,2r+1) \leq i \leq N.
                          \end{cases}$$

Furthermore, $\underline{i}(i, r) = i$ for $1 \leq i \leq 2r$ and $\underline{i}(i, r) = i - N$ for $2r+1 \leq i \leq N$.  Thus plugging in $\ell=r$, the computation via mutation 
agrees with the formula of Lemma \ref{Lem:cdef}.  We use the fact (for $1 \leq i \leq r$ and $1 \leq j \leq N$) that the equation $(i-j) = A\cdot r + B \cdot (N-r)$ only has a nonnegative integer solution when $i=j$ and $A=B=0$. Hence $d(i-j, r, N-r) = \chi(i=j)$ in this case. Also, the above list of cases is exhaustive since $r \leq N/2$ implies that $2(N-r) + 1 \leq i \leq 2r$ is impossible.

For $\ell > r$, we then note, by quiver mutation and the above formula by induction, that $\mathbf{c}_i^{(\ell)} = \mathbf{c}_i^{(\ell-1)}$ except when $\ell-i \equiv r$, $N-r$, or $N \mod N$.
More specifically, 
\begin{eqnarray}
\label{eq:c}\mathbf{c}_{\ell \mod N}^{(\ell)} &=& - \mathbf{c}_{\ell \mod N}^{(\ell-1)}, \\
\label{eq:c+r}\mathbf{c}_{\ell+r \mod N}^{(\ell)} &=& \mathbf{c}_{\ell \mod N}^{(\ell-1)} + \mathbf{c}_{\ell+r \mod N}^{(\ell-1)}, \mathrm{~and} \\ 
\label{eq:c-r}\mathbf{c}_{\ell-r \mod N}^{(\ell)} &=& \mathbf{c}_{\ell-r \mod N}^{(\ell-1)} + \mathbf{c}_{\ell \mod N}^{(\ell-1)}.
\end{eqnarray}

Next, we verify the identity $F_{\underline{i}} = E_{\underline{i}+N} - E_{\underline{i}+N-r}$.  This results from the following:  If $\underline{i} - j +N$ is not divisible by $N-r$, then 
any solution to $$A\cdot r + B\cdot (N-r) = \underline{i} - j + N,$$ with $A$ and $B$ nonnegative integers, must satisfy $A > 0$.  Consequently, the expression $(A-1)\cdot r + B\cdot(N-r)$ sums to $\underline{i} - j + (N-r)$, and it follows that the coefficient of $\mathbf{e}_j$ in $E_{\underline{i}+N}$ agrees with the coefficient of $\mathbf{e}_j$ in $E_{\underline{i}+N-r}$.  On the other hand, when 
$N-r$ divides $\underline{i} - j +N$, there is one additional solution, $0\cdot r + B\cdot (N-r) = \underline{i} - j + N$, hence $d(\underline{i}-j+N,r, N-r) = d(\underline{i}-j+N-r,r, N-r) + 1$ in these cases.  Comparing this result with the definition of $F_{\underline{i}}$ completes the verification.

Let $i \equiv \ell \mod N$.  We have $\mathbf{c}_i^{(\ell-1)} = E_{\underline{i}(i,\ell-1)}$ and by applying (\ref{eq:c}), and the fact that $i \not \equiv \ell + r \mod N$, we get $\mathbf{c}_i^{(\ell)} = - E_{\underline{i}(i,\ell-1)} = 
- E_{\underline{i}(i,\ell)}$ as desired.

Next, assume that $i \equiv \ell + r \mod N$.  We obtain
\begin{equation} 
\label{eq:cFE} 
\mathbf{c}_i^{(\ell-1)} = F_{\underline{i}(i,\ell-1)} = E_{\underline{i}(i,\ell-1)+N} - E_{\underline{i}(i,\ell-1)+N-r},
\end{equation} 
where the first equality is by induction and the second is by the above identity.  Furthermore, $\underline{i}(i,\ell) = \underline{i}(i,\ell-1) + N = \ell + r = \underline{i}(i+N-r,\ell-1)+r$ in this case, and hence (\ref{eq:cFE}) simplifies to 
$\mathbf{c}_i^{(\ell-1)} = E_{\underline{i}(i,\ell)} - \mathbf{c}_{i+N-r}^{(\ell-1)}$.
We conclude by (\ref{eq:c+r}) that  $\mathbf{c}_i^{(\ell)} = E_{\underline{i}(i,\ell)}$ as desired.  

Lastly, assume that $i \equiv \ell - r \mod N$.  Then $\mathbf{c}_i^{(\ell-1)} = - E_{\underline{i}(i,\ell-1)}$.  Using (\ref{eq:c-r}), we get 
$\mathbf{c}_i^{(\ell)} = -E_{\underline{i}(i,\ell-1)} + E_{\underline{i}(i+r,\ell-1)}$.  Since $i$ and $i+r \not\equiv \ell + r \mod N$, we rewrite this as 
$\mathbf{c}_i^{(\ell)} = -E_{\underline{i}(i,\ell)} + E_{\underline{i}(i,\ell)+r} = F_{\underline{i}(i,\ell)-(N-r)}$, which equals $F_{\underline{i}(i,\ell)}$ by definition.
\end{proof}

With Lemma \ref{Lem:cdef} in hand, we now proceed through the results necessary to prove Theorem \ref{Thm:GR-WH}, beginning with Proposition \ref{prop:GRprinc}.

\begin{proof} [Proof of Proposition \ref{prop:GRprinc}] Using straightforward computation for $0 \leq \ell < r$ and Lemma \ref{Lem:cdef} for $\ell \geq r$, we obtain $\mathbf{c}_{\ell+1 \mod N}^{(\ell)} = E_{\ell+1}$.  Re-indexing $\ell$ as $n - N - 1$ yields (\ref{Eq:GRprinc}), completing the proof. 
\end{proof}

\section{Step 2: Counting faces in horizontal strips of pinecones}

\label{sec:Step2}

We next illustrate that our construction of pinecones, and our construction of central strips $H_n^{(r,N)}$ in particular, utilize squares and rectangles with face labels whose frequencies match the enumerative counts given by the restricted partition function $d(n-N-i,r,N-r)$  This result was stated precisely as  Proposition  \ref{Prop:cent-strip}.

\begin{proof} [Proof of Proposition \ref{Prop:cent-strip}.]
First off, if $1 \leq n-N \leq r$, then $d(n-N-i,r,N-r) = \begin{cases} 1 \mathrm{~if~} i = n-N \\ 0 \mathrm{~otherwise}\end{cases}.$  
This agrees with $H_n^{(r,N)}$ consisting of a single square
labeled $\overline{n} = n-N$, in this case, as in Definition \ref{Def:pinecones}.

Then for $n-N > r$, we compare $d(n-N-i, r, N-r)$ to $d(n-N-i-r, r, N-r)$ and $H_n^{(r,N)}$ to $H_{n-r}^{(r,N)}$.  Analogous to the proof of 
$F_{\underline{i}} = E_{\underline{i}+N} - E_{\underline{i}+N-r}$, we observe that 

\small
$$d(n-N-i, r, N-r) = \begin{cases} d(n-N-i-r, r, N-r)+1 \mathrm{~~if~} n-N-i \mathrm{~is~divisible~by~} N-r,\\ d(n-N-i-r, r, N-r) \mathrm{~otherwise.}\end{cases}$$ 

\normalsize

\noindent Hence, to agree with (\ref{Eq:GRprinc}), the central strip $H_{n}^{(r,N)}$ should consist of $H_{n-r}^{(r,N)}$ plus (one or two) faces labeled $\underline{i}$ satisfying $\underline{i} \equiv n-N \mod N-r$.  Keeping in mind Corollary \ref{Cor:sqhex}, we obtain a rectangle of length two consisting of a single hexagon labeled $\underline{i}$ when $r+1\leq \underline{i} \leq N-r$ or two squares labeled left-to-right as $(N-r+\underline{i})$ then $\underline{i}$, for $1 \leq i\leq r$.  Thus, from a graph theoretic point of view, we have increased the size of the central strip by length two, while keeping the leftmost face label as $\overline{n} = n-N \mod r$, exactly agreeing with the description of $H_{n}^{(r,N)}$.  
\end{proof}

As a corollary of Proposition \ref{Prop:cent-strip}, we obtain a statement about the multiset of face labels arising as the three different superpositions of subgraphs that are relevant for Kuo Condensation.  See the statement in Proposition \ref{Prop:borders}.

\begin{proof} [Proof of Proposition \ref{Prop:borders}]
Since $r$, $s$, and $N$ are fixed, for notational convenience, we suppress the exponents of $H_n^{(r,N)}$ and $G_n^{(r,s,N)}$ for the remainder of this proof.  Recall that $G_{n}$ is constructed by gluing horizontal strips together as 
$$G_n = ( \dots \cup H_{n-2(N-s)} \cup H_{n-(N-s)} \cup \mathbf{H_{n}} \cup H_{n-s} \cup H_{n-2s} \cup \dots).$$    Let $0$ index the row (reading from top to bottom) containing $H_n$, denoted in boldface.  Centering $G_{n-N}$ on top of $G_{n}$, we  construct $G_{n-N}$ as 
$$G_{n-N} = (\dots \cup H_{n-N-2(N-s)} \cup H_{n-N-(N-s)} \mathbf{\cup H_{n-N}} \cup H_{n-N-s} \cup H_{n-N-2s} \cup \dots)$$ where $H_{n-N}$ appears on row $0$ (and to the right of the start of $H_n$).  

\vspace{0.5em}

By construction, see Definition \ref{Def:pinecones}, $G_{n-r}$ is a left-justified copy of $G_{n}$ in which each row of length at least three is exactly two units shorter, and any rows of length one are erased.  Analogously, $G_{n-(N-r)}$ is a right-justified copy of $G_n$ in which each row is an even number of units shorter, chosen so that row $0$ begins with a first square labeled $n \mod r$ (rather than label $\overline n = n-N \mod r$).  Since the overlap, in row $0$, of these two subgraphs is $H_{n-N}$, and the rows above and below shrink accordingly, it follows that $G_{n-r} \cap G_{n-(N-r)} = G_{n-N}$.  

\vspace{0.5em}

On the other hand, if we overlay $G_{n-(N-s)}$ (centered on row $1$) and $G_{n-s}$ (centered on row $-1$), then we construct these graphs via the collections
$$G_{n-(N-s)} = ( \dots \cup H_{n-2(N-s)} \cup \mathbf{H_{n-(N-s)}} \cup H_{n-N} \cup H_{n-N-s} \cup H_{n-N-2s} \cup \dots)$$ and 
$$G_{n-s} = ( \dots \cup H_{n-N-2(N-s)} \cup H_{n-N-(N-s)} \cup H_{n-N} \cup \mathbf{H_{n-s}} \cup H_{n-2s} \cup \dots).$$  The top (resp. bottom) of $G_{n-(N-s)}$ (resp. $G_{n-s}$ agrees with the top (resp. bottom) of $G_n$ as desired, and it is also clear that row $0$ of both is $H_{n-N}$.  Comparing the other intersections, row-by-row, we conclude 
$G_{n-(N-s)} \cap G_{n-s}  = G_{n-N}$.
\end{proof}

\section{Step 3: Kuo condensation for pinecones with weights and heights}
\label{sec:Step3}
We now turn to Proposition \ref{Prop:superpositions}, whose proof heavily relies on Kuo's condensation, which is the main tool \cite{ProppMelouWest2009} used to treat pinecones. After the necessary ingredients are introduced, from Section \ref{sec:condensation-x} on, we start proving the proposition.

\subsection{Kuo's graphical condensation}
\begin{defi}
	A perfect matching $M$ of a graph $G$ is a subgraph of  $G$ such that every vertex of $G$ is incident to exactly one edge of $M$. 
\end{defi}

\begin{thm}[\cite{KUO200429}, The condensation lemma]
	Let $G$ be a planar bipartite graph, and $a,b,c,d$ in clockwise order with \emph{$a,c$ black and $b,d$ white}. Denote $m(\cdot)$ the number of perfect matchings, then 
	$${m(G) m(G-\{a,b,c,d\})=m(G-\{a,b\})m(G-\{c,d\})+m(G-\{a,d\})m(G-\{b,c\})}$$
    where $G-V$ is the full subgraph of $G$ containing those vertices not contained in the vertex subset $V$.
\end{thm}

One can establish a bijection ($M(\cdot)$ denotes the set of perfect matchings) $$\Delta:{M(G) \times M(G-\{a,b,c,d\}}\to M(G-\{a,b\})\times M(G-\{c,d\}) \amalg M(G-\{a,d\})\times (G-\{b,c\})$$ to prove the condensation lemma. Since we will need this bijection during the development, we roughly describe below its construction.

$\implies$: Superimpose\footnote{The superimposition of two subgraphs of a graph is a multi-subgraph of the original graph. The multiplicity of an edge in the superimposition equals the times it appears in the two subgraphs.} the perfect matchings $M_A$ and $M_C$ given from the pair ${M(G) \times M(G-\{a,b,c,d\})}$. The result will be a spanning subgraph (possibly with doubled edges) denoted by $M_A\oplus M_C$ with all the vertices of degree $2$, except at $\{a,b,c,d\}$ whose degree is $1$. Therefore, in $M_A\oplus M_C$, there exist $3$ configurations: open paths,  closed cycles, or  doubled edges.

It is easy to show that there are exactly two open paths in $M_A\oplus M_C$, with two possibilities\footnote{The subscripts A, C, W, E, N, S stands for all, center, west, east, north, south respectively.}:
\begin{itemize}
	\item One  connecting $\{a,b\}$ and the other connecting $\{c,d\}$. This  will imply $$\Delta(M_A,M_C)=(M_N,M_S)\in  M(G-\{c,d\})\times  M(G-\{a,b\});$$
	\item Or one connecting $\{a,d\}$ and the other connecting $\{b,c\}$. This will imply $$\Delta(M_A,M_C)=(M_W,M_E)\in M(G-\{b,c\})\times M(G-\{a,d\}).$$
\end{itemize}  

\begin{rem}
	These implications are nontrivial and rely on the assumption {$a,c$ black and $b,d$ white}. If we instead require $a,b$ black and $c,d$ white, we cannot conclude which set $\Delta(M_A,M_C)$ lies in simply from how $\{a,b,c,d\}$ are connected. However, there is still a bijection described in \cite{KUO200429} in this case, but we will not need this case in our treatment.
\end{rem}
Let's assume without loss of generality that $\Delta(M_A,M_C)=(M_N,M_S)\in  M(G-\{c,d\})\times  M(G-\{a,b\})$.

The doubled edges are easy to handle. If $e$ is a doubled edge, then $e\in M_A$ and $e\in M_C$, and it will be contained in both $M_N$ and $M_S$.

The two paths connecting $\{a,b\}$ and $\{c,d\}$ are handled as follows: {Since $G$ is bipartite, and given the colors of $a,b,c,$ and $d$, one} can show that both paths are of \emph{odd} length. For the path $a\to b$ ($c\to d$), number the edges starting from one end. Then the odd edges in 
$a\to b$ {and the even edges of $c\to d$ belong to $M_N$, while the even edges in $a\to b$ and the odd edges of $c\to d$ belong to $M_S$.}  Since the paths are of odd length, it does not matter which end we start from.

It remains to consider the cycles appearing in $G-\{a,b,c,d\}$. Take an arbitrary cycle $Z$, and number clockwise the edges of $Z$ successively starting from an arbitrary point on the cycle. $Z$ has two perfect matchings, consisting of odd and even edges respectively,  each of which belongs to  $M_A$ or $M_C$ separately. It remain to choose which one belongs to $M_N$ or $M_S$.  

On one hand, we are \emph{free} to choose whether it belongs to  $M_N$ or $M_S$. On the other hand, there is no preferred way of such choice, and indeed Kuo only showed equicardinality of certain sets, without specifying an explicit bijection between them. We will establish later that our results do not depend on the choice of this bijection.

$\impliedby$: Superimpose the perfect matchings $M_N$ and $M_S$ given from $M(G-\{c,d\})\times  M(G-\{a,b\})$ (or $M_W$ and $M_E$ from $M(G-\{b,c\})\times  M(G-\{a,d\})$). We will arrive at two paths $a\to b$ and $c\to d$ (or $a\to d$ and $b\to c$), plus some doubled edges and cycles. The doubled edges are distributed to $M_A$ and $M_C$. The odd edges in \emph{both} paths belong to $M_A$, while the even edges in \emph{both} paths belong to $M_C$, regardless if we are superimposing $M_N$ and $M_S$ or $M_W$ and $M_E$. The two perfect matchings for each cycle $Z$ are free to be chosen to belong to either $M_A$ or $M_C$.

\begin{rem}
	Again, the assertion that we will arrive at two paths $a\to b$ and $c\to d$ from superimposing $M_N$ and $M_S$ (or arrive at $a\to d$ and $b\to c$ from superimposing $M_W$ and $M_E$) is nontrivial and relies on the assumption {$a,c$ black and $b,d$ white}.
\end{rem}

In any direction, we notice that the superimposition of the  two perfect matchings on each side are the same spanning subgraph of $G$. Therefore we introduce the following notations.

\begin{defi}
	Denote $M_A\oplus M_C \sim M_W\oplus M_E$ (or $M_N\oplus M_S$) if the superimposition of $M_A$ and $M_C$ is the same as  the superimposition of $M_W$ and $M_E$ (or  $M_N$ and $M_S$). 
	
	On the other hand, $M_A\oplus M_C \approx M_W\oplus M_E$ (or $M_N\oplus M_S$) if $\Delta(M_A,M_C)=(M_W,M_E)$ (or $(M_N,M_S)$). Notice that $\approx$ is finer than $\sim$ and depends on the exact bijection $\Delta$ used, i.e.,  how the bijection handles cycles arising in the superimposition. See the figure below for an illustration.
\end{defi}
\begin{figure}[H]
	\centering

\tikzset{every picture/.style={line width=0.75pt}} 

\begin{tikzpicture}[x=0.75pt,y=0.75pt,yscale=-0.35,xscale=0.5]

\draw    (41.14,11.42) -- (40.21,204.27) ;
\draw [color={rgb, 255:red, 208; green, 2; blue, 27 }  ,draw opacity=1 ][line width=2.25]    (191.16,11.91) -- (190.23,204.76) ;
\draw    (41.14,11.42) -- (191.16,11.91) ;
\draw    (40.21,204.27) -- (190.23,204.76) ;
\draw    (80.77,63.46) -- (80.31,151.73) ;
\draw  [dash pattern={on 0.84pt off 2.51pt}]  (155.78,63.68) -- (155.32,151.95) ;
\draw  [dash pattern={on 0.84pt off 2.51pt}]  (80.77,63.46) -- (155.78,63.68) ;
\draw  [dash pattern={on 0.84pt off 2.51pt}]  (80.31,151.73) -- (155.32,151.95) ;
\draw [color={rgb, 255:red, 208; green, 2; blue, 27 }  ,draw opacity=1 ][line width=2.25]    (41.14,11.42) -- (80.77,63.46) ;
\draw  [dash pattern={on 0.84pt off 2.51pt}]  (191.16,11.91) -- (155.78,63.68) ;
\draw [color={rgb, 255:red, 208; green, 2; blue, 27 }  ,draw opacity=1 ][line width=2.25]    (40.21,204.27) -- (80.31,151.73) ;
\draw  [dash pattern={on 0.84pt off 2.51pt}]  (190.23,204.76) -- (155.32,151.95) ;
\draw [color={rgb, 255:red, 74; green, 144; blue, 226 }  ,draw opacity=1 ][line width=2.25]    (280.53,15.19) -- (279.6,208.04) ;
\draw [color={rgb, 255:red, 144; green, 19; blue, 254 }  ,draw opacity=1 ][line width=2.25]    (430.54,15.68) -- (429.62,208.53) ;
\draw    (280.53,15.19) -- (430.54,15.68) ;
\draw    (279.6,208.04) -- (429.62,208.53) ;
\draw    (320.16,67.23) -- (319.7,155.5) ;
\draw [color={rgb, 255:red, 74; green, 144; blue, 226 }  ,draw opacity=1 ][line width=2.25]    (395.17,67.45) -- (394.7,155.72) ;
\draw    (320.16,67.23) -- (395.17,67.45) ;
\draw    (319.7,155.5) -- (394.7,155.72) ;
\draw [color={rgb, 255:red, 208; green, 2; blue, 27 }  ,draw opacity=1 ][line width=2.25]    (280.53,15.19) -- (320.16,67.23) ;
\draw    (430.54,15.68) -- (395.17,67.45) ;
\draw [color={rgb, 255:red, 208; green, 2; blue, 27 }  ,draw opacity=1 ][line width=2.25]    (279.6,208.04) -- (319.7,155.5) ;
\draw    (429.62,208.53) -- (394.7,155.72) ;
\draw [color={rgb, 255:red, 74; green, 144; blue, 226 }  ,draw opacity=1 ][line width=2.25]    (515.2,16.45) -- (514.27,209.3) ;
\draw [color={rgb, 255:red, 74; green, 144; blue, 226 }  ,draw opacity=1 ][line width=2.25]    (665.21,16.94) -- (664.29,209.79) ;
\draw    (515.2,16.45) -- (665.21,16.94) ;
\draw    (514.27,209.3) -- (664.29,209.79) ;
\draw  [dash pattern={on 0.84pt off 2.51pt}]  (554.83,68.48) -- (554.37,156.76) ;
\draw [color={rgb, 255:red, 74; green, 144; blue, 226 }  ,draw opacity=1 ][line width=2.25]    (629.84,68.71) -- (629.37,156.98) ;
\draw  [dash pattern={on 0.84pt off 2.51pt}]  (554.83,68.48) -- (629.84,68.71) ;
\draw  [dash pattern={on 0.84pt off 2.51pt}]  (554.37,156.76) -- (629.37,156.98) ;
\draw  [dash pattern={on 0.84pt off 2.51pt}]  (515.2,16.45) -- (554.83,68.48) ;
\draw    (665.21,16.94) -- (629.84,68.71) ;
\draw  [dash pattern={on 0.84pt off 2.51pt}]  (514.27,209.3) -- (554.37,156.76) ;
\draw    (664.29,209.79) -- (629.37,156.98) ;
\draw    (425.57,296.83) -- (424.65,489.68) ;
\draw [color={rgb, 255:red, 245; green, 166; blue, 35 }  ,draw opacity=1 ][line width=2.25]    (575.59,297.31) -- (574.67,490.16) ;
\draw    (425.57,296.83) -- (575.59,297.31) ;
\draw    (424.65,489.68) -- (574.67,490.16) ;
\draw    (465.21,348.86) -- (464.74,437.13) ;
\draw [color={rgb, 255:red, 245; green, 166; blue, 35 }  ,draw opacity=1 ][line width=2.25]    (540.21,349.08) -- (539.75,437.36) ;
\draw    (465.21,348.86) -- (540.21,349.08) ;
\draw    (464.74,437.13) -- (539.75,437.36) ;
\draw [color={rgb, 255:red, 245; green, 166; blue, 35 }  ,draw opacity=1 ][line width=2.25]    (425.57,296.83) -- (465.21,348.86) ;
\draw    (575.59,297.31) -- (540.21,349.08) ;
\draw [color={rgb, 255:red, 245; green, 166; blue, 35 }  ,draw opacity=1 ][line width=2.25]    (424.65,489.68) -- (464.74,437.13) ;
\draw    (574.67,490.16) -- (539.75,437.36) ;
\draw [color={rgb, 255:red, 65; green, 117; blue, 5 }  ,draw opacity=1 ][line width=2.25]    (131.94,298.08) -- (131.02,490.93) ;
\draw [color={rgb, 255:red, 65; green, 117; blue, 5 }  ,draw opacity=1 ][line width=2.25]    (281.96,298.57) -- (281.03,491.42) ;
\draw    (131.94,298.08) -- (281.96,298.57) ;
\draw    (131.02,490.93) -- (281.03,491.42) ;
\draw  [dash pattern={on 0.84pt off 2.51pt}]  (171.57,350.12) -- (171.11,438.39) ;
\draw  [dash pattern={on 0.84pt off 2.51pt}]  (246.58,350.34) -- (246.12,438.61) ;
\draw  [dash pattern={on 0.84pt off 2.51pt}]  (171.57,350.12) -- (246.58,350.34) ;
\draw  [dash pattern={on 0.84pt off 2.51pt}]  (171.11,438.39) -- (246.12,438.61) ;
\draw  [dash pattern={on 0.84pt off 2.51pt}]  (131.94,298.08) -- (171.57,350.12) ;
\draw  [dash pattern={on 0.84pt off 2.51pt}]  (281.96,298.57) -- (246.58,350.34) ;
\draw  [dash pattern={on 0.84pt off 2.51pt}]  (131.02,490.93) -- (171.11,438.39) ;
\draw  [dash pattern={on 0.84pt off 2.51pt}]  (281.03,491.42) -- (246.12,438.61) ;
\draw [color={rgb, 255:red, 74; green, 144; blue, 226 }  ,draw opacity=1 ][line width=2.25]    (341.41,91.18) -- (341.98,122.19) ;
\draw [color={rgb, 255:red, 208; green, 2; blue, 27 }  ,draw opacity=1 ][line width=2.25]    (341.41,91.18) -- (368.35,91) ;
\draw [color={rgb, 255:red, 208; green, 2; blue, 27 }  ,draw opacity=1 ][line width=2.25]    (341.98,122.19) -- (368.51,122.42) ;
\draw [color={rgb, 255:red, 74; green, 144; blue, 226 }  ,draw opacity=1 ][line width=2.25]    (368.35,91) -- (368.51,122.42) ;
\draw    (104.9,92.18) -- (105.47,123.19) ;
\draw [color={rgb, 255:red, 208; green, 2; blue, 27 }  ,draw opacity=1 ][line width=2.25]    (104.9,92.18) -- (131.85,92) ;
\draw [color={rgb, 255:red, 208; green, 2; blue, 27 }  ,draw opacity=1 ][line width=2.25]    (105.47,123.19) -- (132,123.42) ;
\draw    (131.85,92) -- (132,123.42) ;
\draw [color={rgb, 255:red, 74; green, 144; blue, 226 }  ,draw opacity=1 ][line width=2.25]    (578.9,99.18) -- (579.47,130.19) ;
\draw    (578.9,99.18) -- (605.85,99) ;
\draw    (579.47,130.19) -- (606,130.42) ;
\draw [color={rgb, 255:red, 74; green, 144; blue, 226 }  ,draw opacity=1 ][line width=2.25]    (605.85,99) -- (606,130.42) ;
\draw [color={rgb, 255:red, 65; green, 117; blue, 5 }  ,draw opacity=1 ][line width=2.25]    (194.9,381.18) -- (195.47,412.19) ;
\draw    (194.9,381.18) -- (221.85,381) ;
\draw    (195.47,412.19) -- (222,412.42) ;
\draw [color={rgb, 255:red, 65; green, 117; blue, 5 }  ,draw opacity=1 ][line width=2.25]    (221.85,381) -- (222,412.42) ;
\draw    (489.9,380.18) -- (490.47,411.19) ;
\draw [color={rgb, 255:red, 245; green, 166; blue, 35 }  ,draw opacity=1 ][line width=2.25]    (489.9,380.18) -- (516.85,380) ;
\draw [color={rgb, 255:red, 245; green, 166; blue, 35 }  ,draw opacity=1 ][line width=2.25]    (490.47,411.19) -- (517,411.42) ;
\draw    (516.85,380) -- (517,411.42) ;

\end{tikzpicture}

	\caption{Kuo's condensation of perfect matchings. First row: $M_W$, $M_W\oplus M_E$,  $M_E$. Second row: $M_C$ and $M_A$. The coloured edges are those edges contained in the perfect matchings. Note that we have attributed all the red edges in the inner cycle of superimposition to $M_A$, and blue to $M_C$. This will be denoted $M_A\oplus M_C\approx M_W\oplus M_E$.}
\end{figure}
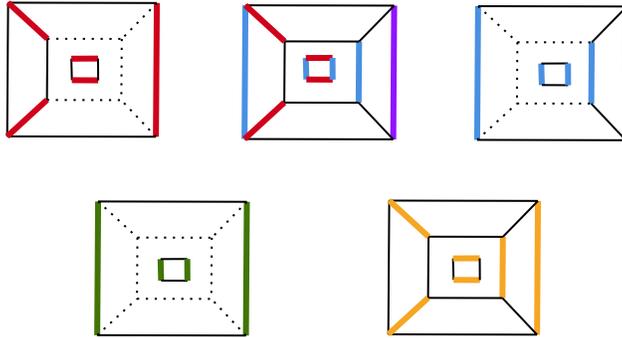
\subsection{Pinecones and their core}
We cite several properties of pinecones that will be useful later, and note that they are properties of pinecones defined in the \cite{ProppMelouWest2009}-fashion. After Remark \ref{rem:center-is-intersection}, we identify these pinecones with the Gale-Robinson subgraphs $G_n^{(r,s,N)}$ 
defined in Definition \ref{Def:pinecones}, as discussed in Remark \ref{rem:Aztec-diamond}. These pinecones are \emph{closed},  can be viewed as generalized Aztec diamonds, and can be placed on the $\ZZ^2$ square grid of vertices in the plane. Denote below $G$ the closed pinecone defined for the $n$th term of the Gale-Robinson sequence of type $(r,s,N)$.

\begin{defi}
    We refer to elements in $(\ZZ^2)^*$, the dual graph of the $\ZZ^2$ square grid, as \emph{cells}, and the closure of bounded connected components of $\R^2\bs G$ \emph{faces}. The boundary of each face of $G$ is either a quadrilateral called \emph{square}, or a hexagon called  \emph{rectangle}. We emphasize that square faces coincide with cells but rectangular faces span two cells. Using this correspondence, we will often let ``the face $(a,b)$'' for $(a,b) \in (\ZZ^2)^*$, refer to the face containing the cell $(a,b)$.  
\end{defi}

In the following, the  pinecone $G$ will be placed in such a way that the upper vertex of the rightmost edge (upper-right vertex for short) is at $(0,0)$. The rightmost cell will also be coordinated $(0,0)$, as the face on row $0$, column $0$. Declaring vertices $(a,b)$ with $a+b$ even being white and $a+b$ odd being black makes $G$ a planar bipartite graph (See also Figure 
\ref{fig:coordsystem-pinecone}).

\begin{rem}
    There are non-closed pinecones which are produced from closed pinecones by removing certain edges, without altering their placement on the square grid. See \cite{ProppMelouWest2009} for which edges we are allowed to remove. 
\end{rem}

\begin{figure}[H]
    \centering
    \includegraphics[scale=0.5]{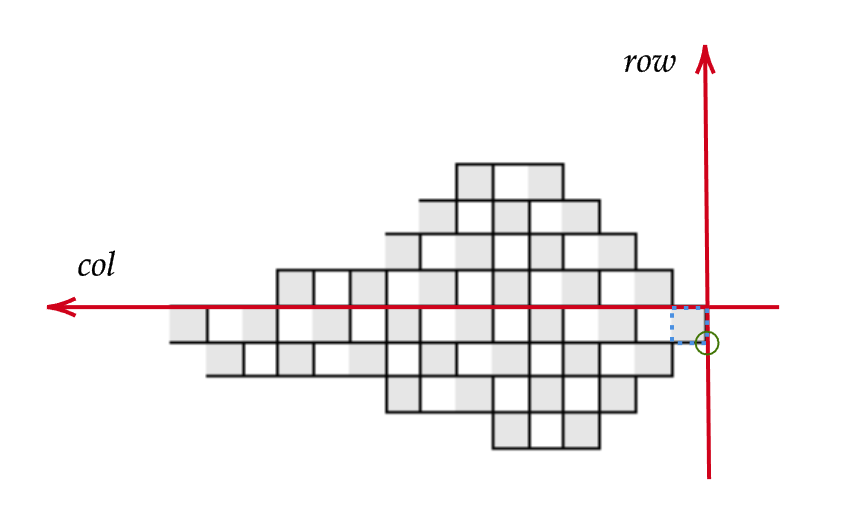}
    \caption{The coordinate system that we will use.  We always write row before column, so the circled vertex will have coordinate $(-1,0)\in \ZZ^2$. The grey cell enclosed in dashed box will have coordinate $(0,0)\in (\ZZ^2)^*$, and the white cell on its left will have coordinate $(0,1)\in (\ZZ^2)^*$. We remark that this system is a reflection of the one used in \cite{ProppMelouWest2009}.}
    \label{fig:coordsystem-pinecone}
\end{figure}

Then,
\begin{itemize}
	\item All the rectangles are horizontal, so that the pinecone can be split into 3 parts. The subgraph enclosing those faces with row coordinate $\ge 1$,  $\le -1$ or $=0$ will be called the upper part, lower part and \emph{central strip} respectively. In this way, the horizontal edges lying on row $0$ or $-1$ will be considered both in the upper (or lower) part and in the central strip.
	\item The central strip is always of \emph{odd} length, i.e., the four corners are given as $(0,0),(-1,0),(0,2k+1),(-1,2k+1)$. Once we fix a pinecone, we  keep $k$ for this parameter\footnote{See \cite{ProppMelouWest2009} for how to calculate this $k$, as well as the parameters $k_i$ and $l_i$ introduced later.}.
	\item The vertices $(0,1),(0,2k),(-1,2k),(-1,1)$ will be denoted $a,b,c,d$ in our development below.
	\item Not only the central strip, but also all the horizontal strips, are of odd length. Every horizontal strip in the upper part has corners $(i,i),(i-1,i)$ and $(i,i+2{k_i}+1),(i-1,i+2{k_i}+1)$. The integer sequence $k_i\ge 0$ is strictly decreasing as $i\ge 0$ increases. Similarly, every row in the lower part has corners $(-i,i),(-i-1,i)$ and $(-i,i+2{l_i}+1),(-i-1,i+2{l_i}+1)$. The integer sequence $l_i\ge0$ is strictly decreasing as $i\ge 1$ increases.
\end{itemize}

 \begin{defi}[\cite{ProppMelouWest2009}]\label{def:core}
	An edge $e$ of $G$ is  forced if it has to belong to every perfect matching of $G$. It is called forbidden if it does not belong to any perfect matching of $G$. The \emph{core} of $G$ is the subgraph of $G$ containing every edge that is neither forced nor forbidden, and the vertices that are incident to these edges.
\end{defi}

\begin{rem}[\cite{ProppMelouWest2009}]\label{rem:alternate-definition-core}
    There is an equivalent construction giving the core of $G$ when $G$ is a pinecone. Let $b_0$ be the column of the leftmost black square in row 0 of $G$, let $b_1$ be the column of leftmost black square in row 1 of $G$ that lies strictly to the right of $b_0$, let $b_2$ be the column of the leftmost black square in row 2 of $G$ that lies strictly to the right of $b_1$, and so on (proceeding upwards); likewise, let $b_{-1}$ be the column of the leftmost black square in row $-1$ of $G$ that lies strictly to the right of $b_0$, and so on (proceeding downwards). If at some point there is no black square that satisfies the requirement, we leave $b_m$ undefined. Consider all the faces of $G$ that lie in the same row as, and lie weakly to the right of, one of the $b_k$’s. This set of faces gives a closed pinecone, which equals the core of $G$.
\end{rem}

\begin{lem}[\cite{ProppMelouWest2009}]
	There is a bijection between the set  $\{\widehat{M}\}$ of perfect matchings of $G$ and that  $\{M\}$ of the core of $G$, by {deleting (resp. adding)} the forced edges of $G$. 
\end{lem}
 Recall the vertices $a,b,c,d$ as above. Denote temporarily $G_n^{r,s,N}$ as the \cite{ProppMelouWest2009}-constructed pinecone for the $n$th term in the Gale-Robinson sequence, which is ultimately shown to be equivalent to that defined in Definition \ref{Def:pinecones}.
\begin{lem}\label{lem:subpinecone-properties}

	In the below list, the core of the former graph is isomorphic to the latter pinecone graph:\begin{itemize}
        \item {$G \hspace{5.5em} =: \widehat{G_A}$ and $G_{n}^{r,s,N} =: G_A$}\footnote{$G$ is closed, so $G_A$ and $\widehat{G_A}$ are the same.},
		\item  $G-\{a,b,c,d\}=:\widehat{G_C}$ and $G_{n-N}^{r,s,N}=:G_C$,
		\item  $G-\{a,b\}=:\widehat{G_S}$ and $G_{n-N+s}^{r,s,N}=:G_S$,
		\item $G-\{b,c\}=:\widehat{G_W}$ and $G_{n-N+r}^{r,s,N}=:G_W$,
		\item $G-\{c,d\}=:\widehat{G_N}$ and $G_{n-s}^{r,s,N}=:G_N$,
		\item $G-\{a,d\}=:\widehat{G_C}$ and $G_{n-r}^{r,s,N}=:G_E$.
	\end{itemize}

\end{lem}
\begin{rem}
    The original Kuo's condensation induces a bijection $${M(G_A)\times M(G_C)}\to M(G_N)\times M(G_S)\amalg M(G_W)\times M(G_E).$$
    Below, we will reserve $M_A\oplus M_C\approx M_N\oplus M_S$ or $M_W\oplus M_E$ for a correspondence from this bijection. 
    
    On the other hand, we will use $\widehat{M_A}\oplus\widehat{M_C}\approx\widehat{M_N}\oplus\widehat{M_S}$ or $\widehat{M_W}\oplus\widehat{M_E}$ for a correspondence from the bijection  $$
    {M(G)\times M(G-\{a,b,c,d\})}\to M(G-\{a,b\})\times M(G-\{c,d\}) \amalg M(G-\{a,d\})\times (G-\{b,c\}).$$

    So as multi-subgraphs of $G$, $\widehat{M_A}\oplus\widehat{M_C}$ and $\widehat{M_W}\oplus\widehat{M_E}$ (or  $\widehat{M_W}\oplus\widehat{M_E}$) are equal, but $M_A\oplus M_C$ and $ M_N\oplus M_S$ (or $M_W\oplus M_E$) are not.
\end{rem}
\begin{rem}[\cite{ProppMelouWest2009}]\label{rem:center-is-intersection}
    It holds that $G_N\cap G_S=G_W\cap G_E=G_C$.
\end{rem}
For later use we  record: \begin{lem}\label{lem:subpinecone-place}
As a subgraph of $G$, $G_*$ has its upper-right vertex  placed at \begin{itemize}
		\item $(0,0)$ if $*=W$,
		\item $(0,2)$ if $*=E$ or $C$,
		\item $(1,1)$ if $*=N$,
		\item $(-1,1)$ if $*=S$.
	\end{itemize}
\end{lem}
		
We also describe for each pair of graphs in Lemma \ref{lem:subpinecone-properties} which edges are forced. For all these pairs,  the horizontal edges $(i,j)-(i,j+1)$ with $i>0$ and $i+j$ even, $(i,j)-(i,j+1)$ with $i<-1$ and $i+j$ odd, and vertical edges $(0,0)-(-1,0)$ and $(0,2k+1)-(-1,2k+1)$, that do not lie in $G_*$, are forced. The remaining edges forced, which lie on the \emph{central strip $H$}, are
\begin{itemize}
	\item $G-\{a,b,c,d\}$ and $G_C$:   horizontal edges $(0,2t)-(0,2t+1)$ 
	
	 and $(-1,2t)-(-1,2t+1)$ for $0<t<k$ that lie outside $G_*$;
	\item $G-\{a,b\}$ and $G_S$:  horizontal edges $(0,2t)-(0,2t+1)$ for $0<t<k$, 
	
	and $(-1,2t-1)-(-1,2t)$ for $0<t \le k$ that lie outside $G_*$;
	\item $G-\{c,d\}$ and $G_N$: horizontal edges of form $(0,2t-1)-(0,2t)$ for $0<t\le k$, 
	
	and of form $(-1,2t)-(-1,2t+1)$ for $0<t<k$ that lie outside $G_*$;
	\item $G-\{b,c\}$ and $G_W$:  horizontal edges $(0,2t)-(0,2t+1)$ 
	
	and $(-1,2t)-(-1,2t+1)$ for $0<t<k$ that lie outside $G_*$;
	
	\item $G-\{a,d\}$ and $G_E$: horizontal edges $(0,2t)-(0,2t+1)$ 
	
	and $(-1,2t)-(-1,2t+1)$ for $0<t<k$ that lie outside $G_*$.
\end{itemize}
Note that the pattern of edges added among $G_C, G_W, G_E$  are  different from that in $G_N$ and $G_S$. This difference plays a role in Lemmas \ref{lem:superimposition-difference} and \ref{lem:induction-step}.

In the next two subsections, we will prove Proposition \ref{Prop:superpositions}, and Equation (\ref{Eq:GRprinc}) 
in particular, by using Kuo's condensation to treat $x$-variables and $y$-variables separately.

\subsection{Kuo's condensation for $x$-variables}\label{sec:condensation-x}

\begin{thm}[\cite{Speyer2007May}]\label{thm:trivial-coefficients}
	Let $w'(G)=cm(G)\sum x(M)$. Then $w_n:=w'(G_n^{r,s,N})$ equals the $n$th term in the Gale-Robinson sequence of type $(r,s,N)$ \begin{equation*}
		w_nw_{n-N} = w_{n-r}w_{n-N+r} + w_{n-s}w_{n-N+s}
	\end{equation*} with initial values $x_1,\cdots,x_N$.
\end{thm}
\begin{rem}
    Originally in \cite{Speyer2007May}, Speyer provided two proofs, one using urban renewal and one using Kuo Condensation.  However his proof using Kuo Condensation did not utilize edge weights, a feature necessary for us to keep track of.  Our proof below modifies Speyer's proof to use Kuo's condensation with edge weights.
\end{rem}
\begin{proof}
     By Kuo's condensation, once we have shown that $$cm(G)cm(G_C)x(M_A)x(M_C)=cm(G_N)cm(G_S)x(M_N)x(M_S)$$ for $M_A\oplus M_C\approx M_N\oplus M_S$ and the similar statement for $M_A\oplus M_C\approx M_W\oplus M_E$, we will arrive at Speyer's theorem.

   We first treat the case where $M_A\oplus M_C\approx M_N \oplus M_S$. Recall the notations in Definition \ref{Def:Weight} and \ref{Def:CovMonom}.
    \begin{itemize}
        \item Given the fact that $G_N\cap G_S=G_C$ but $G_N\cup G_S\subsetneq G_A$, we know the inner face labels satisfy 
        $$I(G_N)I(G_S)\cdot \prod_{F\in G_A\bs (G_N\cup G_S)} m(F) =I(G_C)I(G_A).$$ 
        
        \item 
        
        We have the tautological equation on edge weights that will be used later:
        $$x(M_N)x(M_S)\cdot \prod_{e\in S_1} x(e)=x(M_A)x(M_C).$$
        Here $S_1$ is the set $(M_A \uplus M_C)\bs (M_N\uplus M_S)$ with  $\uplus$ meaning multiset union.
        By the paragraph after Lemma \ref{lem:subpinecone-place}, we can calculate that $S_1$ consists of vertical edges $(0,0)-(-1,0)$ and $(0,2k+1)-(-1,2k+1)$, and those edges of form $(0, 2t-1)-(0,2t)$ or $(-1,2t-1)-(-1,2t)$ for $0<t<k$ that lie outside both $M_N$ and $M_S$. We therefore conclude that this set is independent of the perfect matching  $M_A$ and $M_C$.
        \item 
        \begin{figure}[H]
            \centering
            \includegraphics[scale=.4]{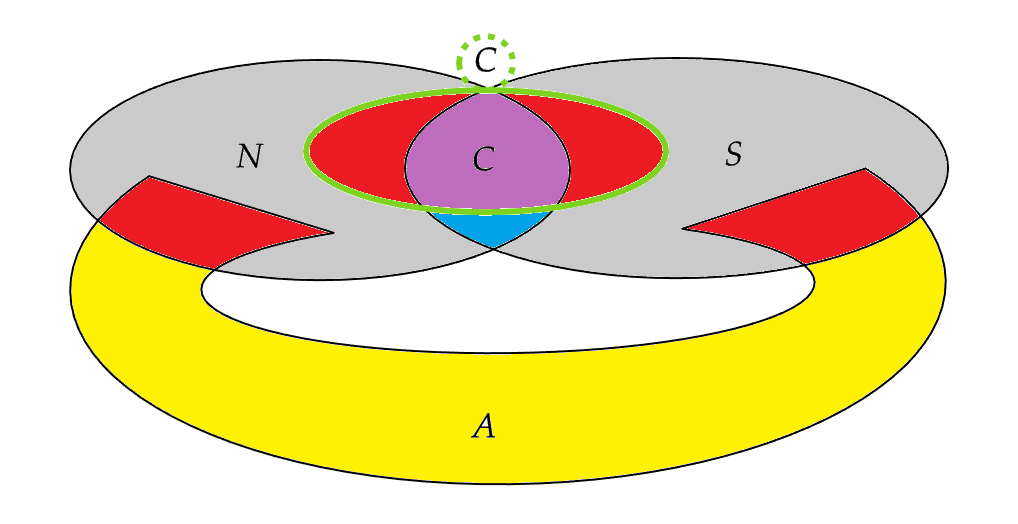}
            \caption{
            The Venn diagram for outer labels of $G_A, G_C, G_N$ and $G_S$, with those of $G_C$ enclosed by the green oval. The red parts do not contribute to any term in Equation (\ref{eqn:outer-eqn}). The blue part contributes to $O_{NS}$. The outer labels of $G_C$ are always outer labels of $G_N$ or $G_S$ except possibly when $G_C$ is empty, and are never outer labels of $G_A$. This diagram is also valid if we replace $N$ and $S$ with $W$ and $E$, with the modification that the blue part disappears  and $O_{NS}=1$.}
            \label{fig:venn-outereqn}
        \end{figure}

        Starting from the trivial equation $$O(G_A)O(G_C)/O(G_A)O(G_C)=O(G_N)O(G_S)/O(G_N)O(G_S),$$ if we cancel out from the denominator every outer face label  that contributes exactly once to both  products $O(G_A)O(G_C)$ and $O(G_N)O(G_S)$, we will arrive at the equation:\begin{equation}\label{eqn:outer-eqn}
            O(G_A)O(G_C)/ O_A O_C= O(G_N)O(G_S)/O_{N/S}O_{NSC}O_{NS},
        \end{equation}
        where we denote by $O_A $  the product of labels of faces that are outer to $G_A$ but to neither $G_N$ nor $G_S$ (part yellow in Figure \ref{fig:venn-outereqn}),
        $O_C$ the face label that is outer to $G_C$ but neither $G_N$ nor $G_S$ (It corresponds to part $C$ in dashed circle, which only happens when $G_k$ is just a square, see below the convention). $O_{N/S}$  the product of labels of faces that are outer to either $G_N$ or $G_S$ but to neither $G_A$ nor $G_C$ (part grey), $O_{NSC}$  the product of labels of faces that are outer to $G_N$, $G_S$ and $G_C$ (part purple), and $O_{NS}$  the \emph{square} product of labels of faces that are outer to both $G_N$ and $G_S$ but to neither $G_A$ nor $G_C$ (part blue). 
        
        When the relevant pinecone $G_*$ is an empty graph, it is  declared to have as outer face the cell $(0,1)$ (whether or not it is a face) when $*=C$,  the cell $(1,0)$ when $*=N$, the cell $(0,1)$ when $*=S$, the face to the left of the central strip of $G$ when $*=W$, and the face containing cell $(-1,0)$ when $*=E$. When the graph $G_*=G_n^{(r,s,N)}$ is empty, by the construction of the brane tiling, the face associated to $G_*$ is indeed labeled as $n$.
    \end{itemize}  

    \begin{rem}
    For any of the open faces
	$F \in  \overline{G}\bs G$, with label $i$, we have $m(F) = x_i$. 
    \end{rem}
    \begin{proof}
    By the definition of outer face labels, the only way the contrary holds is to have some face that has at least 3 edges incident to $G$. Given the shape of the brane tiling and the fact that faces are squares or rectangles, that could happen only when there's a rectangle placed on the corner. However, as explained in \cite[Section 2.3]{ProppMelouWest2009}, on that corner, by considering the parity, we must have a square. (See part (3) of their definition of a standard pinecone.) 
    \end{proof}
    
    Since each term $cm(G)x(M)$ is equal to $I(G)O(G)x(M)$,  we expect that the above three equations multiply together and the extra terms cancel out.
    In other words, we need to prove that $$\prod_{F\in G_A\bs (G_N\cup G_S)} m(F)\cdot \prod_{e\in S_1} x(e) \cdot \frac{1}{ O_{N/S}O_{NSC}O_{NS}} =\frac{1}{O_AO_C} .$$ Since $x_e=\frac{1}{x_ix_j}$ we would rather show that \begin{equation}\label{eqn:cancellation-face-labels-NS}
        \prod_{F\in G_A\bs (G_N\cup G_S)} m(F) \cdot O_AO_C=\prod_{e\in S_1} \frac{1}{x(e)}\cdot O_{N/S}O_{NSC}O_{NS}.
    \end{equation}
    
    \begin{rem}\label{rem:correspondence}
        We will prove this equation using the correspondence that the face labels of the two (inner or outer) faces sharing one edge $e\in (M_A \uplus M_C)\bs (M_N\uplus M_S)$ correspond to the (inverse) edge label $\frac{1}{x(e)}$. In this correspondence, we will regard an \emph{inner} rectangle $F$ labeled $i$ as providing \emph{two} face labels $x_i$, since we have $m(F)=x_i^2$.
    \end{rem}

    We first deal with the case where the superimposition decomposes as $M_N\oplus M_S$. In this case, the faces of $ G_A\bs (G_N\cup G_S)$ are all contained the central strip $H$.
    
    When $G_C=G_N\cap G_S$ is nonempty, by Proposition 4 of  \cite{ProppMelouWest2009}, the faces are: square + multiple rectangles+ square on the left / a rectangle or two squares on the right. 
    
    The right rectangle can be split into two squares. One gives an outer face label, the other, together with the outer face on its right, offsets the edge label of the rightmost edge which is included in $S_1=(M_A \uplus M_C)\bs (M_N\uplus M_S)$

    To treat the left strip we need the following lemma.

    \begin{lemma}\label{lem:outer-pattern-NS}
        Suppose row 0 occupies columns $[0,b]$, row 1 occupies columns $[1,b_+]$ and row -1 occupies columns $[1,b_-]$. Then face $(1,b_++1)$ is a square, and  row 1 contains only rectangles between columns $b_++2$ and $b$. Similarly, $(-1,b_-+1)$ is a square, and row -1 contains only rectangles between columns $b_-+2$ and $b$.
    \end{lemma}
    \begin{proof}
        We prove the statement on row 1. Place $G$ inside some bigger pinecone $\tilde{G}$ such that $G\cong \tilde{G}_S$. It follows from Proposition 4 of  \cite{ProppMelouWest2009} that the faces of row 0 of $\tilde{G}$ are: square + multiple rectangles + square on the left of the center of $G$ / a rectangle on the right. Now this row 0 of $\tilde{G}$ is row 1 of our pinecone which is $\tilde{G}_S$, and the left of the center of $G$ is the left of the boundary of row 1 in $\tilde{G}_S$. Since the center (row 0) is always more extended to the left than row 1, the lemma follows. Similar proof follows for row $-1$ once we place $G$ as $\tilde{G}_N$.
    \end{proof}
    For two squares on the left strip, the square on the right serves as an outer face of $G_N$, $G_S$ and $G_C$; the label of the other square on the left, together with the label of the outer face on its left, correspond to the leftmost edge label. For the other rectangles on the left strip, they can be incident to either $G_N$ or $G_S$ or both or neither.
    \begin{itemize}
        \item If it is incident to $G_N$ or  $G_S$, then it contains one forced edge in $(M_A \uplus M_C)\bs (M_N\uplus M_S)$. One $x$ of its label $x^2$ serves as the outer face label of $G_N$ or  $G_S$, the other $x$, together with the face label of the outer rectangle that is below or above it, corresponds to the edge label it contains;
        \item If it is incident to both $G_N$ and $G_S$, then $x^2$ serves as the product of outer face labels for both $G_N$ and  $G_S$;
        \item If it is incident to neither of $G_N$ or  $G_S$, then it contains two forced edges. Its label $x^2$, together with the face labels of two outer rectangles that sandwich it, correspond to the two edge labels. 
    \end{itemize}  
     By the correspondence \ref{rem:correspondence}, it remains to show that all outer faces of $G_A$ that contribute to $O_A$ contain exactly one forced edge. This is due to the above Lemma \ref{lem:outer-pattern-NS}, since every other edge on the top or bottom boundary is forced. Collecting these facts together, we arrive at Equation (\ref{eqn:cancellation-face-labels-NS}). We  illustrate the proof below by using the term $x_{16}$ in Example \ref{Ex:237}.

    \begin{figure}[H]
        \centering
        \includegraphics[width=0.8\linewidth]{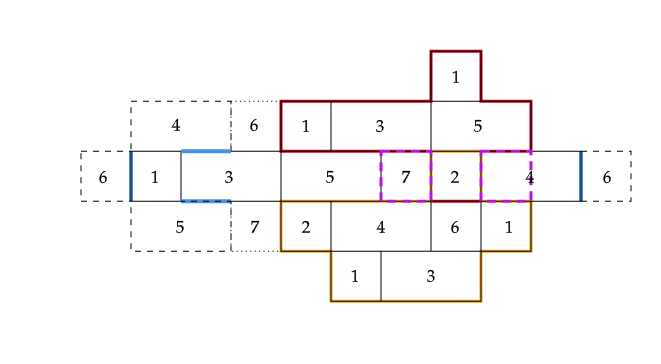}
        \caption{Here through Figure \ref{fig:WE-Eempty}, we illustrate the cases by the pinecones as in Example \ref{Ex:237}. Here the term is $x_{16}$. $G_N$ and $G_S$ are enclosed in the red and orange boxes respectively. The blue edges are those forced edges of $S_1$. The two purple inner squares contribute to $O_{NSC}$, all the outer faces contributing to $O_A$ are dashed, and the dotted squares illustrate Lemma \ref{lem:outer-pattern-NS}.}
        \label{fig:NS-full}
    \end{figure}
When $G_N\cap G_S$ is empty, the $G\bs(G_N\cup G_S)$ is the whole central strip. In this case, by Proposition 4 of \cite{ProppMelouWest2009}, if row 0 occupies columns $[0,b]$, then $(0,b)$ is a square, and there are only rectangles between column 2 and $b-1$ on row 0.

 If $G_N\cap G_S$ is empty but neither $G_N$ nor $G_S$ is empty, then after the above process, we are left with the cell $(0,1)$ on the central strip which is both an inner face and an outer face of $G_N,G_S$ and $G_C$, so Equation (\ref{eqn:cancellation-face-labels-NS}) still holds. See an illustration below using the term $x_{14}$ in Example \ref{Ex:237}.
 
\begin{figure}[H]
    \centering
    \includegraphics[width=0.8\linewidth]{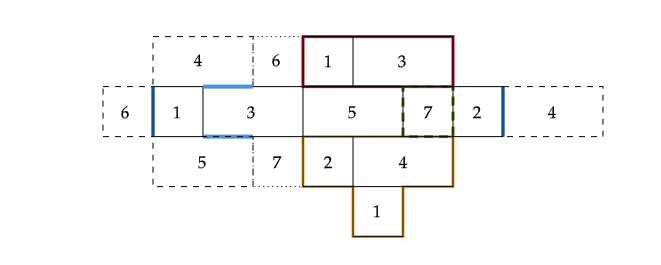}
    \caption{This corresponds to term $x_{14}$. Though $G_C$ is empty in this case, its corresponding face label is $x_7$ as in the Gale-Robinson sequence, which is the label of the green dashed square. So it is both an inner face and an outer face of $G_N,G_S$ and $G_C$, which is balanced on the two sides of Equation (\ref{eqn:cancellation-face-labels-NS}.)}
    \label{fig:NS-Cempty}
\end{figure}

  If $G_N$ is empty, then the forced edges are: $(0,0)-(-1,0)$, $(0,2k+1)-(-1,2k+1)$ as well as those edges of form $(0,2t-1)-(0,2t)$ and $(-1,2t'-1)-(-1,2t')$ for $0<t\le k$ and $l<t'\le k$ , where $l\ge1$ is defined so that $(-1,2t'-1)-(-1,2t')$ lie outside $G_S$. Using Lemma \ref{lem:outer-pattern-WE} below, we associate to each edge $(0,2t-1)-(0,2t)$ with the inner cell $(0,2t-1)$ and the outer face $(1,2t-1)$, edge $(0,0)-(-1,0)$ with inner cell $(0,0)$ and outer face $(0,-1)$, edge $(0,2k+1)-(-1,2k+1)$ with inner square $(0,2k)$ and outer face $(0,2k+1)$, and $(-1,2t'-1)-(-1,2t')$ with inner cell $(0,2t'-2)$ and outer face $(-1,2t'-1)$. Then all the terms in $O_A$ have been associated. Equation (\ref{eqn:cancellation-face-labels-NS}) holds once we notice that $(1,0)$ nor $(0,1)$ contribute to any outer term in the equation, since $(1,0)$ is outer to $G_A$ and $G_N$, and $(0,1)$ is outer to $G_C$ and $G_S$. See an illustration below using the term $x_{11}$ in Example \ref{Ex:237}.
\begin{figure}[H]
    \centering
    \includegraphics[width=0.8\linewidth]{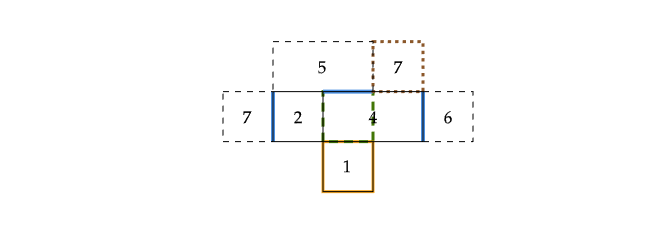}
    \caption{
    Seemingly only outer to $G_A$, the brown dotted square has the face label $x_7$, which is the outer face label of $G_N$ by our convention. Similarly, seemingly only outer to $G_S$, the green dashed square is also by convention outer to $G_C$. Hence in this pinecone corresponding to $x_{11}$, our Equation (\ref{eqn:cancellation-face-labels-NS}) essentially records the correspondence \ref{rem:correspondence} for three blue edges.
    }
    \label{fig:NS-Nempty}
\end{figure}
    If both $G_N$ and $G_S$ are empty then $G_A$ has only one row. In this case the forced edges are: $(0,0)-(-1,0)$, $(0,2k+1)-(-1,2k+1)$ as well as those edges of form $(0,2t-1)-(0,2t)$ and $(-1,2t-1)-(-1,2t)$ for $0<t\le k$. We can still associate edges to inner cells and outer faces like above. Equation (\ref{eqn:cancellation-face-labels-NS}) holds once we notice that the faces $(\pm 1,0)$ do not contribute to any outer term in the equation, since $(1,0)$ is outer to $G_A$ and $G_N$, and $(-1,0)$ is outer to $G_A$ and $G_S$, and $X=x_k$ is a face label for face $(0,1)$ in this case. See an illustration below using the term $x_{10}$ in Example \ref{Ex:237}.
\begin{figure}[H]
    \centering
    \includegraphics[width=0.8\linewidth]{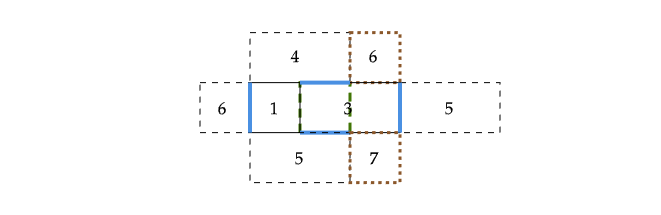}
    \caption{This corresponds to term $x_{10}$. The explanation of two brown dotted squares is similar to the one in Figure \ref{fig:NS-Nempty}. $x_3$ appears 3 times in edge labels, 2 times in inner face labels, and 1 time as $O_C$, so it is balanced on the two sides of Equation (\ref{eqn:cancellation-face-labels-NS}).}
    \label{fig:NSempty}
\end{figure}

    Then we treat the case where $M_A\oplus M_C\approx M_W\oplus M_E$.
    This time, the equations we have are:
    \begin{itemize}
        \item Inner face labels satisfy $$I(G_W)I(G_E)\cdot \prod_{F\in G_A\bs (G_W\cup G_E)} m(F) =I(G_C)I(G_A)\,.$$ 
        \item The set $S_2=(M_A \uplus M_C)\bs (M_W\uplus M_E)$ consists of  the horizontal edges $(i,j)-(i,j+1)$ with $i>0$ and $i+j$ even, $(i,j)-(i,j+1)$ with $i<-1$ and $i+j$ odd that do not lie in $G_W\cup G_E$;  This set does not depend on the perfect matchings as well. The
        edge labels satisfy $$x(M_W)x(M_E)\cdot \prod_{e\in S_2} x(e)=x(M_A)x(M_C).$$
         
        \item Outer face labels satisfy $$O(G_A)O(G_C)/ O_AO_C = O(G_W)O(G_E)/O_{W/E}O_{WEC}, O_{WE}=1.$$
    \end{itemize}

    The equation we want to show is \textbf{ \begin{equation}\label{eqn:cancellation-face-labels-WE}
        \prod_{F\in G_A\bs (G_W\cup G_E)} m(F) \cdot O_AO_C=\prod_{e\in S_2} \frac{1}{x(e)}\cdot O_{W/E}O_{WEC}.
    \end{equation}}
    
    We also have a lemma similar to Lemma \ref{lem:outer-pattern-NS} whose proof is different:

    \begin{lem}\label{lem:outer-pattern-WE}
        If the top row $a$ of $G$ occupies columns $[a,b]$, then the row $a+1$ above the top of $G$ contains only rectangles between columns $a+1$ and $b$\footnote{Row $a$ starts from column $a$, and $b-a$ is always even, by construction.}. Similar result holds for the row below $G$.
    \end{lem}
    \begin{rem}
        If the top row $a$ of $G$ consists of a single square, then there are no faces on row $a+1$ between  columns $a+1$ and $b$. In this case, the lemma holds vacuously.
    \end{rem}
    \begin{proof}
        Place $G$ inside some bigger pinecone $\tilde{G}$ such that $G\cong \tilde{G}_S$. By the result $\tilde{G}_S = core(\tilde{G} \cap A_S)$ in \cite{ProppMelouWest2009}, where $A$ is the Aztec diamond constructed for $\tilde{G}$ as in Remark \ref{rem:Aztec-diamond} and the description of core given in Remark \ref{rem:alternate-definition-core}, we know that if there were some square on the row $a+1$ between  columns $a+1$ and $b$, then it (together with those faces on the right of it, until column $a+1$) must also be contained in the core. This violates the definition of top row. The same proof applies to the row below $G$ if we place it as $\tilde{G}_N$.
    \end{proof}
    
    We notice that the edges in $S_2=(M_A \uplus M_C)\bs (M_W\uplus M_E)$ have an interlacing pattern, which means: If we split all the inner rectangles into squares, and distribute the two face labels of each inner rectangle to the two squares it split into, then it follows from the description of $S_2$ that
    \begin{enumerate}
        \item No face in $ G_A\bs (G_W\cup G_E)$ can contain two forced edges in $S_2$.
        \item If we denote by $L$ the set of all the faces in $ G_A\bs (G_W\cup G_E)$ that contain a forced edge in $S_2$, then $ G_A\bs (G_W\cup G_E \cup L)$ equals the set of faces that contribute to either $O_{W/E}$ or $O_{WEC}$. 
    \end{enumerate}
    Using the correspondence \ref{rem:correspondence}, it remains to  show that all outer faces of $G_A$ that contribute to $O_A$ contain exactly one forced edge, which similarly holds because of Lemma \ref{lem:outer-pattern-WE}. Collecting these facts together, we arrive at the Equation (\ref{eqn:cancellation-face-labels-WE}). We  illustrate the proof below by using again the term $x_{16}$ in Example \ref{Ex:237}.

    \begin{figure}[H]
        \centering
        \includegraphics[width=0.8\linewidth]{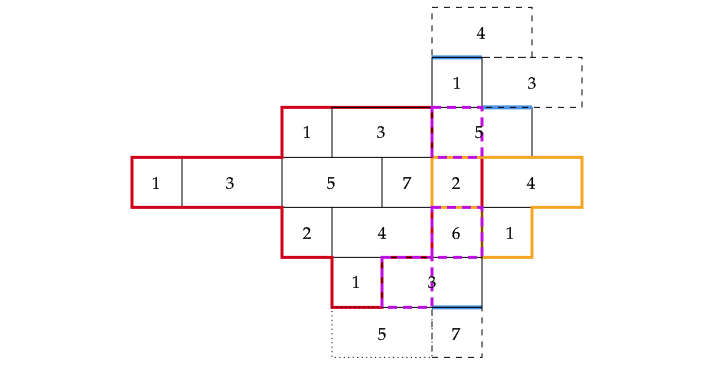}
        \caption{The legend of this figure is similar to the one of Figure \ref{fig:NS-full}, with the minor modification that red and orange box enclose $G_W$ and $G_E$, and the dotted  rectangle illustrates Lemma \ref{lem:outer-pattern-WE}. The two rectangles containing purple boxes are each split into two squares, one of which has a blue edge to form the correspondence \ref{rem:correspondence}.}
        \label{fig:WE-full}
    \end{figure}
    If $G_W\cap G_E$ is empty but neither $G_W$ nor $G_E$ are empty, then in this case, the face $(0,0)=G_E$ and $(0,1)$ are two squares by construction of brane tiling. After the above process, we are left with face $(0,1)$ on the central strip which is both an inner face and an outer face of $G_W,G_E$ and $G_C$, so Equation (\ref{eqn:cancellation-face-labels-WE}) still holds. See an illustration below using again the term $x_{14}$  in Example \ref{Ex:237}.
    \begin{figure}[H]
    \centering
    \includegraphics[width=0.8\linewidth]{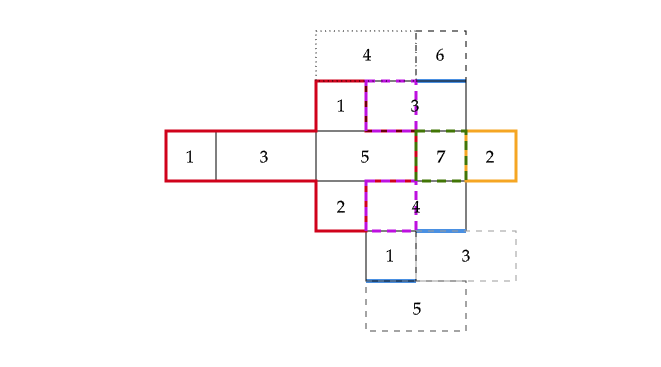}
    \caption{This corresponds to $x_{14}$ as well, and have similar legend and explanation as in Figure \ref{fig:NS-Cempty} and \ref{fig:WE-full}.}
    \label{fig:WE-Cempty}
\end{figure}
    If $G_E$ is empty but $G_W$ is nonempty, then  the central strip of $G\bs G_W$ consists of one rectangle $R=(0,0)\cup (0,1)$\footnote{This must be a rectangle. Otherwise, $G_E$ would be nonempty with face $(0,0)$, as in the case above Figure \ref{fig:WE-full}.}. Comparing with the case where $G_C$ is nonempty, there are two more forced edges $(0,0)-(0,1)$ and $(-1,0)-(-1,1)$ to consider. So after the above process we are left with two outer faces $(\pm1,0)$ of $G_A$, the inner rectangle $R$,  these two edges, one "outer face" $(0,1)$ to $C$, and one "outer face" $(0,-1)$ to $G_A$.  Equation (\ref{eqn:cancellation-face-labels-WE}) still holds once we notice that the faces $(0, \pm 1)$ do
not contribute to any outer term in the equation, since face $(0,-1)$ is outer to both $G_A$ and $G_E$, and face $(0,1)$ is outer to both $G_W$ and $G_C$, by our convention. See an illustration below using the term $x_{10}$ in Example \ref{Ex:237}.
    \begin{figure}[H]
    \centering
    \includegraphics[width=0.8\linewidth]{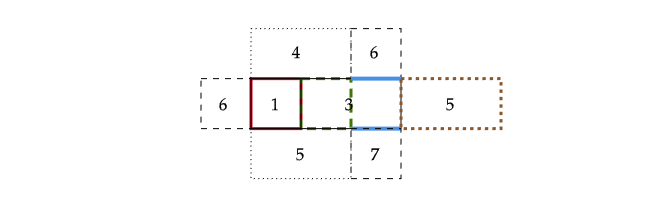}
    \caption{This corresponds to $x_{10}$. The legend and explanation are similar to those of Figure \ref{fig:NS-Nempty}.}
    \label{fig:WE-Eempty}
\end{figure}

    If both $G_W$ and $G_E$ are empty then $G_A$ itself is a square. In this case, the edges are $(0,0)-(0,1)$ and $(-1,0)-(-1,1)$, $O_A$ only consists of face labels of $(1,0)$ and $(-1,0)$ since face $(0,\pm1)$ are also outer to $G_W$ or $G_E$ respectively by our convention. Since in this case $X=x_k$, Equation (\ref{eqn:cancellation-face-labels-WE}) becomes trivial.
\end{proof}

\subsection{Kuo's condensation for $y$-variables}

\begin{defi}
	For a perfect matching $M$ of $G$, the \emph{log} height function $v_M$ defined on the set of faces of $G$, is given by: $v_M(F)={\# \text { cycles of } M^+ \text { enclosing the face } F}$. Here $F$ runs through all faces of $G$, and $M^+$ is the  $M\oplus M_-$ appearing in Definition \ref{Def:Height} of height function.
\end{defi}

Recall by Lemma \ref{lem:subpinecone-properties} we may identify $G_*$ as a subgraph of $G$. Therefore we may also regard $v_{M_*}$ as a function defined on a \emph{subset} of faces of $G$. We extend $v_{M_*}$ by setting them as zero on the other faces to get a function defined on all faces of $G$, by abusing of notation, still denoted by $v_{M_*}$. We are going to establish Proposition \ref{Prop:superpositions} via the following:

\begin{lem}\label{lem:superimposition-difference}
	   Denote by  $1_H$  the characteristic function whose values for  central strip faces are $1$, and other faces are $0$. 

           If $M_A\oplus M_C\approx M_N\oplus M_S$, then \begin{equation}
               v_{M_N}+v_{M_S}+1_H=v_{M_A}+v_{M_C}.
           \end{equation}
           
         If $M_A\oplus M_C\approx M_W\oplus M_E$, then 
           \begin{equation}
               v_{M_W}+v_{M_E}=v_{M_A}+v_{M_C}.
           \end{equation}
     \end{lem}
\begin{proof}[Proof of Proposition \ref{Prop:superpositions} (assuming Lemma \ref{lem:superimposition-difference})]
	Each term in $x_nx_{n-N}$ is given by \\$cm(G)cm(G_C)x(M_A)x(M_C)y(M_A)y(M_C)$. If $M_A\oplus M_C\approx M_W\oplus M_E$, then  we have $$cm(G)cm(G_C)x(M_A)x(M_C)=cm(G_W)cm(G_E)x(M_W)x(M_E)$$ as shown in the proof of Speyer's Theorem \ref{thm:trivial-coefficients}, 
	and $y(M_A)y(M_C)=y(M_W)y(M_E)$ by the above lemma. If $M_A\oplus M_C\approx M_N\oplus M_S$, then $$cm(G)cm(G_C)x(M_A)x(M_C)=cm(G_N)cm(G_S)x(M_N)x(M_S)$$ by Speyer's Theorem as well, and $$y(M_A)y(M_C)=\prod_i y_i^{d(n-N-i, r, N-r)}y(M_W)y(M_E)$$ by the above lemma.
	The proposition thus follows since Kuo's condensation is a bijection.
\end{proof}
The proof of Lemma \ref{lem:superimposition-difference} goes by induction. We set the stage below for doing the induction.

\subsection{The poset structure on perfect matchings}\label{subsec:poset-structure-perfect-matchings}
\begin{thm}[\cite{Propp02}]
	The set of perfect matchings of $G$ form a distributive lattice.
\end{thm} 
We also recall other necessary results.

\begin{lem}
	If $v_B-v_A\ge 0$ then $A\le B$ in the poset.
\end{lem} 
It is therefore natural to ask what is the meaning of cover relations in this poset. 

\begin{lem}\label{lem:cover-relation}
	$A\lessdot B$ if and only if there exists exactly one face $F$ which is enclosed not by $A^+$ but by $B^+$. In terms of log height functions, we have $v_B-v_A=1_F$. To emphasize this face we may also denote $A\lessdot_FB$.
\end{lem}

There is also a twisting operation on the set of perfect matchings. 
It is easy to see that there are two different perfect matchings on a square (and on a rectangle).
Consider a perfect matching $M$. If $M$ contains a perfect matching of a face $F$, then we call $F$ twistable (for $M$), and the twisting $\tau_F(M)$ is given by replacing the perfect matching of $F$ (which is a square or rectangle) with the other one.

\begin{lem}\label{lem:twist-properties}
	The followings hold:
	\begin{enumerate}
		\item $\tau_F(\tau_F(M))=M.$
		\item The twisting operation is transitive on the set of perfect matchings.
		\item Exactly one of the following holds: 
		$M\lessdot_F \tau_F(M)$, denoted by  $\tau_F(M)=\tau_F^+(M)$, or $\tau_F(M)\lessdot_F M$, denoted by  $\tau_F(M)=\tau_F^-(M)$.
	\end{enumerate}
\end{lem}
One can further decide the sign of $\tau_F$. For this, we need to invoke a parity argument. Place the pinecone $G$ on the square grid $\ZZ^2$ such that the upper vertex of the leftmost edge (upper-left vertex) is placed at $(0,0)$. Call a face $F$ of $G$ even if $i+j$ is even, where $(i,j)$ is the coordinate  of the upper-left vertex of $F$ . Otherwise call it odd.
\begin{lem}\label{lem:above-or-below}
	$\tau_F^+$ is given by the following table
	
\begin{table}[H]
	\centering
	\begin{tabular}{|l|l|l|}
		\hline
	$F$ is & Square & Rectangle \\ \hline
		Odd                         & \begin{tikzpicture}
		\draw (0,0)--(1,0);\draw[red,line width=2pt] (1,0)--(1,1);\draw (1,1)--(0,1);\draw[red,line width=2pt] (0,1)--(0,0);\draw[->>](1.25,0.5)--(1.75,0.5);\draw[red,line width=2pt] (2,0)--(3,0);\draw (3,0)--(3,1);\draw[red,line width=2pt] (3,1)--(2,1);\draw (2,1)--(2,0);
		\end{tikzpicture}      &      \begin{tikzpicture}
		\draw (0,0)--(1,0);\draw[red,line width=2pt] (1,0)--(2,0);\draw (2,0)--(2,1);\draw[red,line width=2pt] (2,1)--(1,1);\draw (1,1)--(0,1);\draw[red,line width=2pt](0,1)--(0,0);\draw[->>](2.25,0.5)--(2.75,0.5);\draw[red,line width=2pt] (3,0)--(4,0);\draw (4,0)--(5,0);\draw[red,line width=2pt] (5,0)--(5,1);\draw (5,1)--(4,1);\draw[red,line width=2pt] (4,1)--(3,1);\draw(3,1)--(03,0);
		\end{tikzpicture}
		   
		     \\ \hline
		Even                        & 
		\begin{tikzpicture}
			\draw[red,line width=2pt] (0,0)--(1,0);\draw (1,0)--(1,1);\draw[red,line width=2pt] (1,1)--(0,1);\draw (0,1)--(0,0);\draw[->>](1.25,0.5)--(1.75,0.5);\draw (2,0)--(3,0);\draw[red,line width=2pt] (3,0)--(3,1);\draw (3,1)--(2,1);\draw[red,line width=2pt] (2,1)--(2,0);
		\end{tikzpicture}   
		       &   
		       \begin{tikzpicture}
		       	\draw[red,line width=2pt] (0,0)--(1,0);\draw (1,0)--(2,0);\draw[red,line width=2pt] (2,0)--(2,1);\draw (2,1)--(1,1);\draw[red,line width=2pt] (1,1)--(0,1);\draw(0,1)--(0,0);\draw[->>](2.25,0.5)--(2.75,0.5);\draw (3,0)--(4,0);\draw[red,line width=2pt] (4,0)--(5,0);\draw (5,0)--(5,1);\draw[red,line width=2pt] (5,1)--(4,1);\draw (4,1)--(3,1);\draw[red,line width=2pt](3,1)--(03,0);
		       \end{tikzpicture}
		       
		               \\ \hline
	\end{tabular}
\end{table}
	Here the bold red edges are chosen in the perfect matching.
	
	This construction is independent of the choice of matching $M$, which means the above statements hold for \emph{all} perfect matchings twistable at $F$.
\end{lem}
\subsection{Proof of Lemma \ref{lem:superimposition-difference}}
\begin{lem}
	Let   $M_N\in M(G_N)$ and $M_S\in M(G_S)$. If $M_N\oplus M_S\sim M_A\oplus M_C$ and $M_N\oplus M_S\sim M_A'\oplus M_C'$, then $$v_{M_A}+v_{M_C}= v_{M_A'}+v_{M_C'}.$$
\end{lem}
\begin{rem}
	This lemma amounts to saying that the specific choice involved for cycles as in Kuo's condensation does not affect the validity of Lemma \ref{lem:superimposition-difference}.
\end{rem}
\begin{proof}
	It follows from the procedure of Kuo's condensation that $ M_A\oplus M_C$ and $M_A'\oplus M_C'$ produce the same spanning subgraph of $G-\{a,b,c,d\}$, and there exist some cycles in the superimposition whose odd (say) edges belong to $M_A$ and $M_C'$, and even edges belong to $M_A'$ and $M_C$.
		
 By induction, we may assume that there exists exactly one such cycle $Z=A_1\cup A_2$, where $A_1\subset M_A$ and $M_C'$, and $A_2\subset M_C$ and $M_A'$. Denote all the edges enclosed by $Z$ which belong to $M_A$ as $B_1$, and similarly define $B_2$ for $M_C$.

	We first twist $M_A$ to $M_A^\circ$ such that $M_A^\circ$ contains $A_2$ and $B_2$. Reversing the twist makes $M_C$ to $M_C^\circ$ which contains $A_1$ and $B_1$. Then twist from $B_2$ to $B_1$ so that $M_A^\circ$ twists to $M_A'$ containing $A_2$ and $B_1$. The reverse of this twist turns $M_C^\circ$ to $M_C'$ which contains $A_1$ and $B_2$.
	
	In more detail, by Lemma \ref{lem:twist-properties}, there exist a sequence of twists such that $M_A=\tau_{F_k}^{\ve_k}\circ \cdots\circ \tau_{F_1}^{\ve_1} (M_A^\circ)$.  Consequently, $M_C=\tau_{F_1}^{-\ve_1} \circ \cdots\circ \tau_{F_k}^{-\ve_k}  (M_C^\circ)$.  Hence by Lemma \ref{lem:cover-relation} we have \begin{gather*}
		v_{M_A}=v_{M_A^\circ}+\sum \ve_i 1_{F_i},\\
		v_{M_C}=v_{M_C^\circ}-\sum \ve_i1_{F_i}
	\end{gather*}
	which gives $v_{M_A}+v_{M_C}=v_{M_A^\circ}+v_{M_C^\circ}$. The same reasoning shows $v_{M_A^\circ}+v_{M_C^\circ}=v_{M_A'}+v_{M_C'}$.
\end{proof}
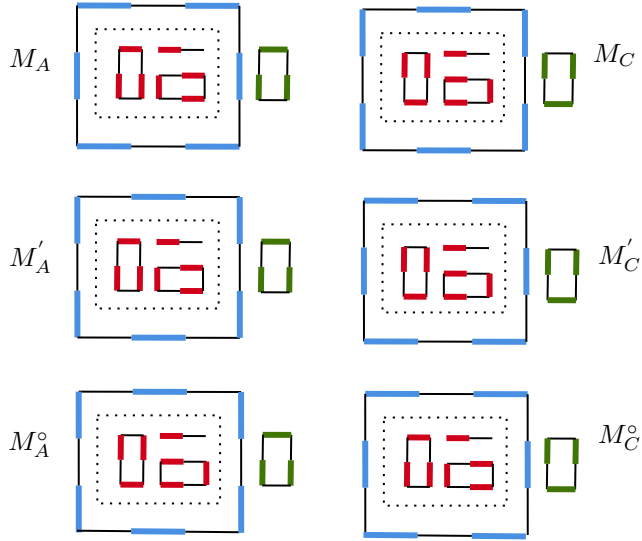
\begin{figure}[H]
	\centering
    
\tikzset{every picture/.style={line width=0.75pt}} 

\begin{tikzpicture}[x=0.75pt,y=0.75pt,yscale=-0.5,xscale=0.5]

\draw [color={rgb, 255:red, 74; green, 144; blue, 226 }  ,draw opacity=1 ][line width=2.25]    (94.23,43.75) -- (148.84,43.89) ;
\draw [color={rgb, 255:red, 74; green, 144; blue, 226 }  ,draw opacity=1 ][line width=2.25]    (203.45,44.02) -- (258.06,44.15) ;
\draw    (148.84,43.89) -- (203.45,44.02) ;
\draw    (94.23,43.75) -- (94.27,91.13) ;
\draw    (94.3,138.51) -- (94.34,185.89) ;
\draw [color={rgb, 255:red, 74; green, 144; blue, 226 }  ,draw opacity=1 ][line width=2.25]    (94.27,91.13) -- (94.3,138.51) ;
\draw [color={rgb, 255:red, 74; green, 144; blue, 226 }  ,draw opacity=1 ][line width=2.25]    (94.34,185.89) -- (148.95,186.03) ;
\draw [color={rgb, 255:red, 74; green, 144; blue, 226 }  ,draw opacity=1 ][line width=2.25]    (203.56,186.16) -- (258.17,186.29) ;
\draw    (148.95,186.03) -- (203.56,186.16) ;
\draw    (258.13,138.91) -- (258.17,186.29) ;
\draw [color={rgb, 255:red, 74; green, 144; blue, 226 }  ,draw opacity=1 ][line width=2.25]    (258.1,91.53) -- (258.13,138.91) ;
\draw    (258.06,44.15) -- (258.1,91.53) ;
\draw [color={rgb, 255:red, 0; green, 0; blue, 0 }  ,draw opacity=1 ][line width=0.75]    (159.74,87.89) -- (159.5,112.9) ;
\draw [color={rgb, 255:red, 208; green, 2; blue, 27 }  ,draw opacity=1 ][line width=2.25]    (159.5,112.9) -- (159.25,137.91) ;
\draw [color={rgb, 255:red, 208; green, 2; blue, 27 }  ,draw opacity=1 ][line width=2.25]    (159.74,87.89) -- (136.57,87.76) ;
\draw [color={rgb, 255:red, 208; green, 2; blue, 27 }  ,draw opacity=1 ][line width=2.25]    (136.57,112.61) -- (136.32,137.62) ;
\draw [color={rgb, 255:red, 0; green, 0; blue, 0 }  ,draw opacity=1 ][line width=0.75]    (136.82,87.6) -- (136.57,112.61) ;
\draw [color={rgb, 255:red, 0; green, 0; blue, 0 }  ,draw opacity=1 ][line width=0.75]    (159.5,137.75) -- (136.32,137.62) ;
\draw [color={rgb, 255:red, 0; green, 0; blue, 0 }  ,draw opacity=1 ][line width=0.75]    (176.86,114.38) -- (200,114.47) ;
\draw [color={rgb, 255:red, 208; green, 2; blue, 27 }  ,draw opacity=1 ][line width=2.25]    (200,114.47) -- (223.14,114.57) ;
\draw [color={rgb, 255:red, 208; green, 2; blue, 27 }  ,draw opacity=1 ][line width=2.25]    (176.86,114.38) -- (176.89,138.05) ;
\draw [color={rgb, 255:red, 208; green, 2; blue, 27 }  ,draw opacity=1 ][line width=2.25]    (199.89,137.89) -- (223.03,137.99) ;
\draw [color={rgb, 255:red, 0; green, 0; blue, 0 }  ,draw opacity=1 ][line width=0.75]    (176.74,137.8) -- (199.89,137.89) ;
\draw [color={rgb, 255:red, 0; green, 0; blue, 0 }  ,draw opacity=1 ][line width=0.75]    (222.99,114.31) -- (223.03,137.99) ;
\draw [color={rgb, 255:red, 0; green, 0; blue, 0 }  ,draw opacity=1 ][line width=0.75]    (306.66,88.22) -- (306.31,113.87) ;
\draw [color={rgb, 255:red, 65; green, 117; blue, 5 }  ,draw opacity=1 ][line width=2.25]    (306.31,113.87) -- (305.97,139.53) ;
\draw [color={rgb, 255:red, 65; green, 117; blue, 5 }  ,draw opacity=1 ][line width=2.25]    (306.66,88.22) -- (277.83,88.12) ;
\draw [color={rgb, 255:red, 65; green, 117; blue, 5 }  ,draw opacity=1 ][line width=2.25]    (277.79,113.61) -- (277.44,139.26) ;
\draw [color={rgb, 255:red, 0; green, 0; blue, 0 }  ,draw opacity=1 ][line width=0.75]    (278.14,87.95) -- (277.79,113.61) ;
\draw [color={rgb, 255:red, 0; green, 0; blue, 0 }  ,draw opacity=1 ][line width=0.75]    (306.27,139.36) -- (277.44,139.26) ;
\draw  [dash pattern={on 0.84pt off 2.51pt}] (113.07,67.71) -- (236.88,67.71) -- (236.85,156.63) -- (113.04,156.63) -- cycle ;
\draw [color={rgb, 255:red, 208; green, 2; blue, 27 }  ,draw opacity=1 ][line width=2.25]    (176.14,88.47) -- (199.28,88.56) ;
\draw [color={rgb, 255:red, 0; green, 0; blue, 0 }  ,draw opacity=1 ][line width=0.75]    (199.28,88.56) -- (222.43,88.64) ;
\draw [color={rgb, 255:red, 0; green, 0; blue, 0 }  ,draw opacity=1 ]   (382.27,47.74) -- (436.88,47.87) ;
\draw [color={rgb, 255:red, 0; green, 0; blue, 0 }  ,draw opacity=1 ]   (491.49,48.01) -- (546.1,48.14) ;
\draw [color={rgb, 255:red, 74; green, 144; blue, 226 }  ,draw opacity=1 ][line width=2.25]    (436.88,47.87) -- (491.49,48.01) ;
\draw [color={rgb, 255:red, 74; green, 144; blue, 226 }  ,draw opacity=1 ][line width=2.25]    (382.27,47.74) -- (382.3,95.12) ;
\draw [color={rgb, 255:red, 74; green, 144; blue, 226 }  ,draw opacity=1 ][line width=2.25]    (382.34,142.5) -- (382.37,189.88) ;
\draw [color={rgb, 255:red, 0; green, 0; blue, 0 }  ,draw opacity=1 ]   (382.3,95.12) -- (382.34,142.5) ;
\draw [color={rgb, 255:red, 0; green, 0; blue, 0 }  ,draw opacity=1 ]   (382.37,189.88) -- (436.98,190.02) ;
\draw [color={rgb, 255:red, 0; green, 0; blue, 0 }  ,draw opacity=1 ]   (491.59,190.15) -- (546.2,190.28) ;
\draw [color={rgb, 255:red, 74; green, 144; blue, 226 }  ,draw opacity=1 ][line width=2.25]    (436.98,190.02) -- (491.59,190.15) ;
\draw [color={rgb, 255:red, 74; green, 144; blue, 226 }  ,draw opacity=1 ][line width=2.25]    (546.17,142.9) -- (546.2,190.28) ;
\draw [color={rgb, 255:red, 0; green, 0; blue, 0 }  ,draw opacity=1 ]   (546.13,95.52) -- (546.17,142.9) ;
\draw [color={rgb, 255:red, 74; green, 144; blue, 226 }  ,draw opacity=1 ][line width=2.25]    (546.1,48.14) -- (546.13,95.52) ;
\draw [color={rgb, 255:red, 208; green, 2; blue, 27 }  ,draw opacity=1 ][line width=2.25]    (447.78,91.88) -- (447.53,116.89) ;
\draw [color={rgb, 255:red, 0; green, 0; blue, 0 }  ,draw opacity=1 ]   (447.53,116.89) -- (447.28,141.9) ;
\draw [color={rgb, 255:red, 0; green, 0; blue, 0 }  ,draw opacity=1 ]   (447.78,91.88) -- (424.61,91.75) ;
\draw [color={rgb, 255:red, 0; green, 0; blue, 0 }  ,draw opacity=1 ]   (424.6,116.6) -- (424.36,141.61) ;
\draw [color={rgb, 255:red, 208; green, 2; blue, 27 }  ,draw opacity=1 ][line width=2.25]    (424.85,91.58) -- (424.6,116.6) ;
\draw [color={rgb, 255:red, 208; green, 2; blue, 27 }  ,draw opacity=1 ][line width=2.25]    (447.53,141.74) -- (424.36,141.61) ;
\draw [color={rgb, 255:red, 208; green, 2; blue, 27 }  ,draw opacity=1 ][line width=2.25]    (464.89,118.36) -- (488.03,118.46) ;
\draw [color={rgb, 255:red, 0; green, 0; blue, 0 }  ,draw opacity=1 ]   (488.03,118.46) -- (511.18,118.56) ;
\draw [color={rgb, 255:red, 0; green, 0; blue, 0 }  ,draw opacity=1 ]   (464.89,118.36) -- (464.93,142.04) ;
\draw [color={rgb, 255:red, 0; green, 0; blue, 0 }  ,draw opacity=1 ]   (487.92,141.88) -- (511.07,141.98) ;
\draw [color={rgb, 255:red, 208; green, 2; blue, 27 }  ,draw opacity=1 ][line width=2.25]    (464.78,141.79) -- (487.92,141.88) ;
\draw [color={rgb, 255:red, 208; green, 2; blue, 27 }  ,draw opacity=1 ][line width=2.25]    (511.03,118.3) -- (511.07,141.98) ;
\draw [color={rgb, 255:red, 65; green, 117; blue, 5 }  ,draw opacity=1 ][line width=2.25]    (594.69,92.21) -- (594.35,117.86) ;
\draw [color={rgb, 255:red, 0; green, 0; blue, 0 }  ,draw opacity=1 ]   (594.35,117.86) -- (594,143.51) ;
\draw [color={rgb, 255:red, 0; green, 0; blue, 0 }  ,draw opacity=1 ]   (594.69,92.21) -- (565.86,92.11) ;
\draw [color={rgb, 255:red, 0; green, 0; blue, 0 }  ,draw opacity=1 ]   (565.82,117.6) -- (565.48,143.25) ;
\draw [color={rgb, 255:red, 65; green, 117; blue, 5 }  ,draw opacity=1 ][line width=2.25]    (566.17,91.94) -- (565.82,117.6) ;
\draw [color={rgb, 255:red, 65; green, 117; blue, 5 }  ,draw opacity=1 ][line width=2.25]    (594.31,143.35) -- (565.48,143.25) ;
\draw  [color={rgb, 255:red, 0; green, 0; blue, 0 }  ,draw opacity=1 ][dash pattern={on 0.84pt off 2.51pt}] (401.1,71.7) -- (524.92,71.7) -- (524.89,160.62) -- (401.07,160.62) -- cycle ;
\draw [color={rgb, 255:red, 208; green, 2; blue, 27 }  ,draw opacity=1 ][line width=2.25]    (464.17,92.46) -- (487.32,92.55) ;
\draw [color={rgb, 255:red, 0; green, 0; blue, 0 }  ,draw opacity=1 ]   (487.32,92.55) -- (510.46,92.63) ;
\draw [color={rgb, 255:red, 0; green, 0; blue, 0 }  ,draw opacity=1 ]   (94.58,236.75) -- (149.19,236.89) ;
\draw [color={rgb, 255:red, 0; green, 0; blue, 0 }  ,draw opacity=1 ]   (203.8,237.02) -- (258.41,237.15) ;
\draw [color={rgb, 255:red, 74; green, 144; blue, 226 }  ,draw opacity=1 ][line width=2.25]    (149.19,236.89) -- (203.8,237.02) ;
\draw [color={rgb, 255:red, 74; green, 144; blue, 226 }  ,draw opacity=1 ][line width=2.25]    (94.58,236.75) -- (94.61,284.14) ;
\draw [color={rgb, 255:red, 74; green, 144; blue, 226 }  ,draw opacity=1 ][line width=2.25]    (94.65,331.52) -- (94.68,378.9) ;
\draw [color={rgb, 255:red, 0; green, 0; blue, 0 }  ,draw opacity=1 ]   (94.61,284.14) -- (94.65,331.52) ;
\draw [color={rgb, 255:red, 0; green, 0; blue, 0 }  ,draw opacity=1 ]   (94.68,378.9) -- (149.29,379.03) ;
\draw [color={rgb, 255:red, 0; green, 0; blue, 0 }  ,draw opacity=1 ]   (203.9,379.16) -- (258.51,379.29) ;
\draw [color={rgb, 255:red, 74; green, 144; blue, 226 }  ,draw opacity=1 ][line width=2.25]    (149.29,379.03) -- (203.9,379.16) ;
\draw [color={rgb, 255:red, 74; green, 144; blue, 226 }  ,draw opacity=1 ][line width=2.25]    (258.48,331.91) -- (258.51,379.29) ;
\draw [color={rgb, 255:red, 0; green, 0; blue, 0 }  ,draw opacity=1 ]   (258.44,284.53) -- (258.48,331.91) ;
\draw [color={rgb, 255:red, 74; green, 144; blue, 226 }  ,draw opacity=1 ][line width=2.25]    (258.41,237.15) -- (258.44,284.53) ;
\draw [color={rgb, 255:red, 65; green, 117; blue, 5 }  ,draw opacity=1 ][line width=2.25]    (598.08,290.19) -- (597.73,315.84) ;
\draw [color={rgb, 255:red, 0; green, 0; blue, 0 }  ,draw opacity=1 ]   (597.73,315.84) -- (597.38,341.5) ;
\draw [color={rgb, 255:red, 0; green, 0; blue, 0 }  ,draw opacity=1 ]   (598.08,290.19) -- (569.24,290.09) ;
\draw [color={rgb, 255:red, 0; green, 0; blue, 0 }  ,draw opacity=1 ]   (569.21,315.58) -- (568.86,341.23) ;
\draw [color={rgb, 255:red, 65; green, 117; blue, 5 }  ,draw opacity=1 ][line width=2.25]    (569.55,289.92) -- (569.21,315.58) ;
\draw [color={rgb, 255:red, 65; green, 117; blue, 5 }  ,draw opacity=1 ][line width=2.25]    (597.69,341.33) -- (568.86,341.23) ;
\draw  [color={rgb, 255:red, 0; green, 0; blue, 0 }  ,draw opacity=1 ][dash pattern={on 0.84pt off 2.51pt}] (113.41,260.71) -- (237.23,260.71) -- (237.2,349.63) -- (113.38,349.63) -- cycle ;
\draw [color={rgb, 255:red, 0; green, 0; blue, 0 }  ,draw opacity=1 ]   (95.95,433.55) -- (150.56,433.68) ;
\draw [color={rgb, 255:red, 0; green, 0; blue, 0 }  ,draw opacity=1 ]   (205.17,433.81) -- (259.78,433.95) ;
\draw [color={rgb, 255:red, 74; green, 144; blue, 226 }  ,draw opacity=1 ][line width=2.25]    (150.56,433.68) -- (205.17,433.81) ;
\draw [color={rgb, 255:red, 74; green, 144; blue, 226 }  ,draw opacity=1 ][line width=2.25]    (95.95,433.55) -- (95.99,480.93) ;
\draw [color={rgb, 255:red, 74; green, 144; blue, 226 }  ,draw opacity=1 ][line width=2.25]    (96.02,528.31) -- (96.06,575.69) ;
\draw [color={rgb, 255:red, 0; green, 0; blue, 0 }  ,draw opacity=1 ]   (95.99,480.93) -- (96.02,528.31) ;
\draw [color={rgb, 255:red, 0; green, 0; blue, 0 }  ,draw opacity=1 ]   (94.21,573.94) -- (148.82,574.45) ;
\draw [color={rgb, 255:red, 0; green, 0; blue, 0 }  ,draw opacity=1 ]   (203.43,574.96) -- (258.04,575.48) ;
\draw [color={rgb, 255:red, 74; green, 144; blue, 226 }  ,draw opacity=1 ][line width=2.25]    (148.82,574.45) -- (203.43,574.96) ;
\draw [color={rgb, 255:red, 74; green, 144; blue, 226 }  ,draw opacity=1 ][line width=2.25]    (259.85,528.71) -- (259.89,576.09) ;
\draw [color={rgb, 255:red, 0; green, 0; blue, 0 }  ,draw opacity=1 ]   (259.82,481.33) -- (259.85,528.71) ;
\draw [color={rgb, 255:red, 74; green, 144; blue, 226 }  ,draw opacity=1 ][line width=2.25]    (259.78,433.95) -- (259.82,481.33) ;
\draw [color={rgb, 255:red, 208; green, 2; blue, 27 }  ,draw opacity=1 ][line width=2.25]    (161.47,477.68) -- (161.22,502.7) ;
\draw [color={rgb, 255:red, 0; green, 0; blue, 0 }  ,draw opacity=1 ]   (161.22,502.7) -- (160.97,527.71) ;
\draw [color={rgb, 255:red, 0; green, 0; blue, 0 }  ,draw opacity=1 ]   (161.47,477.68) -- (138.29,477.55) ;
\draw [color={rgb, 255:red, 0; green, 0; blue, 0 }  ,draw opacity=1 ]   (138.29,502.4) -- (138.04,527.42) ;
\draw [color={rgb, 255:red, 208; green, 2; blue, 27 }  ,draw opacity=1 ][line width=2.25]    (138.54,477.39) -- (138.29,502.4) ;
\draw [color={rgb, 255:red, 208; green, 2; blue, 27 }  ,draw opacity=1 ][line width=2.25]    (161.22,527.55) -- (138.04,527.42) ;
\draw [color={rgb, 255:red, 208; green, 2; blue, 27 }  ,draw opacity=1 ][line width=2.25]    (178.58,504.17) -- (201.72,504.27) ;
\draw [color={rgb, 255:red, 0; green, 0; blue, 0 }  ,draw opacity=1 ]   (201.72,504.27) -- (224.86,504.36) ;
\draw [color={rgb, 255:red, 0; green, 0; blue, 0 }  ,draw opacity=1 ]   (178.58,504.17) -- (178.62,527.85) ;
\draw [color={rgb, 255:red, 0; green, 0; blue, 0 }  ,draw opacity=1 ]   (201.61,527.69) -- (224.75,527.79) ;
\draw [color={rgb, 255:red, 208; green, 2; blue, 27 }  ,draw opacity=1 ][line width=2.25]    (178.47,527.59) -- (201.61,527.69) ;
\draw [color={rgb, 255:red, 208; green, 2; blue, 27 }  ,draw opacity=1 ][line width=2.25]    (224.71,504.11) -- (224.75,527.79) ;
\draw [color={rgb, 255:red, 65; green, 117; blue, 5 }  ,draw opacity=1 ][line width=2.25]    (597.41,481) -- (597.06,506.65) ;
\draw [color={rgb, 255:red, 0; green, 0; blue, 0 }  ,draw opacity=1 ]   (597.06,506.65) -- (596.72,532.3) ;
\draw [color={rgb, 255:red, 0; green, 0; blue, 0 }  ,draw opacity=1 ]   (597.41,481) -- (568.58,480.9) ;
\draw [color={rgb, 255:red, 0; green, 0; blue, 0 }  ,draw opacity=1 ]   (568.54,506.39) -- (568.19,532.04) ;
\draw [color={rgb, 255:red, 65; green, 117; blue, 5 }  ,draw opacity=1 ][line width=2.25]    (568.88,480.73) -- (568.54,506.39) ;
\draw [color={rgb, 255:red, 65; green, 117; blue, 5 }  ,draw opacity=1 ][line width=2.25]    (597.02,532.14) -- (568.19,532.04) ;
\draw  [color={rgb, 255:red, 0; green, 0; blue, 0 }  ,draw opacity=1 ][dash pattern={on 0.84pt off 2.51pt}] (114.79,457.5) -- (238.6,457.5) -- (238.57,546.42) -- (114.76,546.42) -- cycle ;
\draw [color={rgb, 255:red, 208; green, 2; blue, 27 }  ,draw opacity=1 ][line width=2.25]    (177.86,478.27) -- (201,478.35) ;
\draw [color={rgb, 255:red, 0; green, 0; blue, 0 }  ,draw opacity=1 ]   (201,478.35) -- (224.15,478.44) ;
\draw [color={rgb, 255:red, 74; green, 144; blue, 226 }  ,draw opacity=1 ][line width=2.25]    (383.62,240.74) -- (438.23,240.87) ;
\draw [color={rgb, 255:red, 74; green, 144; blue, 226 }  ,draw opacity=1 ][line width=2.25]    (492.84,241) -- (547.45,241.14) ;
\draw    (438.23,240.87) -- (492.84,241) ;
\draw    (383.62,240.74) -- (383.65,288.12) ;
\draw    (383.68,335.5) -- (383.72,382.88) ;
\draw [color={rgb, 255:red, 74; green, 144; blue, 226 }  ,draw opacity=1 ][line width=2.25]    (383.65,288.12) -- (383.68,335.5) ;
\draw [color={rgb, 255:red, 74; green, 144; blue, 226 }  ,draw opacity=1 ][line width=2.25]    (383.72,382.88) -- (438.33,383.01) ;
\draw [color={rgb, 255:red, 74; green, 144; blue, 226 }  ,draw opacity=1 ][line width=2.25]    (492.94,383.14) -- (547.55,383.28) ;
\draw    (438.33,383.01) -- (492.94,383.14) ;
\draw    (547.51,335.9) -- (547.55,383.28) ;
\draw [color={rgb, 255:red, 74; green, 144; blue, 226 }  ,draw opacity=1 ][line width=2.25]    (547.48,288.52) -- (547.51,335.9) ;
\draw    (547.45,241.14) -- (547.48,288.52) ;
\draw [color={rgb, 255:red, 0; green, 0; blue, 0 }  ,draw opacity=1 ][line width=0.75]    (309.01,282.21) -- (308.67,307.86) ;
\draw [color={rgb, 255:red, 65; green, 117; blue, 5 }  ,draw opacity=1 ][line width=2.25]    (308.67,307.86) -- (308.32,333.51) ;
\draw [color={rgb, 255:red, 65; green, 117; blue, 5 }  ,draw opacity=1 ][line width=2.25]    (309.01,282.21) -- (280.18,282.11) ;
\draw [color={rgb, 255:red, 65; green, 117; blue, 5 }  ,draw opacity=1 ][line width=2.25]    (280.14,307.6) -- (279.8,333.25) ;
\draw [color={rgb, 255:red, 0; green, 0; blue, 0 }  ,draw opacity=1 ][line width=0.75]    (280.49,281.94) -- (280.14,307.6) ;
\draw [color={rgb, 255:red, 0; green, 0; blue, 0 }  ,draw opacity=1 ][line width=0.75]    (308.63,333.35) -- (279.8,333.25) ;
\draw  [dash pattern={on 0.84pt off 2.51pt}] (402.45,264.69) -- (526.26,264.69) -- (526.23,353.61) -- (402.42,353.61) -- cycle ;
\draw [color={rgb, 255:red, 74; green, 144; blue, 226 }  ,draw opacity=1 ][line width=2.25]    (384.98,435.77) -- (439.59,435.9) ;
\draw [color={rgb, 255:red, 74; green, 144; blue, 226 }  ,draw opacity=1 ][line width=2.25]    (494.2,436.03) -- (548.81,436.17) ;
\draw    (439.59,435.9) -- (494.2,436.03) ;
\draw    (384.98,435.77) -- (385.01,483.15) ;
\draw    (385.05,530.53) -- (385.08,577.91) ;
\draw [color={rgb, 255:red, 74; green, 144; blue, 226 }  ,draw opacity=1 ][line width=2.25]    (385.01,483.15) -- (385.05,530.53) ;
\draw [color={rgb, 255:red, 74; green, 144; blue, 226 }  ,draw opacity=1 ][line width=2.25]    (383.21,578.17) -- (437.82,578.69) ;
\draw [color={rgb, 255:red, 74; green, 144; blue, 226 }  ,draw opacity=1 ][line width=2.25]    (492.43,579.2) -- (547.04,579.71) ;
\draw    (437.82,578.69) -- (492.43,579.2) ;
\draw    (548.88,530.93) -- (548.91,578.31) ;
\draw [color={rgb, 255:red, 74; green, 144; blue, 226 }  ,draw opacity=1 ][line width=2.25]    (548.84,483.55) -- (548.88,530.93) ;
\draw    (548.81,436.17) -- (548.84,483.55) ;
\draw [color={rgb, 255:red, 0; green, 0; blue, 0 }  ,draw opacity=1 ][line width=0.75]    (450.49,479.9) -- (450.24,504.91) ;
\draw [color={rgb, 255:red, 208; green, 2; blue, 27 }  ,draw opacity=1 ][line width=2.25]    (450.24,504.91) -- (449.99,529.93) ;
\draw [color={rgb, 255:red, 208; green, 2; blue, 27 }  ,draw opacity=1 ][line width=2.25]    (450.49,479.9) -- (427.31,479.77) ;
\draw [color={rgb, 255:red, 208; green, 2; blue, 27 }  ,draw opacity=1 ][line width=2.25]    (427.31,504.62) -- (427.07,529.64) ;
\draw [color={rgb, 255:red, 0; green, 0; blue, 0 }  ,draw opacity=1 ][line width=0.75]    (427.56,479.61) -- (427.31,504.62) ;
\draw [color={rgb, 255:red, 0; green, 0; blue, 0 }  ,draw opacity=1 ][line width=0.75]    (450.24,529.77) -- (427.07,529.64) ;
\draw [color={rgb, 255:red, 0; green, 0; blue, 0 }  ,draw opacity=1 ][line width=0.75]    (467.6,506.39) -- (490.74,506.48) ;
\draw [color={rgb, 255:red, 208; green, 2; blue, 27 }  ,draw opacity=1 ][line width=2.25]    (490.74,506.48) -- (513.89,506.58) ;
\draw [color={rgb, 255:red, 208; green, 2; blue, 27 }  ,draw opacity=1 ][line width=2.25]    (467.6,506.39) -- (467.64,530.06) ;
\draw [color={rgb, 255:red, 208; green, 2; blue, 27 }  ,draw opacity=1 ][line width=2.25]    (490.63,529.91) -- (513.78,530) ;
\draw [color={rgb, 255:red, 0; green, 0; blue, 0 }  ,draw opacity=1 ][line width=0.75]    (467.49,529.81) -- (490.63,529.91) ;
\draw [color={rgb, 255:red, 0; green, 0; blue, 0 }  ,draw opacity=1 ][line width=0.75]    (513.74,506.33) -- (513.78,530) ;
\draw [color={rgb, 255:red, 0; green, 0; blue, 0 }  ,draw opacity=1 ][line width=0.75]    (310.37,477.23) -- (310.03,502.89) ;
\draw [color={rgb, 255:red, 65; green, 117; blue, 5 }  ,draw opacity=1 ][line width=2.25]    (310.03,502.89) -- (309.68,528.54) ;
\draw [color={rgb, 255:red, 65; green, 117; blue, 5 }  ,draw opacity=1 ][line width=2.25]    (310.37,477.23) -- (281.54,477.14) ;
\draw [color={rgb, 255:red, 65; green, 117; blue, 5 }  ,draw opacity=1 ][line width=2.25]    (281.5,502.62) -- (281.16,528.28) ;
\draw [color={rgb, 255:red, 0; green, 0; blue, 0 }  ,draw opacity=1 ][line width=0.75]    (281.85,476.97) -- (281.5,502.62) ;
\draw [color={rgb, 255:red, 0; green, 0; blue, 0 }  ,draw opacity=1 ][line width=0.75]    (309.99,528.38) -- (281.16,528.28) ;
\draw  [dash pattern={on 0.84pt off 2.51pt}] (403.81,459.72) -- (527.63,459.72) -- (527.59,548.64) -- (403.78,548.64) -- cycle ;
\draw [color={rgb, 255:red, 208; green, 2; blue, 27 }  ,draw opacity=1 ][line width=2.25]    (466.88,480.49) -- (490.03,480.57) ;
\draw [color={rgb, 255:red, 0; green, 0; blue, 0 }  ,draw opacity=1 ][line width=0.75]    (490.03,480.57) -- (513.17,480.65) ;
\draw [color={rgb, 255:red, 0; green, 0; blue, 0 }  ,draw opacity=1 ][line width=0.75]    (158.1,281.9) -- (157.85,306.92) ;
\draw [color={rgb, 255:red, 208; green, 2; blue, 27 }  ,draw opacity=1 ][line width=2.25]    (157.85,306.92) -- (157.6,331.93) ;
\draw [color={rgb, 255:red, 208; green, 2; blue, 27 }  ,draw opacity=1 ][line width=2.25]    (158.1,281.9) -- (134.92,281.77) ;
\draw [color={rgb, 255:red, 208; green, 2; blue, 27 }  ,draw opacity=1 ][line width=2.25]    (134.92,306.63) -- (134.68,331.64) ;
\draw [color={rgb, 255:red, 0; green, 0; blue, 0 }  ,draw opacity=1 ][line width=0.75]    (135.17,281.61) -- (134.92,306.63) ;
\draw [color={rgb, 255:red, 0; green, 0; blue, 0 }  ,draw opacity=1 ][line width=0.75]    (157.85,331.77) -- (134.68,331.64) ;
\draw [color={rgb, 255:red, 0; green, 0; blue, 0 }  ,draw opacity=1 ][line width=0.75]    (175.21,308.39) -- (198.35,308.49) ;
\draw [color={rgb, 255:red, 208; green, 2; blue, 27 }  ,draw opacity=1 ][line width=2.25]    (198.35,308.49) -- (221.5,308.58) ;
\draw [color={rgb, 255:red, 208; green, 2; blue, 27 }  ,draw opacity=1 ][line width=2.25]    (175.21,308.39) -- (175.25,332.07) ;
\draw [color={rgb, 255:red, 208; green, 2; blue, 27 }  ,draw opacity=1 ][line width=2.25]    (198.24,331.91) -- (221.39,332.01) ;
\draw [color={rgb, 255:red, 0; green, 0; blue, 0 }  ,draw opacity=1 ][line width=0.75]    (175.1,331.82) -- (198.24,331.91) ;
\draw [color={rgb, 255:red, 0; green, 0; blue, 0 }  ,draw opacity=1 ][line width=0.75]    (221.35,308.33) -- (221.39,332.01) ;
\draw [color={rgb, 255:red, 208; green, 2; blue, 27 }  ,draw opacity=1 ][line width=2.25]    (174.49,282.49) -- (197.64,282.57) ;
\draw [color={rgb, 255:red, 0; green, 0; blue, 0 }  ,draw opacity=1 ][line width=0.75]    (197.64,282.57) -- (220.78,282.66) ;
\draw [color={rgb, 255:red, 208; green, 2; blue, 27 }  ,draw opacity=1 ][line width=2.25]    (447.15,287.89) -- (446.9,312.9) ;
\draw [color={rgb, 255:red, 0; green, 0; blue, 0 }  ,draw opacity=1 ]   (446.9,312.9) -- (446.65,337.91) ;
\draw [color={rgb, 255:red, 0; green, 0; blue, 0 }  ,draw opacity=1 ]   (447.15,287.89) -- (423.97,287.75) ;
\draw [color={rgb, 255:red, 0; green, 0; blue, 0 }  ,draw opacity=1 ]   (423.97,312.61) -- (423.73,337.62) ;
\draw [color={rgb, 255:red, 208; green, 2; blue, 27 }  ,draw opacity=1 ][line width=2.25]    (424.22,287.59) -- (423.97,312.61) ;
\draw [color={rgb, 255:red, 208; green, 2; blue, 27 }  ,draw opacity=1 ][line width=2.25]    (446.9,337.75) -- (423.73,337.62) ;
\draw [color={rgb, 255:red, 208; green, 2; blue, 27 }  ,draw opacity=1 ][line width=2.25]    (464.26,314.37) -- (487.4,314.47) ;
\draw [color={rgb, 255:red, 0; green, 0; blue, 0 }  ,draw opacity=1 ]   (487.4,314.47) -- (510.55,314.56) ;
\draw [color={rgb, 255:red, 0; green, 0; blue, 0 }  ,draw opacity=1 ]   (464.26,314.37) -- (464.3,338.05) ;
\draw [color={rgb, 255:red, 0; green, 0; blue, 0 }  ,draw opacity=1 ]   (487.29,337.89) -- (510.44,337.99) ;
\draw [color={rgb, 255:red, 208; green, 2; blue, 27 }  ,draw opacity=1 ][line width=2.25]    (464.15,337.8) -- (487.29,337.89) ;
\draw [color={rgb, 255:red, 208; green, 2; blue, 27 }  ,draw opacity=1 ][line width=2.25]    (510.4,314.31) -- (510.44,337.99) ;
\draw [color={rgb, 255:red, 208; green, 2; blue, 27 }  ,draw opacity=1 ][line width=2.25]    (463.54,288.47) -- (486.68,288.56) ;
\draw [color={rgb, 255:red, 0; green, 0; blue, 0 }  ,draw opacity=1 ]   (486.68,288.56) -- (509.83,288.64) ;

\draw (23,84.4) node [anchor=north west][inner sep=0.75pt]    {$M_{A}$};
\draw (612,79.4) node [anchor=north west][inner sep=0.75pt]    {$M_{C}$};
\draw (23,278.4) node [anchor=north west][inner sep=0.75pt]    {$M_{A}^{'}$};
\draw (617,276.4) node [anchor=north west][inner sep=0.75pt]    {$M_{C}^{'}$};
\draw (22,471.22) node [anchor=north west][inner sep=0.75pt]    {$M_{A}^{\circ }$};
\draw (617,464.22) node [anchor=north west][inner sep=0.75pt]    {$M_{C}^{\circ }$};

\end{tikzpicture}

	\caption{An illustration of the above proof: $(M_A,M_C)$ and $(M_A',M_C')$ differ only because of distribution of edges on the outer cycle $Z$. The sets $B_1$ and $B_2$ in the proof correspond to the inner edges enclosed in the dashed box. We omit the edges that have an intersection with the dashed box. From $(M_A,M_C)$ to $(M_A^\circ,M_C^\circ)$ we ``switch" the blue and red edges. Then From $(M_A^\circ,M_C^\circ)$ to $(M_A',M_C')$ we switch the red edges only. The validity of switching is shown in the last paragraph of the proof.}
\end{figure}
\begin{lem}
	Suppose $M_N$ is twistable at $F$, then it cannot happen that each of the two paths $a\to b$ and $c\to d$ both contain one or more edge of $F$ lying in $M_N$. Similar statements hold if we switch $N$ to $W,S$ or $E$.
\end{lem}
\begin{proof}
  First suppose $F$ is a square. If $a\to b$ and $c\to d$ both contain some edge of $F$ which lies in $M_N$, then $M_N\oplus M_S$, being an element of $M(G_N)\times M(G_S)$, produces paths $a\to b$ and $c\to d$. The twisted $\tau_F(M_N)$ still belongs to $M(G_N)$, but $\tau_F(M_N)\oplus M(G_S)$ produces paths $a\to d$ and $b \to c$, as shown in the figure. This contradicts the Kuo's condensation lemma, as a superimposition with paths $a\to d$ and $b \to c$ should correspond to a pair in $M(G_W)\times M(G_E)$ instead of $M(G_N)\times M(G_S)$. Hence this will not happen.
  
  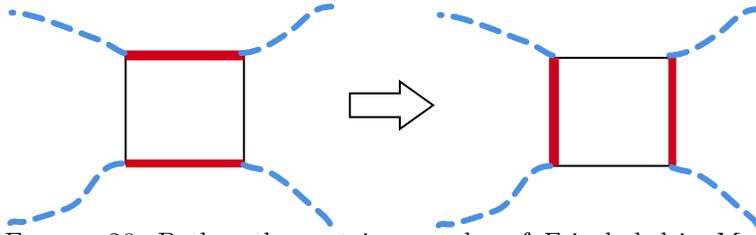
\begin{figure}[H]
  	\centering

  	\tikzset{every picture/.style={line width=0.75pt}} 
  	
  	\begin{tikzpicture}[x=0.75pt,y=0.75pt,yscale=-.6,xscale=.6]
  		
  		\draw   (114.54,98.09) -- (214,98.09) -- (214,189) -- (114.54,189) -- cycle ;
  		\draw [color={rgb, 255:red, 208; green, 2; blue, 27 }  ,draw opacity=1 ][line width=3.75]    (114.54,98.09) -- (214,98.09) ;
  		\draw [color={rgb, 255:red, 208; green, 2; blue, 27 }  ,draw opacity=1 ][line width=3]    (114.54,189) -- (214,189) ;
  		\draw [color={rgb, 255:red, 74; green, 144; blue, 226 }  ,draw opacity=1 ][line width=2.25] [line join = round][line cap = round] [dash pattern={on 6.75pt off 4.5pt}]  (113.54,96.09) .. controls (110.45,96.09) and (97.99,86.91) .. (92.54,85.09) .. controls (71.68,78.14) and (38.07,62.09) .. (18.54,62.09) ;
  		\draw [color={rgb, 255:red, 74; green, 144; blue, 226 }  ,draw opacity=1 ][line width=2.25] [line join = round][line cap = round] [dash pattern={on 6.75pt off 4.5pt}]  (213.54,190.09) .. controls (219.8,193.22) and (227.75,192.92) .. (234.54,196.09) .. controls (248.83,202.76) and (259.81,215.71) .. (271.54,225.09) .. controls (276.76,229.27) and (283.65,234.3) .. (286.54,240.09) ;
  		\draw [color={rgb, 255:red, 74; green, 144; blue, 226 }  ,draw opacity=1 ][line width=2.25] [line join = round][line cap = round] [dash pattern={on 6.75pt off 4.5pt}]  (111.54,189.09) .. controls (87.46,189.09) and (85.82,219.49) .. (64.54,227.09) .. controls (52.73,231.31) and (40.75,235.18) .. (28.54,238.09) .. controls (24.96,238.95) and (20.15,236.49) .. (17.54,239.09) ;
  		\draw [color={rgb, 255:red, 74; green, 144; blue, 226 }  ,draw opacity=1 ][line width=2.25] [line join = round][line cap = round] [dash pattern={on 6.75pt off 4.5pt}]  (211.54,96.09) .. controls (226.17,96.09) and (242.48,85.39) .. (253.54,77.09) .. controls (266.57,67.32) and (270.96,57.09) .. (287.54,57.09) ;
  		\draw   (474.54,100.09) -- (574,100.09) -- (574,191) -- (474.54,191) -- cycle ;
  		\draw [color={rgb, 255:red, 208; green, 2; blue, 27 }  ,draw opacity=1 ][line width=3.75]    (474.54,100.09) -- (474.54,191) ;
  		\draw [color={rgb, 255:red, 208; green, 2; blue, 27 }  ,draw opacity=1 ][line width=3]    (574,100.09) -- (574,191) ;
  		\draw [color={rgb, 255:red, 74; green, 144; blue, 226 }  ,draw opacity=1 ][line width=2.25] [line join = round][line cap = round] [dash pattern={on 6.75pt off 4.5pt}]  (473.54,98.09) .. controls (470.45,98.09) and (457.99,88.91) .. (452.54,87.09) .. controls (431.68,80.14) and (398.07,64.09) .. (378.54,64.09) ;
  		\draw [color={rgb, 255:red, 74; green, 144; blue, 226 }  ,draw opacity=1 ][line width=2.25] [line join = round][line cap = round] [dash pattern={on 6.75pt off 4.5pt}]  (573.54,192.09) .. controls (579.8,195.22) and (587.75,194.92) .. (594.54,198.09) .. controls (608.83,204.76) and (619.81,217.71) .. (631.54,227.09) .. controls (636.76,231.27) and (643.65,236.3) .. (646.54,242.09) ;
  		\draw [color={rgb, 255:red, 74; green, 144; blue, 226 }  ,draw opacity=1 ][line width=2.25] [line join = round][line cap = round] [dash pattern={on 6.75pt off 4.5pt}]  (471.54,191.09) .. controls (447.46,191.09) and (445.82,221.49) .. (424.54,229.09) .. controls (412.73,233.31) and (400.75,237.18) .. (388.54,240.09) .. controls (384.96,240.95) and (380.15,238.49) .. (377.54,241.09) ;
  		\draw [color={rgb, 255:red, 74; green, 144; blue, 226 }  ,draw opacity=1 ][line width=2.25] [line join = round][line cap = round] [dash pattern={on 6.75pt off 4.5pt}]  (571.54,98.09) .. controls (586.17,98.09) and (602.48,87.39) .. (613.54,79.09) .. controls (626.57,69.32) and (630.96,59.09) .. (647.54,59.09) ;
  		\draw   (303,130) -- (345,130) -- (345,120) -- (373,140) -- (345,160) -- (345,150) -- (303,150) -- cycle ;

  	\end{tikzpicture}
  	
  	\caption{Both paths contain one edge of $F$ included in $M_N$, which is shown to be   impossible.}
  \end{figure}
  
  Then suppose $F$ is a rectangle.	If both paths $a\to b$ and $c\to d$ contain one (or more edges) of $F$ lying in $M_N$, then we can list all the possibilities and show that, in all these cases, while $M_N\oplus M_S$, being an element of $M(G_N)\times M(G_S)$, produces paths $a\to b$ and $c\to d$, the twisted  $\tau_F(M_N)\oplus M(G_S)$ produces paths $a\to d$ and $b \to c$, which leads to a similar contradiction as above when $F$ is a square.
  
  The following figures show the forbidden possibilities. Using reflections when necessary, we extend to all the possibilities for the perfect matching of $F$ on the left hand side.

  \begin{figure}[H]
  	\centering

  	\tikzset{every picture/.style={line width=0.75pt}} 
  	
  	\begin{tikzpicture}[x=0.75pt,y=0.75pt,yscale=-.4,xscale=.5]
  		
  		\draw   (80.23,99.37) -- (202.49,99.37) -- (202.49,168.81) -- (80.23,168.81) -- cycle ;
  		\draw  [color={rgb, 255:red, 74; green, 144; blue, 226 }  ,draw opacity=1 ][dash pattern={on 6.75pt off 4.5pt}][line width=2.25] [line join = round][line cap = round] (78.4,97.1) .. controls (69.43,87.65) and (56.71,81.97) .. (46.25,74.25) .. controls (44.5,72.96) and (33.54,68.82) .. (33.54,67.95) ;
  		\draw  [color={rgb, 255:red, 74; green, 144; blue, 226 }  ,draw opacity=1 ][dash pattern={on 6.75pt off 4.5pt}][line width=2.25] [line join = round][line cap = round] (78.4,168.02) .. controls (67.67,171.79) and (55.27,183.66) .. (47.75,190.88) .. controls (46.71,191.87) and (34.29,199.41) .. (34.29,201.12) ;
  		\draw  [color={rgb, 255:red, 74; green, 144; blue, 226 }  ,draw opacity=1 ][dash pattern={on 6.75pt off 4.5pt}][line width=2.25] [line join = round][line cap = round] (140.44,97.1) .. controls (140.44,91.05) and (148.06,85.13) .. (151.65,81.34) .. controls (156.26,76.48) and (159.38,70.83) .. (163.61,66.37) .. controls (170.91,58.68) and (181.71,52.02) .. (188.28,45.09) ;
  		\draw  [color={rgb, 255:red, 74; green, 144; blue, 226 }  ,draw opacity=1 ][dash pattern={on 6.75pt off 4.5pt}][line width=2.25] [line join = round][line cap = round] (143.43,167.24) .. controls (146.4,171.93) and (156.98,190.7) .. (160.62,191.66) .. controls (168.73,193.8) and (176.05,196.3) .. (185.29,195.6) .. controls (186.58,195.51) and (191.03,189.56) .. (192.02,188.51) .. controls (196.54,183.75) and (202.49,175.87) .. (202.49,168.02) ;
  		\draw [color={rgb, 255:red, 208; green, 2; blue, 27 }  ,draw opacity=1 ][line width=3]    (80.23,99.37) -- (141.94,99.47) ;
  		\draw [color={rgb, 255:red, 208; green, 2; blue, 27 }  ,draw opacity=1 ][line width=3.75]    (80.23,168.81) -- (142.68,168.02) ;
  		\draw [color={rgb, 255:red, 208; green, 2; blue, 27 }  ,draw opacity=1 ][line width=3.75]    (202.49,97.18) -- (202.49,168.81) ;
  		\draw   (296,123) -- (338,123) -- (338,113) -- (366,133) -- (338,153) -- (338,143) -- (296,143) -- cycle ;
  		\draw   (434.23,99.37) -- (556.49,99.37) -- (556.49,168.81) -- (434.23,168.81) -- cycle ;
  		\draw  [color={rgb, 255:red, 74; green, 144; blue, 226 }  ,draw opacity=1 ][dash pattern={on 6.75pt off 4.5pt}][line width=2.25] [line join = round][line cap = round] (432.4,97.1) .. controls (423.43,87.65) and (410.71,81.97) .. (400.25,74.25) .. controls (398.5,72.96) and (387.54,68.82) .. (387.54,67.95) ;
  		\draw  [color={rgb, 255:red, 74; green, 144; blue, 226 }  ,draw opacity=1 ][dash pattern={on 6.75pt off 4.5pt}][line width=2.25] [line join = round][line cap = round] (432.4,168.02) .. controls (421.67,171.79) and (409.27,183.66) .. (401.75,190.88) .. controls (400.71,191.87) and (388.29,199.41) .. (388.29,201.12) ;
  		\draw  [color={rgb, 255:red, 74; green, 144; blue, 226 }  ,draw opacity=1 ][dash pattern={on 6.75pt off 4.5pt}][line width=2.25] [line join = round][line cap = round] (494.44,97.1) .. controls (494.44,91.05) and (502.06,85.13) .. (505.65,81.34) .. controls (510.26,76.48) and (513.38,70.83) .. (517.61,66.37) .. controls (524.91,58.68) and (535.71,52.02) .. (542.28,45.09) ;
  		\draw  [color={rgb, 255:red, 74; green, 144; blue, 226 }  ,draw opacity=1 ][dash pattern={on 6.75pt off 4.5pt}][line width=2.25] [line join = round][line cap = round] (497.43,167.24) .. controls (500.4,171.93) and (510.98,190.7) .. (514.62,191.66) .. controls (522.73,193.8) and (530.05,196.3) .. (539.29,195.6) .. controls (540.58,195.51) and (545.03,189.56) .. (546.02,188.51) .. controls (550.54,183.75) and (556.49,175.87) .. (556.49,168.02) ;
  		\draw [color={rgb, 255:red, 208; green, 2; blue, 27 }  ,draw opacity=1 ][line width=3]    (434.23,99.37) -- (434.23,168.81) ;
  		\draw [color={rgb, 255:red, 208; green, 2; blue, 27 }  ,draw opacity=1 ][line width=3.75]    (556.49,168.81) -- (496.68,168.02) ;
  		\draw [color={rgb, 255:red, 208; green, 2; blue, 27 }  ,draw opacity=1 ][line width=3.75]    (556.49,97.18) -- (495.54,98.09) ;
  		\draw  [color={rgb, 255:red, 74; green, 144; blue, 226 }  ,draw opacity=1 ][dash pattern={on 6.75pt off 4.5pt}][line width=2.25] [line join = round][line cap = round] (200.54,96.85) .. controls (205.98,94.13) and (218.17,86.46) .. (224.54,88.85) .. controls (230.96,91.26) and (238.26,115.74) .. (239.54,120.85) .. controls (243.76,137.73) and (247.02,156.18) .. (249.54,173.85) .. controls (250.42,179.99) and (251.18,194.67) .. (257.54,197.85) .. controls (263.27,200.71) and (267.72,201.17) .. (274.54,201.85) ;
  		\draw  [color={rgb, 255:red, 74; green, 144; blue, 226 }  ,draw opacity=1 ][dash pattern={on 6.75pt off 4.5pt}][line width=2.25] [line join = round][line cap = round] (556.54,96.85) .. controls (561.98,94.13) and (574.17,86.46) .. (580.54,88.85) .. controls (586.96,91.26) and (594.26,115.74) .. (595.54,120.85) .. controls (599.76,137.73) and (603.02,156.18) .. (605.54,173.85) .. controls (606.42,179.99) and (607.18,194.67) .. (613.54,197.85) .. controls (619.27,200.71) and (623.72,201.17) .. (630.54,201.85) ;
  		\draw   (80.23,302.12) -- (202.49,302.12) -- (202.49,371.56) -- (80.23,371.56) -- cycle ;
  		\draw  [color={rgb, 255:red, 74; green, 144; blue, 226 }  ,draw opacity=1 ][dash pattern={on 6.75pt off 4.5pt}][line width=2.25] [line join = round][line cap = round] (78.4,299.85) .. controls (69.43,290.41) and (56.71,284.72) .. (46.25,277) .. controls (44.5,275.71) and (33.54,271.57) .. (33.54,270.7) ;
  		\draw  [color={rgb, 255:red, 74; green, 144; blue, 226 }  ,draw opacity=1 ][dash pattern={on 6.75pt off 4.5pt}][line width=2.25] [line join = round][line cap = round] (78.4,370.78) .. controls (67.67,374.55) and (55.27,386.41) .. (47.75,393.63) .. controls (46.71,394.62) and (34.29,402.16) .. (34.29,403.87) ;
  		\draw  [color={rgb, 255:red, 74; green, 144; blue, 226 }  ,draw opacity=1 ][dash pattern={on 6.75pt off 4.5pt}][line width=2.25] [line join = round][line cap = round] (140.44,299.85) .. controls (140.44,293.8) and (148.06,287.88) .. (151.65,284.09) .. controls (156.26,279.24) and (159.38,273.58) .. (163.61,269.12) .. controls (170.91,261.43) and (181.71,254.78) .. (188.28,247.85) ;
  		\draw [color={rgb, 255:red, 208; green, 2; blue, 27 }  ,draw opacity=1 ][line width=3]    (80.23,302.12) -- (141.94,302.22) ;
  		\draw [color={rgb, 255:red, 208; green, 2; blue, 27 }  ,draw opacity=1 ][line width=3.75]    (80.23,371.56) -- (142.68,370.78) ;
  		\draw [color={rgb, 255:red, 208; green, 2; blue, 27 }  ,draw opacity=1 ][line width=3.75]    (202.49,299.93) -- (202.49,371.56) ;
  		\draw  [color={rgb, 255:red, 74; green, 144; blue, 226 }  ,draw opacity=1 ][dash pattern={on 6.75pt off 4.5pt}][line width=2.25] [line join = round][line cap = round] (140.54,373.85) .. controls (140.54,393.02) and (154.81,405.64) .. (169.54,414.85) .. controls (174.71,418.08) and (181.68,425.85) .. (187.54,425.85) ;
  		\draw [color={rgb, 255:red, 74; green, 144; blue, 226 }  ,draw opacity=1 ][line width=2.25] [line join = round][line cap = round] [dash pattern={on 6.75pt off 4.5pt}]  (202.54,298.85) .. controls (205.94,295.45) and (209.53,293.63) .. (213.54,291.85) .. controls (219.1,289.38) and (223.82,285.7) .. (229.54,286.85) .. controls (239.91,288.92) and (243.12,314.58) .. (242.54,323.85) .. controls (241.76,336.31) and (230.89,352.15) .. (225.54,362.85) .. controls (221.36,371.22) and (202.54,390.42) .. (202.54,368.85) ;
  		\draw   (435.23,296.12) -- (557.49,296.12) -- (557.49,365.56) -- (435.23,365.56) -- cycle ;
  		\draw  [color={rgb, 255:red, 74; green, 144; blue, 226 }  ,draw opacity=1 ][dash pattern={on 6.75pt off 4.5pt}][line width=2.25] [line join = round][line cap = round] (433.4,293.85) .. controls (424.43,284.41) and (411.71,278.72) .. (401.25,271) .. controls (399.5,269.71) and (388.54,265.57) .. (388.54,264.7) ;
  		\draw  [color={rgb, 255:red, 74; green, 144; blue, 226 }  ,draw opacity=1 ][dash pattern={on 6.75pt off 4.5pt}][line width=2.25] [line join = round][line cap = round] (433.4,364.78) .. controls (422.67,368.55) and (410.27,380.41) .. (402.75,387.63) .. controls (401.71,388.62) and (389.29,396.16) .. (389.29,397.87) ;
  		\draw  [color={rgb, 255:red, 74; green, 144; blue, 226 }  ,draw opacity=1 ][dash pattern={on 6.75pt off 4.5pt}][line width=2.25] [line join = round][line cap = round] (495.44,293.85) .. controls (495.44,287.8) and (503.06,281.88) .. (506.65,278.09) .. controls (511.26,273.24) and (514.38,267.58) .. (518.61,263.12) .. controls (525.91,255.43) and (536.71,248.78) .. (543.28,241.85) ;
  		\draw [color={rgb, 255:red, 208; green, 2; blue, 27 }  ,draw opacity=1 ][line width=3]    (557.49,293.93) -- (496.94,296.22) ;
  		\draw [color={rgb, 255:red, 208; green, 2; blue, 27 }  ,draw opacity=1 ][line width=3.75]    (435.23,365.56) -- (435.23,296.12) ;
  		\draw [color={rgb, 255:red, 208; green, 2; blue, 27 }  ,draw opacity=1 ][line width=3.75]    (497.68,364.78) -- (557.49,365.56) ;
  		\draw  [color={rgb, 255:red, 74; green, 144; blue, 226 }  ,draw opacity=1 ][dash pattern={on 6.75pt off 4.5pt}][line width=2.25] [line join = round][line cap = round] (495.54,367.85) .. controls (495.54,387.02) and (509.81,399.64) .. (524.54,408.85) .. controls (529.71,412.08) and (536.68,419.85) .. (542.54,419.85) ;
  		\draw [color={rgb, 255:red, 74; green, 144; blue, 226 }  ,draw opacity=1 ][line width=2.25] [line join = round][line cap = round] [dash pattern={on 6.75pt off 4.5pt}]  (557.54,292.85) .. controls (560.94,289.45) and (564.53,287.63) .. (568.54,285.85) .. controls (574.1,283.38) and (578.82,279.7) .. (584.54,280.85) .. controls (594.91,282.92) and (598.12,308.58) .. (597.54,317.85) .. controls (596.76,330.31) and (585.89,346.15) .. (580.54,356.85) .. controls (576.36,365.22) and (557.54,384.42) .. (557.54,362.85) ;
  		\draw   (299,324.75) -- (341,324.75) -- (341,314.75) -- (369,334.75) -- (341,354.75) -- (341,344.75) -- (299,344.75) -- cycle ;
  		\draw   (78.23,524.15) -- (200.49,524.15) -- (200.49,593.59) -- (78.23,593.59) -- cycle ;
  		\draw  [color={rgb, 255:red, 74; green, 144; blue, 226 }  ,draw opacity=1 ][dash pattern={on 6.75pt off 4.5pt}][line width=2.25] [line join = round][line cap = round] (76.4,521.88) .. controls (67.43,512.43) and (54.71,506.75) .. (44.25,499.03) .. controls (42.5,497.73) and (31.54,493.6) .. (31.54,492.73) ;
  		\draw  [color={rgb, 255:red, 74; green, 144; blue, 226 }  ,draw opacity=1 ][dash pattern={on 6.75pt off 4.5pt}][line width=2.25] [line join = round][line cap = round] (138.44,521.88) .. controls (138.44,515.83) and (146.06,509.91) .. (149.65,506.12) .. controls (154.26,501.26) and (157.38,495.61) .. (161.61,491.15) .. controls (168.91,483.46) and (179.71,476.8) .. (186.28,469.87) ;
  		\draw [color={rgb, 255:red, 208; green, 2; blue, 27 }  ,draw opacity=1 ][line width=3]    (78.23,524.15) -- (139.94,524.25) ;
  		\draw [color={rgb, 255:red, 208; green, 2; blue, 27 }  ,draw opacity=1 ][line width=3.75]    (78.23,593.59) -- (140.68,592.8) ;
  		\draw [color={rgb, 255:red, 208; green, 2; blue, 27 }  ,draw opacity=1 ][line width=3.75]    (200.49,521.96) -- (200.49,593.59) ;
  		\draw   (295,542.78) -- (337,542.78) -- (337,532.78) -- (365,552.78) -- (337,572.78) -- (337,562.78) -- (295,562.78) -- cycle ;
  		\draw  [color={rgb, 255:red, 74; green, 144; blue, 226 }  ,draw opacity=1 ][dash pattern={on 6.75pt off 4.5pt}][line width=2.25] [line join = round][line cap = round] (202.54,521.91) .. controls (207.98,519.18) and (220.17,511.52) .. (226.54,513.91) .. controls (232.96,516.31) and (240.26,540.79) .. (241.54,545.91) .. controls (245.76,562.78) and (249.02,581.24) .. (251.54,598.91) .. controls (252.42,605.04) and (253.18,619.73) .. (259.54,622.91) .. controls (265.27,625.77) and (269.72,626.22) .. (276.54,626.91) ;
  		\draw  [color={rgb, 255:red, 74; green, 144; blue, 226 }  ,draw opacity=1 ][dash pattern={on 6.75pt off 4.5pt}][line width=2.25] [line join = round][line cap = round] (78.54,593.07) .. controls (78.54,613.41) and (84.95,617.6) .. (100.54,621.07) .. controls (122.56,625.96) and (138.54,616.85) .. (138.54,594.07) ;
  		\draw  [color={rgb, 255:red, 74; green, 144; blue, 226 }  ,draw opacity=1 ][dash pattern={on 6.75pt off 4.5pt}][line width=2.25] [line join = round][line cap = round] (198.54,593.07) .. controls (191.36,602.04) and (184.64,630.02) .. (174.54,635.07) .. controls (170.58,637.05) and (166.95,640.92) .. (162.54,644.07) .. controls (151.1,652.24) and (134.95,660.37) .. (120.54,663.07) .. controls (103.47,666.27) and (80.65,666.73) .. (63.54,666.07) .. controls (49.56,665.53) and (33.45,659.07) .. (18.54,659.07) ;
  		\draw   (427.23,521.95) -- (549.49,521.95) -- (549.49,591.39) -- (427.23,591.39) -- cycle ;
  		\draw  [color={rgb, 255:red, 74; green, 144; blue, 226 }  ,draw opacity=1 ][dash pattern={on 6.75pt off 4.5pt}][line width=2.25] [line join = round][line cap = round] (425.4,519.68) .. controls (416.43,510.23) and (403.71,504.54) .. (393.25,496.82) .. controls (391.5,495.53) and (380.54,491.39) .. (380.54,490.52) ;
  		\draw  [color={rgb, 255:red, 74; green, 144; blue, 226 }  ,draw opacity=1 ][dash pattern={on 6.75pt off 4.5pt}][line width=2.25] [line join = round][line cap = round] (487.44,519.68) .. controls (487.44,513.62) and (495.06,507.7) .. (498.65,503.92) .. controls (503.26,499.06) and (506.38,493.41) .. (510.61,488.94) .. controls (517.91,481.26) and (528.71,474.6) .. (535.28,467.67) ;
  		\draw [color={rgb, 255:red, 208; green, 2; blue, 27 }  ,draw opacity=1 ][line width=3]    (427.23,521.95) -- (427.23,591.39) ;
  		\draw [color={rgb, 255:red, 208; green, 2; blue, 27 }  ,draw opacity=1 ][line width=3.75]    (549.49,591.39) -- (489.68,590.6) ;
  		\draw [color={rgb, 255:red, 208; green, 2; blue, 27 }  ,draw opacity=1 ][line width=3.75]    (549.49,521.95) -- (488.54,523.01) ;
  		\draw  [color={rgb, 255:red, 74; green, 144; blue, 226 }  ,draw opacity=1 ][dash pattern={on 6.75pt off 4.5pt}][line width=2.25] [line join = round][line cap = round] (551.54,519.7) .. controls (556.98,516.98) and (569.17,509.31) .. (575.54,511.7) .. controls (581.96,514.11) and (589.26,538.59) .. (590.54,543.7) .. controls (594.76,560.58) and (598.02,579.03) .. (600.54,596.7) .. controls (601.42,602.84) and (602.18,617.52) .. (608.54,620.7) .. controls (614.27,623.56) and (618.72,624.02) .. (625.54,624.7) ;
  		\draw  [color={rgb, 255:red, 74; green, 144; blue, 226 }  ,draw opacity=1 ][dash pattern={on 6.75pt off 4.5pt}][line width=2.25] [line join = round][line cap = round] (427.54,590.86) .. controls (427.54,611.21) and (433.95,615.4) .. (449.54,618.86) .. controls (471.56,623.76) and (487.54,614.64) .. (487.54,591.86) ;
  		\draw  [color={rgb, 255:red, 74; green, 144; blue, 226 }  ,draw opacity=1 ][dash pattern={on 6.75pt off 4.5pt}][line width=2.25] [line join = round][line cap = round] (547.54,590.86) .. controls (540.36,599.83) and (533.64,627.81) .. (523.54,632.86) .. controls (519.58,634.85) and (515.95,638.72) .. (511.54,641.86) .. controls (500.1,650.03) and (483.95,658.16) .. (469.54,660.86) .. controls (452.47,664.06) and (429.65,664.52) .. (412.54,663.86) .. controls (398.56,663.32) and (382.45,656.86) .. (367.54,656.86) ;
  		\draw   (79.23,765.64) -- (201.49,765.64) -- (201.49,835.08) -- (79.23,835.08) -- cycle ;
  		\draw  [color={rgb, 255:red, 74; green, 144; blue, 226 }  ,draw opacity=1 ][dash pattern={on 6.75pt off 4.5pt}][line width=2.25] [line join = round][line cap = round] (206.44,762.37) .. controls (206.44,756.31) and (214.06,750.4) .. (217.65,746.61) .. controls (222.26,741.75) and (225.38,736.1) .. (229.61,731.64) .. controls (236.91,723.95) and (247.71,717.29) .. (254.28,710.36) ;
  		\draw [color={rgb, 255:red, 208; green, 2; blue, 27 }  ,draw opacity=1 ][line width=3]    (79.23,765.64) -- (140.94,765.73) ;
  		\draw [color={rgb, 255:red, 208; green, 2; blue, 27 }  ,draw opacity=1 ][line width=3.75]    (79.23,835.08) -- (141.68,834.29) ;
  		\draw [color={rgb, 255:red, 208; green, 2; blue, 27 }  ,draw opacity=1 ][line width=3.75]    (201.49,763.45) -- (201.49,835.08) ;
  		\draw  [color={rgb, 255:red, 74; green, 144; blue, 226 }  ,draw opacity=1 ][dash pattern={on 6.75pt off 4.5pt}][line width=2.25] [line join = round][line cap = round] (200.54,836.36) .. controls (200.54,855.54) and (214.81,868.15) .. (229.54,877.36) .. controls (234.71,880.59) and (241.68,888.36) .. (247.54,888.36) ;
  		\draw   (296,788.27) -- (338,788.27) -- (338,778.27) -- (366,798.27) -- (338,818.27) -- (338,808.27) -- (296,808.27) -- cycle ;
  		\draw  [color={rgb, 255:red, 74; green, 144; blue, 226 }  ,draw opacity=1 ][dash pattern={on 6.75pt off 4.5pt}][line width=2.25] [line join = round][line cap = round] (78.54,763.09) .. controls (57.67,763.09) and (53.87,795.03) .. (56.54,811.09) .. controls (56.75,812.36) and (60.57,818.14) .. (61.54,820.09) .. controls (63.62,824.25) and (69.5,831.06) .. (72.54,834.09) .. controls (73.35,834.9) and (77.54,836.29) .. (77.54,837.09) ;
  		\draw  [color={rgb, 255:red, 74; green, 144; blue, 226 }  ,draw opacity=1 ][dash pattern={on 6.75pt off 4.5pt}][line width=2.25] [line join = round][line cap = round] (137.54,765.09) .. controls (108.25,735.8) and (84.55,729.09) .. (41.54,729.09) ;
  		\draw  [color={rgb, 255:red, 74; green, 144; blue, 226 }  ,draw opacity=1 ][dash pattern={on 6.75pt off 4.5pt}][line width=2.25] [line join = round][line cap = round] (138.54,833.09) .. controls (124.04,866.93) and (84.11,869.93) .. (49.54,872.09) .. controls (43.5,872.47) and (37.29,875.09) .. (31.54,875.09) ;
  		\draw   (427.23,771.72) -- (549.49,771.72) -- (549.49,841.16) -- (427.23,841.16) -- cycle ;
  		\draw  [color={rgb, 255:red, 74; green, 144; blue, 226 }  ,draw opacity=1 ][dash pattern={on 6.75pt off 4.5pt}][line width=2.25] [line join = round][line cap = round] (554.44,768.45) .. controls (554.44,762.39) and (562.06,756.48) .. (565.65,752.69) .. controls (570.26,747.83) and (573.38,742.18) .. (577.61,737.72) .. controls (584.91,730.03) and (595.71,723.37) .. (602.28,716.44) ;
  		\draw [color={rgb, 255:red, 208; green, 2; blue, 27 }  ,draw opacity=1 ][line width=3]    (427.23,771.72) -- (427.23,841.16) ;
  		\draw [color={rgb, 255:red, 208; green, 2; blue, 27 }  ,draw opacity=1 ][line width=3.75]    (549.49,841.16) -- (489.68,840.37) ;
  		\draw [color={rgb, 255:red, 208; green, 2; blue, 27 }  ,draw opacity=1 ][line width=3.75]    (549.49,771.72) -- (487.54,772.12) ;
  		\draw  [color={rgb, 255:red, 74; green, 144; blue, 226 }  ,draw opacity=1 ][dash pattern={on 6.75pt off 4.5pt}][line width=2.25] [line join = round][line cap = round] (548.54,842.44) .. controls (548.54,861.62) and (562.81,874.23) .. (577.54,883.44) .. controls (582.71,886.67) and (589.68,894.44) .. (595.54,894.44) ;
  		\draw  [color={rgb, 255:red, 74; green, 144; blue, 226 }  ,draw opacity=1 ][dash pattern={on 6.75pt off 4.5pt}][line width=2.25] [line join = round][line cap = round] (426.54,769.17) .. controls (405.67,769.17) and (401.87,801.11) .. (404.54,817.17) .. controls (404.75,818.44) and (408.57,824.22) .. (409.54,826.17) .. controls (411.62,830.33) and (417.5,837.14) .. (420.54,840.17) .. controls (421.35,840.98) and (425.54,842.37) .. (425.54,843.17) ;
  		\draw  [color={rgb, 255:red, 74; green, 144; blue, 226 }  ,draw opacity=1 ][dash pattern={on 6.75pt off 4.5pt}][line width=2.25] [line join = round][line cap = round] (485.54,771.17) .. controls (456.25,741.88) and (432.55,735.17) .. (389.54,735.17) ;
  		\draw  [color={rgb, 255:red, 74; green, 144; blue, 226 }  ,draw opacity=1 ][dash pattern={on 6.75pt off 4.5pt}][line width=2.25] [line join = round][line cap = round] (486.54,839.17) .. controls (472.04,873.01) and (432.11,876.01) .. (397.54,878.17) .. controls (391.5,878.55) and (385.29,881.17) .. (379.54,881.17) ;

  	\end{tikzpicture}
  	
  	\caption{Several impossibilities when $F$ is a rectangle}
  \end{figure}
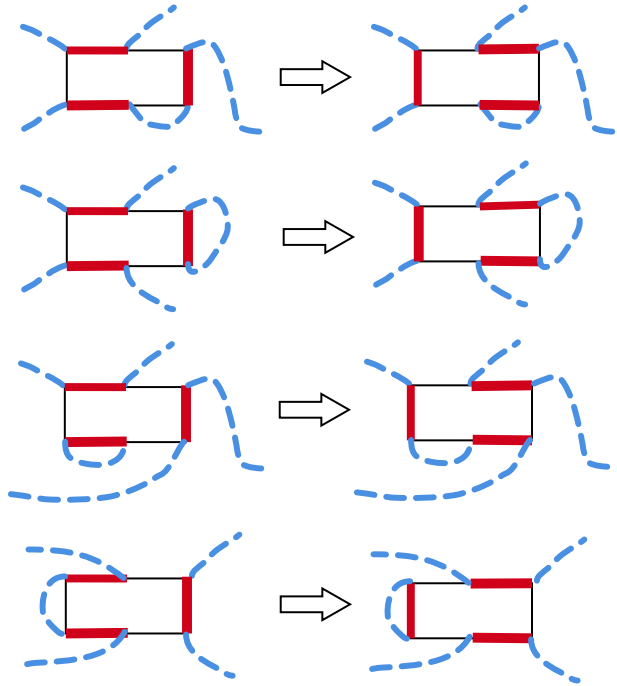

	Similar arguments hold for $W,S$ and $E$. The proof of lemma is thus complete.
\end{proof}
\begin{lem}\label{lem:induction-step}
	
	Let   $M_N\in M(G_N)$ and $M_S\in M(G_S)$.  Suppose $M_N$ is twistable at $F$. 
	
	\begin{itemize}
		\item We can choose a decomposition $M_N\oplus M_S\sim M_A\oplus M_C$ so that $M_A$ or $M_C$ is also twistable at $F$. 
		\item  Given the above decomposition, $\tau_F(M_N)\oplus M_S\sim \tau_F(M_A)\oplus M_C$ if $M_A$ is twistable at $F$, or $M_A \oplus \tau_F(M_C)$ if $M_C$ is twistable at $F$. 
		\item Assume without loss of generality that $\tau_F(M_N)=\tau_F^+(M_N)$, then in the above decomposition after twist, we have $\tau_F(M_A)=\tau_F^+(M_A)$ (or $\tau_F(M_C)=\tau_F^+(M_C)$ respectively).
	\end{itemize}
	
	 Similar statements hold if we switch to perfect matchings of $G_W$ and $G_E$, $\tau^+$ to $\tau^-$, or twisting $M_S$ instead of $M_N$, as long as the statements are well defined.
\end{lem}
\begin{rem}
	This lemma is the induction step for Lemma \ref{lem:superimposition-difference}. Indeed, by twisting at some face $F$, both sides either increases $1_F$, or decreases $1_F$ by Lemma \ref{lem:cover-relation}.
\end{rem}
\begin{proof}

\begin{enumerate}

\item 	For the first statement, recall that in the superimposition $M_N\oplus M_S$, there are two paths $a\to b$ and $c\to d$, and some closed cycles (including doubled edges).
	
	 First suppose $F$ being a square, with the two edges chosen in $M_N$ denoted by $e_1$ and $e_2$. 
	 If $e_1$ and $e_2$ are on the same connected component in $M_N\oplus M_S$, then the first statement holds since $e_1$ and $e_2$, both belonging to $M_N$, have the same parity on that component. Hence they also both belong to $M_A$ or $M_C$. 
	 Otherwise, consider the two paths $a\to b$ and $c \to d$. From the above lemma, there are two possibilities: No path, or exactly one path, contains some edge of $F$ lying in $M_N$.
	
	\begin{figure}[h]
		
		\centering

		\tikzset{every picture/.style={line width=0.75pt}} 
		
		\begin{tikzpicture}[x=0.75pt,y=0.75pt,yscale=-.5,xscale=.6]
			
			\draw   (102.54,101.09) -- (202,101.09) -- (202,192) -- (102.54,192) -- cycle ;
			\draw [color={rgb, 255:red, 208; green, 2; blue, 27 }  ,draw opacity=1 ][line width=3.75]    (102.54,101.09) -- (202,101.09) ;
			\draw [color={rgb, 255:red, 208; green, 2; blue, 27 }  ,draw opacity=1 ][line width=3]    (102.54,192) -- (202,192) ;
			\draw [color={rgb, 255:red, 74; green, 144; blue, 226 }  ,draw opacity=1 ][line width=3] [line join = round][line cap = round] [dash pattern={on 7.88pt off 4.5pt}]  (101.54,99.09) .. controls (97.23,94.78) and (94.34,86.29) .. (91.54,82.09) .. controls (90.23,80.12) and (87.63,76.46) .. (87.54,74.09) .. controls (87.17,64.08) and (82.46,51.18) .. (89.54,44.09) .. controls (103.32,30.32) and (154.55,28.59) .. (173.54,29.09) .. controls (177.49,29.2) and (181.64,30.9) .. (185.54,31.09) .. controls (200.9,31.86) and (212.68,30.23) .. (221.54,39.09) .. controls (231.94,49.49) and (224.69,72.88) .. (216.54,85.09) .. controls (215.54,86.6) and (214.93,90.7) .. (213.54,92.09) .. controls (211.8,93.84) and (202.54,99.63) .. (202.54,101.09) ;
			\draw [color={rgb, 255:red, 74; green, 144; blue, 226 }  ,draw opacity=1 ][line width=3] [line join = round][line cap = round] [dash pattern={on 7.88pt off 4.5pt}]  (100.54,191.09) .. controls (95.71,191.09) and (93.27,195.61) .. (89.54,198.09) .. controls (76.82,206.57) and (81.25,217.66) .. (89.54,230.09) .. controls (100.36,246.33) and (126.97,254.63) .. (146.54,255.09) .. controls (159.54,255.4) and (172.55,255.41) .. (185.54,255.09) .. controls (191.8,254.94) and (219.11,249.39) .. (220.54,245.09) .. controls (223.23,237.02) and (225.34,211.89) .. (218.54,205.09) .. controls (217.55,204.1) and (214.15,203.3) .. (213.54,202.09) .. controls (211.84,198.68) and (206.24,191.09) .. (201.54,191.09) ;
			\draw   (427.54,98.09) -- (527,98.09) -- (527,189) -- (427.54,189) -- cycle ;
			\draw [color={rgb, 255:red, 208; green, 2; blue, 27 }  ,draw opacity=1 ][line width=3.75]    (427.54,98.09) -- (427.54,189) ;
			\draw [color={rgb, 255:red, 208; green, 2; blue, 27 }  ,draw opacity=1 ][line width=3]    (527.54,98.09) -- (527,189) ;
			\draw [color={rgb, 255:red, 74; green, 144; blue, 226 }  ,draw opacity=1 ][line width=3] [line join = round][line cap = round] [dash pattern={on 7.88pt off 4.5pt}]  (426.54,96.09) .. controls (422.23,91.78) and (419.34,83.29) .. (416.54,79.09) .. controls (415.23,77.12) and (412.63,73.46) .. (412.54,71.09) .. controls (412.17,61.08) and (407.46,48.18) .. (414.54,41.09) .. controls (428.32,27.32) and (479.55,25.59) .. (498.54,26.09) .. controls (502.49,26.2) and (506.64,27.9) .. (510.54,28.09) .. controls (525.9,28.86) and (537.68,27.23) .. (546.54,36.09) .. controls (556.94,46.49) and (549.69,69.88) .. (541.54,82.09) .. controls (540.54,83.6) and (539.93,87.7) .. (538.54,89.09) .. controls (536.8,90.84) and (527.54,96.63) .. (527.54,98.09) ;
			\draw [color={rgb, 255:red, 74; green, 144; blue, 226 }  ,draw opacity=1 ][line width=3] [line join = round][line cap = round] [dash pattern={on 7.88pt off 4.5pt}]  (425.54,188.09) .. controls (420.71,188.09) and (418.27,192.61) .. (414.54,195.09) .. controls (401.82,203.57) and (406.25,214.66) .. (414.54,227.09) .. controls (425.36,243.33) and (451.97,251.63) .. (471.54,252.09) .. controls (484.54,252.4) and (497.55,252.41) .. (510.54,252.09) .. controls (516.8,251.94) and (544.11,246.39) .. (545.54,242.09) .. controls (548.23,234.02) and (550.34,208.89) .. (543.54,202.09) .. controls (542.55,201.1) and (539.15,200.3) .. (538.54,199.09) .. controls (536.84,195.68) and (531.24,188.09) .. (526.54,188.09) ;
			\draw   (272,134.77) -- (334.73,134.77) -- (334.73,116) -- (376.54,153.55) -- (334.73,191.09) -- (334.73,172.32) -- (272,172.32) -- cycle ;

		\end{tikzpicture}
		\caption{No path contains edge of $F$ lying in $M_N$. Therefore, both $e_1$ and $e_2$ are included in some cycle.}
	\end{figure}
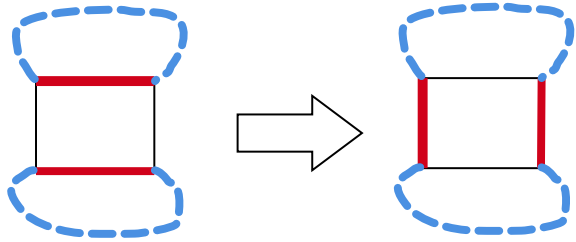
	
	In the case where both edges are included in some cycle, we have the freedom to declare the edge $e_1$ either in $M_A$ or in $M_C$, and the same for $e_2$. We thus choose both edges to belong in $M_A$, and the first statement holds for $M_A$.
	
	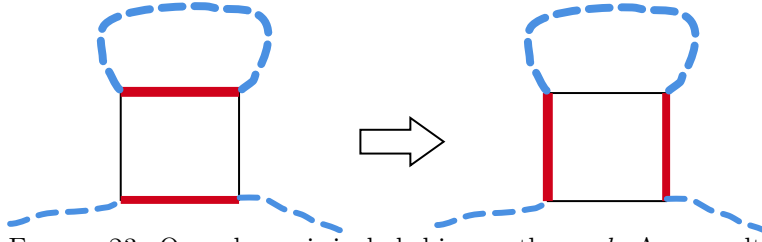
\begin{figure}[H]
		\centering

		\tikzset{every picture/.style={line width=0.75pt}} 
		
		\begin{tikzpicture}[x=0.75pt,y=0.75pt,yscale=-.6,xscale=.6]
			
			\draw   (111.54,99.09) -- (211,99.09) -- (211,190) -- (111.54,190) -- cycle ;
			\draw [color={rgb, 255:red, 208; green, 2; blue, 27 }  ,draw opacity=1 ][line width=3.75]    (111.54,99.09) -- (211,99.09) ;
			\draw [color={rgb, 255:red, 208; green, 2; blue, 27 }  ,draw opacity=1 ][line width=3]    (111.54,190) -- (211,190) ;
			\draw [color={rgb, 255:red, 74; green, 144; blue, 226 }  ,draw opacity=1 ][line width=3] [line join = round][line cap = round] [dash pattern={on 7.88pt off 4.5pt}]  (110.54,97.09) .. controls (106.23,92.78) and (103.34,84.29) .. (100.54,80.09) .. controls (99.23,78.12) and (96.63,74.46) .. (96.54,72.09) .. controls (96.17,62.08) and (91.46,49.18) .. (98.54,42.09) .. controls (112.32,28.32) and (163.55,26.59) .. (182.54,27.09) .. controls (186.49,27.2) and (190.64,28.9) .. (194.54,29.09) .. controls (209.9,29.86) and (221.68,28.23) .. (230.54,37.09) .. controls (240.94,47.49) and (233.69,70.88) .. (225.54,83.09) .. controls (224.54,84.6) and (223.93,88.7) .. (222.54,90.09) .. controls (220.8,91.84) and (211.54,97.63) .. (211.54,99.09) ;
			\draw [color={rgb, 255:red, 74; green, 144; blue, 226 }  ,draw opacity=1 ][line width=2.25] [line join = round][line cap = round] [dash pattern={on 6.75pt off 4.5pt}]  (18.54,210.12) .. controls (27.11,210.12) and (37.12,203.99) .. (45.54,202.12) .. controls (57.04,199.57) and (69.41,200.47) .. (81.54,199.12) .. controls (86.53,198.57) and (94.3,196.9) .. (100.54,196.12) .. controls (103.35,195.77) and (107.84,192.12) .. (111.54,192.12) ;
			\draw [color={rgb, 255:red, 74; green, 144; blue, 226 }  ,draw opacity=1 ][line width=2.25] [line join = round][line cap = round] [dash pattern={on 6.75pt off 4.5pt}]  (211.54,188.12) .. controls (216.88,188.12) and (222.22,187.81) .. (227.54,188.12) .. controls (235.54,188.59) and (242.38,195.26) .. (249.54,198.12) .. controls (254.38,200.06) and (260.71,200.7) .. (265.54,203.12) .. controls (268.64,204.67) and (274.22,208.54) .. (276.54,209.12) .. controls (282.98,210.73) and (289.03,212.37) .. (294.54,215.12) .. controls (295.21,215.45) and (296.54,215.38) .. (296.54,216.12) ;
			\draw   (470.54,100.09) -- (570,100.09) -- (570,191) -- (470.54,191) -- cycle ;
			\draw [color={rgb, 255:red, 208; green, 2; blue, 27 }  ,draw opacity=1 ][line width=3.75]    (470.54,100.09) -- (470.54,191) ;
			\draw [color={rgb, 255:red, 208; green, 2; blue, 27 }  ,draw opacity=1 ][line width=3]    (570,100.09) -- (570,191) ;
			\draw [color={rgb, 255:red, 74; green, 144; blue, 226 }  ,draw opacity=1 ][line width=3] [line join = round][line cap = round] [dash pattern={on 7.88pt off 4.5pt}]  (469.54,98.09) .. controls (465.23,93.78) and (462.34,85.29) .. (459.54,81.09) .. controls (458.23,79.12) and (455.63,75.46) .. (455.54,73.09) .. controls (455.17,63.08) and (450.46,50.18) .. (457.54,43.09) .. controls (471.32,29.32) and (522.55,27.59) .. (541.54,28.09) .. controls (545.49,28.2) and (549.64,29.9) .. (553.54,30.09) .. controls (568.9,30.86) and (580.68,29.23) .. (589.54,38.09) .. controls (599.94,48.49) and (592.69,71.88) .. (584.54,84.09) .. controls (583.54,85.6) and (582.93,89.7) .. (581.54,91.09) .. controls (579.8,92.84) and (570.54,98.63) .. (570.54,100.09) ;
			\draw [color={rgb, 255:red, 74; green, 144; blue, 226 }  ,draw opacity=1 ][line width=2.25] [line join = round][line cap = round] [dash pattern={on 6.75pt off 4.5pt}]  (377.54,211.12) .. controls (386.11,211.12) and (396.12,204.99) .. (404.54,203.12) .. controls (416.04,200.57) and (428.41,201.47) .. (440.54,200.12) .. controls (445.53,199.57) and (453.3,197.9) .. (459.54,197.12) .. controls (462.35,196.77) and (466.84,193.12) .. (470.54,193.12) ;
			\draw [color={rgb, 255:red, 74; green, 144; blue, 226 }  ,draw opacity=1 ][line width=2.25] [line join = round][line cap = round] [dash pattern={on 6.75pt off 4.5pt}]  (570.54,189.12) .. controls (575.88,189.12) and (581.22,188.81) .. (586.54,189.12) .. controls (594.54,189.59) and (601.38,196.26) .. (608.54,199.12) .. controls (613.38,201.06) and (619.71,201.7) .. (624.54,204.12) .. controls (627.64,205.67) and (633.22,209.54) .. (635.54,210.12) .. controls (641.98,211.73) and (648.03,213.37) .. (653.54,216.12) .. controls (654.21,216.45) and (655.54,216.38) .. (655.54,217.12) ;
			\draw   (313,131) -- (355,131) -- (355,121) -- (383,141) -- (355,161) -- (355,151) -- (313,151) -- cycle ;

		\end{tikzpicture}
		\caption{One edge $e_2$ is included in a path $a\to b$. As a result, the other is included in a cycle.}
	\end{figure}
	In the case where exactly one edge is included in a path, the parity (whether in $M_A$ or $M_C$) of edge $e_2$ is already determined by its position in the path $a\to b$. But we are still free to choose whether $e_1$ is in $M_A$ or $M_C$. Thus we choose the one that is the same as the parity of $e_2$. The first statement also holds.

	If $F$ is a rectangle, we also have the fact that no or exactly one path contains some edge of $F$ lying in $M_N$. One can see that the statements above for squares treating several edges being in one connected component, as well as those dealing with no or exactly one path containing some edge of $F$ lying in $M_N$, still hold for $F$ being rectangle. 
	Also, the statements can be adjusted easily for $W,S$ and $E$. The first statement is thus proven in general.

\item	The second statement follows easily from Kuo's condensation, since the only difference between the superimposition $\tau_F(M_N)\oplus M_S$ and $M_N\oplus M_S$ are the edges of $F$.

\item 	The third statement $\tau_F=\tau_F^+$ is a consequence of the relative position of $G_*$ inside $G$. Recall in Lemma \ref{lem:subpinecone-place} that $G_*$ is copied in $G$ at \begin{itemize}
		\item $(0,0)$ if $*=W$,
		\item $(0,2)$ if $*=E$ or $C$,
		\item $(1,1)$ if $*=N$,
		\item $(-1,1)$ if $*=S$.
	\end{itemize}
	Here the vector $(i,j)$ denotes the coordinate of the upper-left vertex of each subgraph $G_*$ as a copy in $G$. Note that all the vectors appearing above are \emph{even}, i.e., $i+j$ are even. Therefore, the parity of a face $F$ in $G_*$ is the same as it is in $G$. The coherence $\tau_F=\tau_F^+$ thus follows from Lemma \ref{lem:above-or-below}. 
\end{enumerate}
\end{proof}

\begin{lem}\label{lem:base-case}
	Lemma \ref{lem:superimposition-difference} holds when $M_N$ and $M_S$ (resp. $M_W$ and $M_E$) are \emph{minimal} matchings.
\end{lem}
 
\begin{proof}
	The minimal matching $M=M_-$ satisfies $v_{M_-}=0$.  The description of minimal matching of pinecones is given right after Definition \ref{Def:Height}. 
	
	We first treat the case where $M_W$ and $M_E$ are miminal matchings of $G_W$ and $G_E$ respectively. According to the description after Lemma \ref{lem:subpinecone-properties}, the superimposition $M_W\oplus M_E$ (copied and extended to the perfect matching of $G-*$) consists of: \emph{Doubled} edges $(i,j)-(i,j+1)$ with $i>0$ and $i+j$ even; \emph{Doubled} edges $(i,j)-(i,j+1)$ with $i<-1$ and $i+j$ odd; On the central strip: two paths $(0,1)-(0,0)-(-1,0)-(-1,1)$ and $(0,2k)-(0,2k+1)-(-1,2k+1)-(-1,2k)$ and doubled edges $(0,2t)-(0,2t+1)$ and $(-1,2t)-(-1,2t+1)$ where $0<t<k$. Then using Kuo's condensation, one can show that $M_A$ and $M_C$ are \emph{minimal} matchings of $G_A$ and $G_C$ respectively. It follows that $$	v_{M_W}+v_{M_E}=v_{M_A}+v_{M_C}=0\,.$$
	
	On the other hand, if $M_N$ and $M_S$ are miminal matchings of $G_N$ and $G_S$ respectively, then $M_N\oplus M_S$ consists of : \emph{Doubled} edges $(i,j)-(i,j+1)$ with $i>0$ and $i+j$ even; \emph{Doubled} edges $(i,j)-(i,j+1)$ with $i<-1$ and $i+j$ odd; On the central strip: vertical edges $(0,0)-(-1,0)$ and $(0,2k+1)-(-1,2k+1)$ are doubled; two paths $(0,1)-(0,2)-\cdots-(0,2k-1)-(0,2k)$ and   $(-1,1)-(-1,2)-\cdots-(-1,2k-1)-(-1,2k)$. It follows from Kuo's condensation that $M_A$ consists of: Edges $(i,j)-(i,j+1)$ with $i>0$ and $i+j$ even; edges $(i,j)-(i,j+1)$ with $i<-1$ and $i+j$ odd; vertical edges  $(0,0)-(-1,0)$  and $(0,2k+1)-(-1,2k+1)$ and  edges $(0,2t-1)-(0,2t)$ with $0<t\le k$, and $M_C$ is the minimal matching of $G_C$. This $M_A$ exactly satisfies $v_{M_A}=1_H$. Hence it follows that 	$$v_{M_N}+v_{M_S}+1_H=v_{M_A}+v_{M_C}=1_H\,.$$
\end{proof}
\begin{figure}[H]
	\centering
	\includegraphics[scale=0.25]{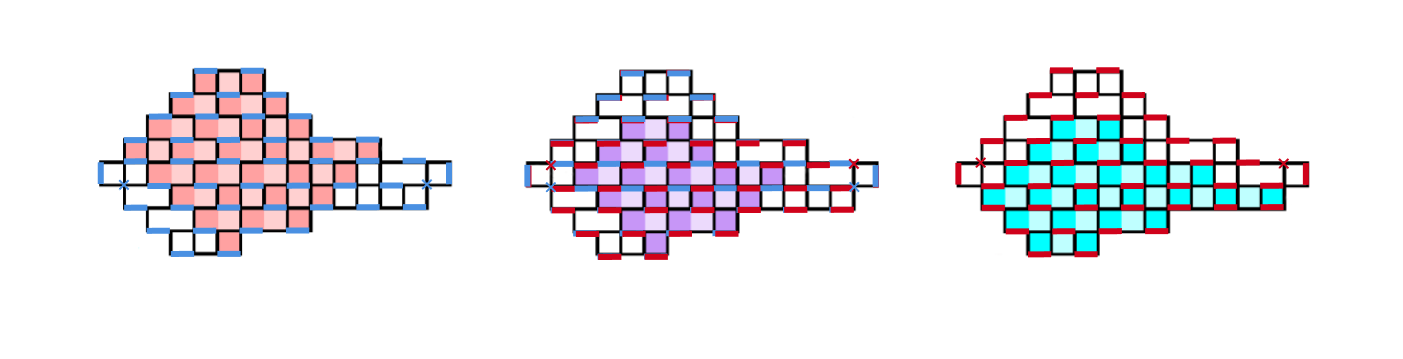}

	\caption{The minimal matching of $G_N$ and $G_S$, and their superimposition in the middle. Based on Figure 9 of \cite{ProppMelouWest2009}. On the two paths in the superimposition: The blue edges in the above path, and the red edges in the below path, belong to $M_A$ by Kuo's condensation. One calculates this $M_A$ satisfies $v_{M_A}=1_H$.}
\end{figure}
\begin{proof}[Proof of Lemma \ref{lem:superimposition-difference}]
	Apply Lemma \ref{lem:induction-step} to the above with $S$ fixed, we see Lemma \ref{lem:superimposition-difference} holds when $S$ is minimal and $N$ is an arbitrary perfect matching by Lemma \ref{lem:cover-relation} and \ref{lem:twist-properties}. Then we fix the $N$ thus got and twist $S$, we see by the same argument that Lemma \ref{lem:superimposition-difference} holds for arbitrary pairs of $(N,S)\in M(G_N)\times M(G_S)$. Similar arguments hold for $(W,E)\in M(G_W)\times M(G_E)$.
\end{proof}

\section{Conclusion and further remarks}
\label{sec:conclusions}

Using the above propositions, we  now conclude this paper with the proof of our main theorem.

\begin{proof}[Proof of Theorem \ref{Thm:GR-WH}]
    We have Proposition \ref{prop:GRprinc} which gives a recurrence of the sequence $x_n$. We also have Proposition \ref{Prop:superpositions} which gives a recurrence of the sequence $\widehat{x_n}$. These two recurrences are equal in form by applying Proposition \ref{Prop:cent-strip}.  Consequently, we have the equality $\widehat{x_n}=w(G_n^{r,s,N}) = x_n$ for every $n\in\mathbb{N}$ thus follows since the initial condition is  fulfilled, i.e., $\widehat{x_n}=x_n$ when $n\le N$, as in Remark \ref{rem:base-case-x}.
\end{proof}

We conjecture that formulas for
a large class of examples from the physics literature \cite{Eager2011,Eager2012,Davey2010,Franco2006Jan2,Hanany2012Jun} can
be generalized similarly as in our formula \ref{eq:GR-WH} where the usage of height functions is introduced. Moreover, in the theory of surface cluster algebras \cite{MUSIKER20112241positivity}, the expansion formula for the cluster variables with principal coefficients is also of the form of summing the weight times height functions of the perfect matchings of a certain graph (either snake or band graph). Given the ubiquitous appearance of perfect matching models in planar bipartite graphs, we wish to ask in what generality can we use the height function to calculate the principal coefficients for cluster variables.  To compute height functions of perfect matchings for other cluster algebras arising from brane tilings, it is necessary to identify the minimal matchings.  For example, in the case of the $dP_3$ quiver and brane tiling, this is obtained in recent work of Chiang-Musiker-Nguyen \cite{CMN}. 
\printbibliography

\end{document}